\documentclass[11pt]{amsart}
\usepackage[utf8]{inputenc}
\usepackage{amssymb,amsmath, amsthm, mathabx}
\usepackage{color, enumerate}
\usepackage[stable]{footmisc}
\usepackage{hyperref}
\usepackage{caption}
\usepackage{subcaption}
\usepackage{a4wide}
\usepackage{accents}
\usepackage{tikz}
\usepackage{graphicx}
\usepackage{chngcntr}
\usepackage{datetime2}
\usepackage{scalerel}    

\usepackage{xifthen}
\def \w {\omega}

\def \wb{\wh{\omega}}

\def \wne{\wc{\omega}}

\def \G{\mathrm G}

\def \Gb{\wh{\G}}

\def \Gne{\wc{\G}}

\def \B{\mathrm B}

\def \Bcdf{\mathrm F}

\def \bcdf{\mathrm f}

\def \b{\mathrm b}

\def \cif{\varphi}

\def \tree{\sT}

\def\Lb{\mathrm{L}}

\def\hor{\mathrm{hor}}
\def\ver{\mathrm{ver}}

\def \P{\mathbf P} 
\def\E{\bfE}

\def\Exp{\mathrm{Exp}}

\def\curv{\sigma}

\def\dd{\mathrm d}

\def\Eh{\mathrm{Z}}
\def\Ev{\mathrm{Z}}

\def\M{\mathrm M} 
\def\Shp{\gamma}

\def\Min{\zeta}

\newcommand{\bbR}{\mathbb{R}}

\newcommand{\bbZ}{\mathbb{Z}}

\newcommand{\bfE}{\mathbf{E}}

\newcommand{\mrR}{\mathrm{R}}
\newcommand{\mrS}{\mathrm{S}}

\newcommand{\sT}{\mathcal{T}}

\newcommand{\Fabs}[4]
{
\ifthenelse{\isempty{#2}}
{
{#1}_{#3}^{#4}
}
{
{#1}_{#3}^{#4}(#2)
}
}

\newcommand{\one}{\mathbf{1}}

\newcommand{\lc}{\lceil}
\newcommand{\rc}{\rceil}
\newcommand{\lf}{\lfloor}
\newcommand{\rf}{\rfloor}
\newcommand{\Var}{\mathbf{Var}}

\def\dd{\mathrm d}

\definecolor{darkblue}{rgb}{.1,.1,.6}

\newcommand{\wt}[1]{\widetilde{#1}}
\newcommand{\wh}[1]{\widehat{#1}}
\newcommand{\wc}[1]{\widecheck{#1}}

\newtheorem{thm}{Theorem}[section]
\newtheorem{prop}[thm]{Proposition}
\newtheorem{lem}[thm]{Lemma}
\newtheorem{cor}[thm]{Corollary}

\theoremstyle{remark}
\newtheorem{rem}{Remark}[section]
\newtheorem{example}[rem]{Example}

\newcommand{\be}{\begin{equation}}
\newcommand{\ee}{\end{equation}}

\usepackage[percent]{overpic}

\counterwithin{figure}{section}
\counterwithin{equation}{section}

\usepackage{todonotes, xargs}
\newcommandx{\note}[2][1=]{\todo[linecolor=yellow,backgroundcolor=yellow!25,bordercolor=yellow,#1]{#2}}

\allowdisplaybreaks[1] 

\title[Coupling approach to exit point bounds in LPP]
{Optimal-order exit point bounds in exponential last-passage percolation via the coupling technique}

\author[E.~Emrah]{Elnur Emrah}
\address{Elnur Emrah\\ University of Bristol \\ School of Mathematics \\ Bristol \\ United Kingdom}
\email{e.emrah@bristol.ac.uk}
\urladdr{https://sites.google.com/view/elnur-emrah}
\thanks{E.\ Emrah was partially supported by the grant KAW 2015.0270 from the Knut and Alice Wallenberg Foundation and by the Mathematical Sciences Department at Carnegie Mellon University through a postdoctoral position.} 

\author[C.~Janjigian]{Christopher Janjigian}
\address{Christopher Janjigian\\ Purdue University\\  Mathematics Department\\  150 N. University Street
\\   West Lafayette, IN 47907\\ USA.}
\email{cjanjigi@purdue.edu}
\urladdr{http://www.math.purdue.edu/~cjanjigi}
\thanks{C.\ Janjigian was partially supported by National Science Foundation grant DMS-1954204 and by a postdoctoral grant from the Fondation Sciences Math\'ematiques de Paris while working at Universit\'e Paris Diderot.}

\author[T.~Sepp\"al\"ainen]{Timo Sepp\"al\"ainen}
\address{Timo Sepp\"al\"ainen\\ University of Wisconsin-Madison\\  Mathematics Department\\ Van Vleck Hall\\ 480 Lincoln Dr.\\   Madison WI 53706-1388\\ USA.}
\email{seppalai@math.wisc.edu}
\urladdr{http://www.math.wisc.edu/~seppalai}
\thanks{T.\ Sepp\"al\"ainen was partially supported by National Science Foundation grants  DMS-1602486, DMS-1854619  and DMS-2152362, and by the Wisconsin Alumni Research Foundation.} 

\thanks{First published by ©2023 MSP (Mathematical Sciences Publishers) in Probability and Mathematical Physics in Vol.\ 4 (2023), No.\ 3.}

\begin{document}

\begin{abstract}
We develop a new probabilistic method for deriving deviation estimates in directed planar polymer and percolation models. The key estimates are for exit points of geodesics as they cross transversal down-right boundaries. These bounds are of optimal cubic-exponential order. We derive them in the context of last-passage percolation with exponential weights for a class of boundary conditions including the stationary case. 
%
%
As a result, the probabilistic coupling method is empowered to treat a variety of problems optimally, which could previously be achieved only via inputs from integrable probability. As applications in the bulk setting, we obtain upper bounds of cubic-exponential order for transversal fluctuations of geodesics, and cube-root upper bounds with a logarithmic correction for distributional Busemann limits and competition interface limits. Several other applications are already in the literature. 
\end{abstract}

\maketitle

\tableofcontents

\section{Introduction}

\subsection{Purpose of this work} 

This paper introduces a new method for deriving probability estimates for directed planar polymer and percolation models. This method operates naturally in the context of  the probabilistic coupling approach. It utilizes a generating function of the process with two   distinct boundary conditions on the axes.  It yields probability bounds of optimal exponential order, for the first time within the coupling approach. This allows the coupling approach to match the strength of the integrable probability approach in the treatment of a variety of problems of planar random growth. 

We develop the technique for the {\it exponential corner growth model}  since this is the most-studied example  in the KPZ class. The ideas will transfer readily to other planar models with tractable increment or ratio stationary versions.  Indeed,  since the first version \cite{emra-janj-sepp-20-md} of this paper,  some key aspects of our approach have been implemented for the geometric corner growth model \cite{groa-janj-rass-21},  integrable lattice polymers \cite{land-soso-22a,  xie-22-phd},  the O'Connell-Yor polymer \cite{land-soso-22b} and,  most remarkably,  a nonintegrable\footnote{In the sense that there is no integrable structure beyond the existence of product-form invariant measures.} model of interacting diffusions \cite{land-soso-22a} that includes the O'Connell-Yor polymer (introduced in \cite{oconn-yor-01}) as a special case. 

Our main result gives control of the exit point of the geodesic from the axes.  This is often the starting point for applications of the coupling methodology because it is based on comparisons of several versions of the process.  
From these bounds follow a number of applications.  Some applications are described in this paper.  Other applications\footnote{Some of the cited works utilized an earlier preprint \cite{emra-janj-sepp-20-md} of this paper.} concern optimal-order (at least up to logarithmic factors) central moment bounds for last-passage times \cite[Theorem 3.1]{emra-geor-ortm-22},  and transversal fluctuations and coalescence bounds for finite and semi-infinite geodesics \cite[Theorems 2.2 and 2.8]{busa-ferr-22},  \cite[Theorem 2.3]{shen-sepp-19}.  These results were previously inaccessible to the coupling approach,  and could be proved only through inputs from integrable probability.  

Next this introduction describes the broader KPZ class, the coupling approach to their study, our results and their first applications, and related literature. The reader familiar with the subject can check our notation in Section \ref{SExPt}  and proceed to the results in Sections \ref{SExBd}  and \ref{SApp}.  The organization of the paper is described in Section \ref{SOrg} below. 

\subsection{Kardar-Parisi-Zhang class of planar stochastic models} 
 
The Kardar-Parisi-Zhang (KPZ) universality \cite{kard-pari-zhan-86} predicts long-time fluctuations in one-dimensional, out-of-equilibrium, stochastic interface growth with nonlinearly slope-dependent vertical speed, local only constraints, a smoothing mechanism and rapidly decorrelating space-time noise. After a long time $T > 0$, the interface is expected to display height fluctuations on $T^{1/3}$ scale and nontrivial spatial correlations on $T^{2/3}$ scale. Furthermore, with suitable centering and rescaling, the evolution of the interface is expected to converge to that of a universal limiting interface, the KPZ fixed point \cite{mate-quas-reme-21}, started from the limiting initial conditions. The broader significance of these predictions is that the same scaling and, at least to a large extent, the limit behavior are also believed and partially confirmed to arise in a diverse array of probabilistic models and physical systems. These form the KPZ universality class in $1+1$ (one space and one time) dimensions, and include certain interacting particle systems, random matrix ensembles, stochastic PDEs, and models of polymers in random media, growth of bacterial colonies, and liquid percolation. In-depth introductions to the subject from both mathematical and physical perspectives can be found in the review articles \cite{Cor-12, krie-krug-10-rev, quas-12-rev, Qua-Spo-15}. The short surveys \cite{baik-22, Cor-16, gang-22} also provide a summary of major milestones and some interesting research directions in this area. 

From a mathematical perspective, much of the KPZ prediction remains conjectural. However, for a small class of integrable (exactly-solvable) models with special structure, it has been possible to rigorously verify some aspects of the KPZ universality. The field of {\it integrable probability}, which exploits the structural properties of models to the greatest extent, often provides the most precise results on the KPZ class models and presently offers the only feasible path for rigorous analysis to the point of computing limit distributions. In the integrable approach, one first derives explicit formulas in the form of Fredholm determinants for the expectations of certain observables of interest, typically through versions of the Robinson-Schensted-Knuth (RSK) correspondence, the machinery of determinantal point processes or the Bethe ansatz, and then computes a suitable limit via the methods of asymptotic analysis. This line of argument was first demonstrated in the breakthrough articles \cite{baik-deif-joha-99, joha} and has since expanded remarkably in scope \cite{boro-gori-16-lec, boro-petr-14-surv, corw-14-icm, joha-06-lec, zygo-18}. 

The techniques of integrable probability do not seem well-suited, however, to study the KPZ universality beyond the integrable settings due to being too closely tied to the special, model-specific properties. Therefore, the development of alternative methodologies reliant on more broadly available structures and with greater potential for generalization merits research attention. 

\subsection{Coupling approach} 

One promising probabilistic approach capable of identifying the KPZ scaling exponents is the coupling method introduced by E.\ Cator and P.\ Groeneboom in the context of Hammersley's process \cite{cato-groe-05,  cato-groe-06}. This is a particularly versatile scheme that has since been further developed and fruitfully adapted to a variety of KPZ class models, including some directed percolation models \cite{bala-cato-sepp, ciec-geor-19}, particle systems with nearest-neighbor interaction \cite{bala-komj-sepp-12jsp, bala-komj-sepp-12, bala-sepp-aom} and directed polymers \cite{bala-quas-sepp, bala-rass-sepp-19, chau-noac-18b, more-sepp-valk-14, sepp-12-aop}. 
In broad strokes, the method compares a model under study with its stationary versions through suitable couplings, and likely produces results as long as the latter models are sufficiently tractable. For example, it would at least in principle be applicable in all integrable directed percolation and polymers on the integer quadrant considered in \cite{baik-rain-01a, bara-corw-17, chau-noac-18, corw-sepp-shen-15, ocon-ortm-15, rain-00, thie-ledo-15}, including their inhomogeneous generalizations. Furthermore, the variants of these models with general i.i.d.\ weights having finite $p$th moment for some $p > 2$ also possess stationary versions \cite{geor-rass-sepp-17-buse, Jan-Ras-20}, although no longer in explicit form. This raises the attractive, albeit presently highly speculative, prospect that the coupling method can potentially be improved in the future to the extent of being able to study the KPZ exponents in such non-integrable settings of great interest. As perhaps 
encouraging developments in this respect,  some aspects of limit shapes and geodesics in i.i.d.\ directed percolation as well as the positive temperature counterparts of these objects in i.i.d.\ directed polymers have been successfully studied in recent works \cite{geor-rass-sepp-17-geod, geor-rass-sepp-17-buse, Jan-Ras-20, Jan-Ras-Sep-20-} through coupling arguments. 

On the other hand, there are also drawbacks to the coupling approach. Besides being unable to access the KPZ limit distributions,  prior to the present work,  
the coupling method produced weaker than optimal results in some applications.  Most notably,  it provided only polynomially decaying\footnote{Before this work,  the best left-tail fluctuation bounds for last-passage times accessible via the coupling approach had cubic decay; see \cite{bala-cato-sepp,  sepp-cgm-18} for example.  On the polymer side,  the results of \cite{noac-soso-20a, noac-soso-20b} imply polynomially (of arbitrary degree) decaying fluctuation bounds for the free energy in the stationary versions of the O'Connell-Yor polymer and integrable lattice polymers.  We also remind here that optimal-order/sharp exponential tail bounds are available through integrable or random matrix techniques \cite{baik-etal-01,  ledo-ride-10}.} left-tail bounds for the last-passage time and free energy,  which was a main source of limitations and suboptimalities in various results.  For example,  before this work, optimal-order bounds were available via the coupling approach only for certain low central moments of these random variables. In a recent advance, refining the coupling method suitably, preprints \cite{noac-soso-20a, noac-soso-20b} managed to establish nearly optimal (with an $\epsilon$-deficiency in the exponents) bounds for all central moments of the free energies in the stationary versions of the O'Connell-Yor polymer and the four basic integrable lattice polymers. As our work demonstrates, however, there is still significant room for further fundamental improvements to the method. 

The purpose of this article is to optimize the coupling method in a key aspect, namely, the exit point bounds. {\it As a result, the method is brought on par with integrable probability in handling a variety of problems of interest.}  
Being able to achieve optimal results via the coupling approach is of some significance because the method can be preferable 
in these situations on account of its aforementioned virtues. 


\subsection{Exponential last-passage percolation}

Our setting is the last-passage percolation (LPP) on the nonnegative integer quadrant $\bbZ_{\ge 0}^2$ with independent exponential weights. The rates of the exponentials equal $1$ in the bulk $\bbZ_{>0}^2$, $w$ on the horizontal axis $\bbZ_{>0} \times \{0\}$ and $1-z$ on the vertical axis $\{0\} \times \bbZ_{>0}$ for some parameters $w > 0$ and $z < 1$. The weight at the origin is irrelevant and set to zero. The basic objects of study are the last-passage times and geodesics defined in Section \ref{SsLpp}. 

The exponential LPP is among the most-studied integrable models in the KPZ universality class \cite{bala-cato-sepp, bena-corw, prah-spoh}, owing in large part to its close connection to the totally asymmetric simple exclusion process (TASEP) started with the two-sided product Bernoulli initial condition and a single second-class particle at the origin. More specifically, the initial occupation probabilities for the sites of $\bbZ_{>0}$ and $\bbZ_{<0}$ equal $z$ and $w$, respectively, assuming now that $w \in (0, 1]$ and $z \in [0, 1)$. For the present work, the particle system picture 
only serves as a motivating context. Our proofs are developed entirely within the LPP framework from the stationarity properties of the equilibrium case $w = z$. A point to stress here is that the choice of the exponential LPP (among the integrable settings to which the coupling approach applies) is not a requirement but made for concreteness as well as relative simplicity and significance of the model. Analogous developments to ours can likely be carried out in all previously listed integrable directed percolation and polymers. These extensions are left for future works. 

\subsection{Overview of main results 
and methodology} An important role in the coupling approach to the LPP is played by the \emph{exit points} of the geodesics out of the origin in the equilibrium regime. The most basic case is the exit point from the axes, which is the last vertex that the geodesic visits on the axes before entering into the bulk.  More generally, we consider the exit point from an arbitrary down-right path; see Section \ref{SsEx} for the precise notion. Our interest is in the right-tail event in which the exit point is at least a given distance away from a fixed base vertex on the down-right path. The collection of such events describes, for example, 
the transversal fluctuations of the geodesic.   

Our main results on the exit points in the exponential LPP are upper and lower bounds of matching cubic-exponential order for the right-tail fluctuations, covering primarily the equilibrium regime $w = z$ (Theorems \ref{TExitUB} and \ref{TExitLB}).  
Crucially for our purposes of empowering the coupling method, the preceding bounds are obtained 
utilizing no more than the knowledge of the explicit equilibrium models and the stationarity of the last-passage increments there. Also worth noting with a view to future extensions is that the stronger distributional feature known as the {\it Burke property} (see \eqref{EBurke}) is also not used at this stage,  although it does come in for our applications mentioned below.  Our exit point bounds can be equivalently rephrased in terms of increment-stationary path-to-point exponential LPP as well (Propositions \ref{PExitUB2} and \ref{PExitLB2}).

Before the present work, known proofs of exponentially decaying fluctuation upper bounds for geodesic exit points relied on LPP fluctuation upper bounds with exponential decay.  Examples of this approach can be found in the proofs of \cite[Theorem 11.1]{basu-sido-sly-16}, \cite[Lemma 3.6]{ferr-ghos-nejj-19} and \cite[Lemma 2.5]{Fer-Occ-18}. 
In particular,  these arguments achieve optimal-order cubic-exponential decay for exit points starting from LPP fluctuation upper bounds with exponent $3/2$.  Prior to \cite{emra-geor-ortm-22} which builds on this article,  the known techniques that can produce such bounds for the left tail analyze exact distributional formulas for last-passage times either via Riemann-Hilbert methods \cite{baik-etal-01, lowe-merk-01, lowe-merk-roll-02}, or via tridiagonalization methods applied to the closely related Laguerre unitary ensemble 
\cite{ledo-ride-10}, or via H.\ Widom's trace trick \cite{wido-02} combined with 
steepest-descent methods applied to the trace of the associated correlation kernel \cite{Bai-Fer-Pec-14}.  On the other hand,  before this article,  the coupling method provided only cubically decaying left-tail fluctuation upper bounds\footnote{After our work,  it has become possible to achieve optimal-order left-tail fluctuation upper bounds within the coupling framework through combining \cite[Proposition 4.5]{emra-geor-ortm-22} with the generic results of \cite{gang-hegd-20}.  See also \cite{land-soso-22b} where similar strategy is employed for the O'Connell-Yor polymer.} for last-passage times \cite{bala-cato-sepp}.  This made it challenging to obtain optimal-order exit point upper bounds via the coupling approach.  In fact, the best upper bounds available via the coupling approach prior to our work were cubically decaying \cite[Theorem 2.2 and 2.5]{bala-cato-sepp},  \cite[Lemma 2.2]{pime-18},  \cite[Lemma 3.7]{pime-21},  \cite[Proposition 5.9]{sepp-cgm-18}. 

The main novelty of the present work is that optimal-order exit point upper bounds are obtained here through the coupling method,  without inputs from integrable probability or random matrix theory.  This is achieved through a moment generating function identity previously observed in a preprint of E.\ Rains \cite{rain-00} and recorded as Proposition \ref{PLMId} below.  In \cite{rain-00},  this identity is derived from determinantal formulas developed in \cite{baik-rain-01a} for the distribution of last-passage times.  We give a short probabilistic proof utilizing the increment-stationary LPP process.  We also find that Rains' identity is a natural generalization of the well-known variance identity of Bal\'{a}zs-Cator-Sepp\"{a}l\"{a}inen; see \eqref{EVar} below.  The latter identity and its variants formed the basis of the fluctuation theory developed with the coupling approach for integrable directed LPP since the seminal articles \cite{bala-cato-sepp, cato-groe-05, cato-groe-06}. See, for example, the recent lecture notes covering the exponential LPP \cite{sepp-cgm-18}. The key observation in the present work is that \eqref{EVar} can be upgraded to Proposition \ref{PLMId} without leaving the coupling context, and the fluctuation theory based on the latter produces optimal results previously inaccessible via the coupling approach. 

From Rains' identity,  we first extract a cubic-exponential order upper bound for the first step probability of a  geodesic out of the origin (Proposition \ref{PGeoInStep}) for the increment-stationary LPP.  We then turn this bound into our main upper bound (Theorem \ref{TExitUB}) with the aid of a known distributional identity \cite[Lemma A.2]{sepp-cgm-18}(stated as Lemma \ref{LExitDistId} below) relating the exit points from the axes to those from general L-shaped paths.  The proof of our main lower bound (Theorem \ref{TExitLB}) follows the same broader strategy with a key step being the derivation of a cubic-exponential lower bound for the first step probability (Proposition \ref{PGeoInStepLB}).  The proof of the latter is more involved than the corresponding upper bound.  It combines a change-of-measure argument originated in \cite{bala-sepp-aom} together with our exit point upper bounds.  We point out that the technique from \cite{bala-sepp-aom} was previously adapted to the exponential LPP \cite{sepp-cgm-18} and has been employed recently to obtain coalescence bounds for semi-infinite geodesics \cite{shen-sepp-19}.  An intermediate result in \cite[Theorem 4.1]{shen-sepp-19} is a lower bound of optimal-order for the exit points from the axes.  Our main lower bound extends this result mainly to arbitrary down-right paths through a similar argument.  

We conclude this section by commenting on some hypotheses in our main results.  First,  our main lower bound requires an arbitrarily small but \emph{fluctuations-scale} distance between the exit point and the base vertex.  Such a condition is expected on grounds that exit points should exhibit a different decay behavior for \emph{small deviations},  which has been recently studied for the bulk model (with i.i.d.\ $\Exp(1)$ weights) in \cite{basu-bhat-21}.  Also,  our exit point bounds apply to geodesics with endpoints away from the axes.  For fluctuation bounds on \emph{steep} geodesics,  see \cite[Theorems 2.5 and 2.7]{basu-hoff-sly-22}. 

\subsection{Some applications and extensions}

Our second set of results demonstrates some initial applications of our exit point upper bounds.  Theorem \ref{TPathFluc} provides cubic-exponential order fluctuation upper bounds for the exit points of the bulk geodesics.  The order of decay in this result should be optimal on account of universality and a recent optimal-order cubic-exponential lower bound for the geodesics in the Poisson LPP \cite[Proposition 1.4]{hamm-sark-20}.  The authors employed some LPP moderate deviation bounds from \cite{lowe-merk-01, lowe-merk-roll-02} as the only inputs from integrable probability.  After the present article,  it has become possible to reproduce these inputs through the coupling approach \cite{emra-geor-ortm-22,  emra-janj-sepp-20-md}\footnote{A forthcoming update to our preprint \cite{emra-janj-sepp-20-md} will present a coupling proof of the right-tail lower bound analogous to the input \cite[Theorem 2.3]{hamm-sark-20} for the exponential LPP.},  which suggests that optimal-order lower bounds complementary to Theorem \ref{TPathFluc} can also be established within the coupling framework.  

The remaining applications utilize the Burke property.  The next one concerns \emph{Busemann functions},  namely,  the a.s.\ directional limits of the last-passage increments.  Our interest is in the speed of distributional convergence.  Theorem \ref{TBuse} provides speed upper bounds with cube-root decay (up to logarithms) in the bulk setting. This result is in the spirit of \cite[Theorem 2.1]{bala-busa-sepp-21},  which proved an upper bound for the total variation distance between the LPP increments in the bulk and in a suitable equilibrium model.  Compared with our speed bounds,  their result provides bounds with respect to a stronger metric but with a weaker decay rate; see Remark \ref{RBBS}.  Our proof is based on Proposition \ref{PGeoInStep} and the \emph{crossing lemma} (Lemma \ref{LCros}).  
We expect that our upper bounds are of optimal-order,  and it might be possible to deduce matching lower bounds from the coalescence and stability results of \cite{bala-busa-sepp-21,  busa-ferr-22} but we do not pursue this point here. 

We also study the speed of distributional convergence for the limiting direction of the \emph{competition interface},  the boundary between two geodesic subtrees sharing the same root vertex.  In the bulk case,  Theorem \ref{TCif} gives upper speed bounds with cube-root decay (up to logarithms).  Encountering the same quantitative bound as in Theorem \ref{TBuse} is not surprising since the distribution of the competition interface direction can be recovered from the Busemann functions in the limit.  This suggest that Theorem \ref{TCif} can possibly be extracted from Theorem \ref{TBuse} as a corollary but we could not find such an argument without some loss in the strength of the bound.  Our proof of Theorem \ref{TCif} is instead a separate application of Proposition \ref{PGeoInStep} and the crossing lemma.  We again expect that our upper bound is optimal in order and a matching lower bound might be achievable in view of \cite{bala-busa-sepp-21,  busa-ferr-22} but presently do not attempt putting together a proof.  Subsequently,  we examine the competition interface in the equilibrium case and observe a dramatically different behavior in the speed of distributional convergence.  Theorem \ref{TCifSt} provides matching order upper and lower speed bounds of cubic-exponential order.  This result is derived from the corresponding bounds in Propositions \ref{PGeoInStep} and \ref{PGeoInStepLB} for the first step probabilities.  

A natural extension of Theorems \ref{TCif} and \ref{TCifSt} would be to determine the speed of convergence for the competition interface for the full range of the boundary parameters $w > 0$ and $z < 1$.  (The bulk and equilibrium cases correspond to taking $(w, z) = (1, 0)$ and $w = z$,  respectively).  It is natural to expect that the convergence speed exhibits an interesting transition that parallels the one from \cite{ferr-mart-pime-06,  ferr-mart-pime-09} describing when the competition interface has a deterministic or random limiting direction.  Another follow-up would be to investigate whether our speed bounds for the competition interface translate to bounds for the second-class particle of TASEP.  As far as we are aware,  there are no results quantifying the speed of convergence for the latter.  We leave these extensions to future works to explore. 

The last pair of results in this work recasts our main upper and lower exit point bounds in terms of increment-stationary down-right-path-to-point exponential LPP.  These reformulations are stated as Theorem \ref{PExitUB2} and \ref{PExitLB2},  respectively.  We derive them as corollaries of Theorems \ref{TExitUB} and \ref{TExitLB} after developing a suitable generalization (Proposition \ref{PExitDistId}) of  the distributional identity \cite[Lemma A.2]{sepp-cgm-18}.  Some line-to-point special cases of these bounds have been utilized in some recent articles \cite{basu-busa-ferr-22,  ferr-ghos-nejj-19,  pime-18}.  Similarly,  our path-to-point extensions can be potentially useful in future works. 

\subsection{Related literature}
\label{SsLit}

We briefly touch on literature related to exit points,  Busemann functions and competition interfaces,  each of which is a major topic of research. 

Exit points (as defined in this work) are of interest since they capture geodesic fluctuations and are also closely connected to geodesic coalescence \cite{pime-16}.  An early work concerning exit points is \cite{joha-ptrf-00} which rigorously verified the \emph{wandering exponent} of geodesics as $2/3$ for the Poisson LPP.  This determines the correct scale of fluctuations for the geodesic exit points.  Starting with articles \cite{bala-cato-sepp,  cato-groe-06} employing the coupling approach and with article  \cite{basu-sido-sly-16} importing more powerful inputs from integrable probability,  an increasingly refined picture of fluctuations emerged through tail bounds.  By now,  fluctuation bounds for geodesic exit points have featured frequently in the literature,  often in service of proving some deeper properties of last-passage times.  Some earliest applications of exit point bounds obtain optimal-order variance bounds \cite{bala-cato-sepp,  cato-groe-06} through identities such as \eqref{EVar} below. Further applications concern the non-existence of infinite bigeodesics \cite{bala-busa-sepp-20,  basu-hoff-sly-22,  groa-janj-rass-21}, geodesic coalescence \cite{basu-sark-sly-19,  shen-sepp-19,  zhan-20}, modulus of continuity \cite{hamm-sark-20}, temporal correlations \cite{basu-gang-18,  basu-gang-zhan-21,  ferr-occe-19} and tightness \cite{cato-pime-15, chhi-ferr-spoh-16, Fer-Occ-18, pime-18} of last-passage times, 
empirical weight distribution along geodesics \cite{mart-sly-zhan-21} and mixing times of TASEP on a ring \cite{schm-sly-22} among other topics. 



Since the influential works 
\cite{lice-newm-96} and 
\cite{hoff-05, hoff-08},  Busemann functions have become a useful instrument in the  study of 
 geodesics in both undirected first-passage percolation (FPP) and directed LPP.  For an overview of the related literature,  see the surveys \cite{auff-damr-hans-17, rass-cgm-18}.  Following the approach of 
 \cite{newm-icm-95},  the existence of the a.s.\ Busemann limits in the exponential LPP was first proved 
 in \cite{ferr-pime-05} for a deterministic set of fixed directions of full Lebesgue measure.  (See also the earlier work \cite{wuth-02} where Busemann functions are constructed for the Poisson LPP).  The result was subsequently extended to each fixed direction in 
 \cite
{coup-11}. The limits were later established in broad generality in a joint work of the third author \cite
{geor-rass-sepp-17-buse}. Their result covers LPP with i.i.d.\ weights bounded from below and of finite $p$th moment for some $p > 2$, and applies to all directions except those into the closed (possibly degenerate) flat regions of the shape function with at least one boundary direction where the shape function is not differentiable. The lower bound requirement on the weights was 
removed afterwards in \cite{Jan-Ras-20}.  See also the discussion in \cite[Appendix A]{Jan-Ras-Sep-20-}.  
  
Competition interface was introduced in \cite{ferr-pime-05} in the context of exponential LPP as a notion of a boundary between competing growth processes.  See also the earlier work \cite{hagg-pema-98} which studied competition in exponential FPP.  A main significance of the competition interface in the exponential LPP is that it captures the trajectory of the second-class particle in the associated TASEP \cite{ferr-mart-pime-06, ferr-mart-pime-09, ferr-pime-05}. Consequently, an initial set of results on the competition interface comes from translating the predating TASEP literature on the second-class particle, e.g.\ \cite{ferr-92, ferr-font-94a, ferr-kipn-95, moun-guio-05, prah-spoh, sepp98ebp}. We refer the reader to \cite{ferr-mart-pime-06, ferr-mart-pime-09} for a more detailed account. 
The competition interface also naturally features in a characterization of the exceptional directions of non-coalescence in the geometry of the semi-infinite geodesics \cite{Jan-Ras-Sep-20-}. 

Article \cite
{ferr-pime-05} proved the a.s.\ existence of the limiting direction of the competition interface and identified its distribution explicitly.  (See also \cite{ferr-kipn-95} and \cite{moun-guio-05} for the corresponding result on the second-class particle).  Articles \cite
{cato-pime-13, ferr-mart-pime-09} extended the preceding result to the down-right-path-to-point exponential LPP assuming that the boundary path has asymptotic directions on both ends.  In another direction of generalization, \cite
{geor-rass-sepp-17-geod} proved the a.s.\ convergence for the LPP with general i.i.d.\ weights on $\bbZ^2_{>0}$ under the assumptions that the weights are bounded from below and have continuous distributions of finite $p$th moment for some $p > 2$, and the shape function is differentiable at the endpoints of its linear segments. The lower bound assumption on the weights was later eliminated in \cite
{Jan-Ras-20}.  
Although it is not within the focus of the present work, the fluctuations of the competition interface around its limiting direction, and the related fluctuations of the second-class particle have also received renewed attention recently.  See \cite{ferr-ghos-nejj-19, ferr-nejj-17} and the references there.  

\subsection{Organization of the paper} \label{SOrg} 

Section \ref{SExPt} defines the exponential LPP model,  its geodesics and exit points, and the increment-stationary version. The key generating function identity is in Proposition \ref{PLMId}. 
Section \ref{SExBd} records the main results, namely, matching upper and lower bounds on 
fluctuations of exit points.  Section \ref{SApp} collects some applications of the main bounds to transversal fluctuations of bulk geodesics,  and to speed of distributional convergence to Busemann functions and the limiting competition interface direction. 
Sections \ref{SPf} and \ref{SPfApp} contain the proofs. Appendix \ref{SAux}  contains auxiliary technical results for LPP with arbitrary real weights. Appendix \ref{SExBnd2} extends our main bounds to certain path-to-point LPP processes,  and relates our work to the recent exit point bounds from \cite{ferr-ghos-nejj-19} and \cite{pime-18}. 
Appendix \ref{SEstim} contains some basic estimates.

\subsection*{Notation and conventions}  

$\emptyset$ denotes the empty set. $\min \emptyset = \inf \emptyset = \infty$ and $\max \emptyset = \sup \emptyset = -\infty$. $\bbZ$ and $\bbR$ denote the sets of integers and reals, respectively. For $A \subset \bbR$, $x \in \bbR$ and relation $\square \in \{\ge, >, \le, <\}$, let $A_{\square x} = \{a \in A: a \square x\}$. For example, $\bbZ_{>0}$ stands for the set of positive integers.  

$[n] = \{1, 2, \dotsc, n\}$ for $n \in \bbZ_{>0}$, and $[0] = \emptyset$. For $x \in \bbR$, $x^+ = \max \{x, 0\}$ and $x^- = (-x)^+$. Also, $\lf x \rf = \sup \bbZ_{\le x}$ and $\lc x \rc = \inf \bbZ_{\ge x}$.  Our convention is that $\bbR^0 = \bbR^{[0]} = \bbR^{\emptyset} = \{\emptyset\}$.   

For a finite sequence $\pi = (\pi_i)_{i \in [n]}$ in $\bbZ^2$, $\ell(\pi) = n$ indicates the number of terms in (the \emph{length} of) $\pi$. We refer to the set $\{\pi_i: i \in [n]\}$ also as $\pi$. 

For any subset $A \subset S$ of an arbitrary space $S$, the indicator function $\one_A: S \to \{0, 1\}$ equals $1$ on $A$ and $0$ on the complement $S \smallsetminus A$. The cardinality of $A$ is denoted by $\hash A$. 

$X \sim \Exp(\lambda)$ for $\lambda > 0$ means that $X$ is a rate $\lambda$ exponential random variable with the mean $E(X)=\lambda^{-1}$ and the moment generating function $E(e^{tX})=\lambda(t-\lambda)^{-1} \one_{\{t < \lambda\}}+ \infty \one_{\{t \ge \lambda\}}$ for $t \in \bbR$. Also, $Y \sim -\Exp(\lambda)$ means that $-Y \sim \Exp(\lambda)$. If $X \sim \Exp(\lambda)$ and $Y \sim -\Exp(\mu)$ are independent, the distribution of $X+Y$ is denoted by $\Exp(\lambda)-\Exp(\mu)$.  

The same name (e.g.\ $C_0, c_0$) may refer to different constants that appear in various steps within a proof. 

\subsection*{Acknowledgements} 
The authors are grateful to Patrik Ferrari for helpful comments and reference suggestions on our preprint \cite{emra-janj-sepp-20-md} that have informed many improvements in this article.  The authors would also like to thank anonymous referees for their careful reading of our work,  pointers to the literature,  and insightful remarks. 


\section{Exit points in exponential last-passage percolation}
\label{SExPt}

This section contains a precise description of the model and the main tools utilized for its treatment in the present work.  

\subsection{Last-passage times with exponential weights}
\label{SsLpp}

Given parameters $w > 0$ and $z < 1$, consider independent random \emph{weights} $\{\wb^{w, z}(i, j): i, j \in \bbZ_{\ge 0}\}$ such that $\wb^{w, z}(0, 0) = 0$, and 
\begin{align}
\wb^{w, z}(i, j) \sim \begin{cases} \Exp(1) \quad &\text{ if } i, j > 0 \\ \Exp(w) \quad &\text{ if } i > 0, j = 0\\ \Exp(1-z) \quad &\text{ if } i = 0, j > 0. \end{cases} \label{Ewbd}
\end{align} 
These weights can be coupled  through a single collection  $\{\eta(i, j): i, j \in \bbZ_{\ge 0}\}$ of  i.i.d.\ $\Exp(1)$-distributed random real numbers by setting 
\begin{align}
\wh{\w}^{w, z}(i, j) &= \eta(i, j)\bigg(\one_{\{i, j > 0\}} + \frac{\one_{\{i > 0, \hspace{0.5pt} j = 0\}}}{w} + \frac{\one_{\{i=0, \hspace{0.5pt} j> 0\}}}{1-z}\bigg) \quad \text{ for } i,j \in \bbZ_{\ge 0} . \label{ECplBlkBd}
\end{align}
The boundary rates in \eqref{Ewbd} are chosen so that the case $w = z$ gives rise to the increment-stationary LPP process to be discussed in Section \ref{SstatLPP}. 

Let $\P$ denote the probability measure on the sample space of the $\eta$-variables, and $\E$ denote the corresponding expectation. 

Throughout, we employ the following notational conventions with respect to \eqref{ECplBlkBd} and various quantities defined from these weights. We drop one $z$ from the superscript when $w = z$ ($\wb^{z}=\wb^{z, z}$), and omit $w$ when $j > 0$ since there is no dependence on $w$ in that case. Similarly, $z$ is omitted when $i > 0$. Finally, to distinguish the \emph{bulk} weights (those on $\bbZ_{>0}^2$), we also remove the hat from the notation 
and write $\w(i, j) = \wb^{w, z}(i, j)$ when $i, j > 0$. 

A finite sequence $\pi = (\pi)_{k \in [\ell(\pi)]}$ in $\bbZ^2$ is called an \emph{up-right path} if $\pi_{k+1}-\pi_{k} \in \{(1, 0), (0, 1)\}$ for $k \in [\ell(\pi)-1]$. Let $\Pi_{p, q}^{m, n}$ denote the set of all up-right paths $\pi$ 
from $\pi_1 = (p, q) \in \bbZ^2$ to $\pi_{\ell(\pi)} = (m, n) \in \bbZ^2$. 

For $(m, n), (p, q) \in \bbZ_{\ge 0}^2$, introduce the \emph{last-passage time} from $(p, q)$ to $(m, n)$ by 
\begin{align}
\Gb^{w, z}_{p, q}(m, n) = \max_{\pi \in \Pi_{p, q}^{m, n}} \sum_{(i, j) \in \pi} \wb^{w, z}(i, j). \label{Elppbd}
\end{align}
The case $m, n, p, q > 0$ of \eqref{Elppbd} defines the bulk last-passage times 
\begin{align}
\G_{p, q}(m, n) = \Gb_{p, q}^{w, z}(m, n) = \max_{\pi \in \Pi_{p, q}^{m, n}} \sum_{(i, j) \in \pi} \w(i, j).  \label{Elpp}
\end{align}
For often-used initial points we abbreviate   $\Gb^{w, z}(m, n)=\Gb_{0,0}^{w, z}(m, n)$  and   $\G(m, n) =\G_{1,1}(m, n)$. 

A maximizing path $\pi \in \Pi_{p, q}^{m, n}$ in \eqref{Elppbd} (and in \eqref{Elpp}) is called a \emph{geodesic} (or $\wb^{w, z}$-geodesic to indicate the weights) from $(p, q)$ to $(m, n)$. Since the marginal distributions of $\wb^{w, z}$ given in \eqref{Ewbd} are continuous,  when $p \le m$ and $q \le n$,  a.s.\ there exists a unique geodesic $\pi_{p, q}^{w, z}(m, n) \in \Pi_{p, q}^{m, n}$ (also denoted by $\pi_{p, q}^{m, n, w, z}$ as convenient). 

\subsection{Exit points of geodesics from down-right paths}
\label{SsEx}

A \emph{down-right path} is a finite sequence $\nu = (\nu_k)_{k \in [\ell(\nu)]}$ in $\bbZ^2$ such that $\nu_{k+1}-\nu_{k} \in \{(1, 0), (0, -1)\}$ for $k \in [\ell(\nu)-1]$. A frequent special case for the sequel is when $\nu$ is the \emph{L-shaped} path $L_{p, q}^{m, n}$ from $\nu_1 = (p, n)$ to $\nu_{\ell(\nu)} = (m, q)$ such that $(p, q) \in \nu$ for some $(m, n), (p, q) \in \bbZ^2$ with $p \le m$ and $q \le n$. 

Let $\pi$ be an up-right path and $\nu$ be a down-right path. If $\pi \cap \nu \neq \emptyset$ then define the \emph{exit point} $Z_{\pi, \nu}$ of $\pi$ from $\nu$ as the unique index $k_0 \in [\ell(\nu)]$ such that 
\begin{align}
\nu_{k_0} = \pi_{l_0} \quad \text{ where } l_0 = \max \{l \in [\ell(\pi)]: \pi_{l} \in \nu\}. \label{EExitDef}
\end{align}
In other words, $\nu_{Z_{\pi, \nu}}$ is the last vertex of $\pi$ on $\nu$. See Figure \ref{FExit}. 

Fix a \emph{base vertex} $\nu_b = (i_0, j_0)$ on $\nu$ for some $b \in [\ell(\nu)]$. We represent the exit point also relative to $(i_0, j_0)$ by 
\begin{align}
\label{EExitPts}
\begin{split}
Z^\square_{\pi, \nu, i_0, j_0} = (Z_{\pi, \nu}-b)^\square \quad \text{ for both signs } \square \in \{+, -\}. 
\end{split}
\end{align}
For all paths $\pi,\nu$ and base points $\nu_b\in\nu$, this quantity is zero for at least one  choice of sign $\square \in \{+, -\}$.  

To guarantee the existence of exit points, paths are restricted in the sequel as follows. With $(p, q) = \pi_1$, assume that $\nu_{\ell(\nu)} \in \{u\} \times \bbZ_{\ge q}$ and $\nu_1 \in \bbZ_{\ge p} \times \{v\}$ for some $u, v \in \bbZ$. Necessarily, $u \ge p$ and $v \ge q$. Let $V_\nu$ denote the set of all vertices $(m, n) \in \bbZ_{\le u} \times \bbZ_{\le v}$ such that 
\begin{align}
m \ge i \text{ and } n \ge j \quad \text{ for some } (i, j) \in \nu. \label{EVDef}
\end{align}
Then, under the further assumption that $\pi_{\ell(\pi)} \in V_{\nu}$, the intersection $\pi \cap \nu \neq \emptyset$ as required for the definition of $Z_{\pi, \nu}$. Figure \ref{FExit} illustrates this.

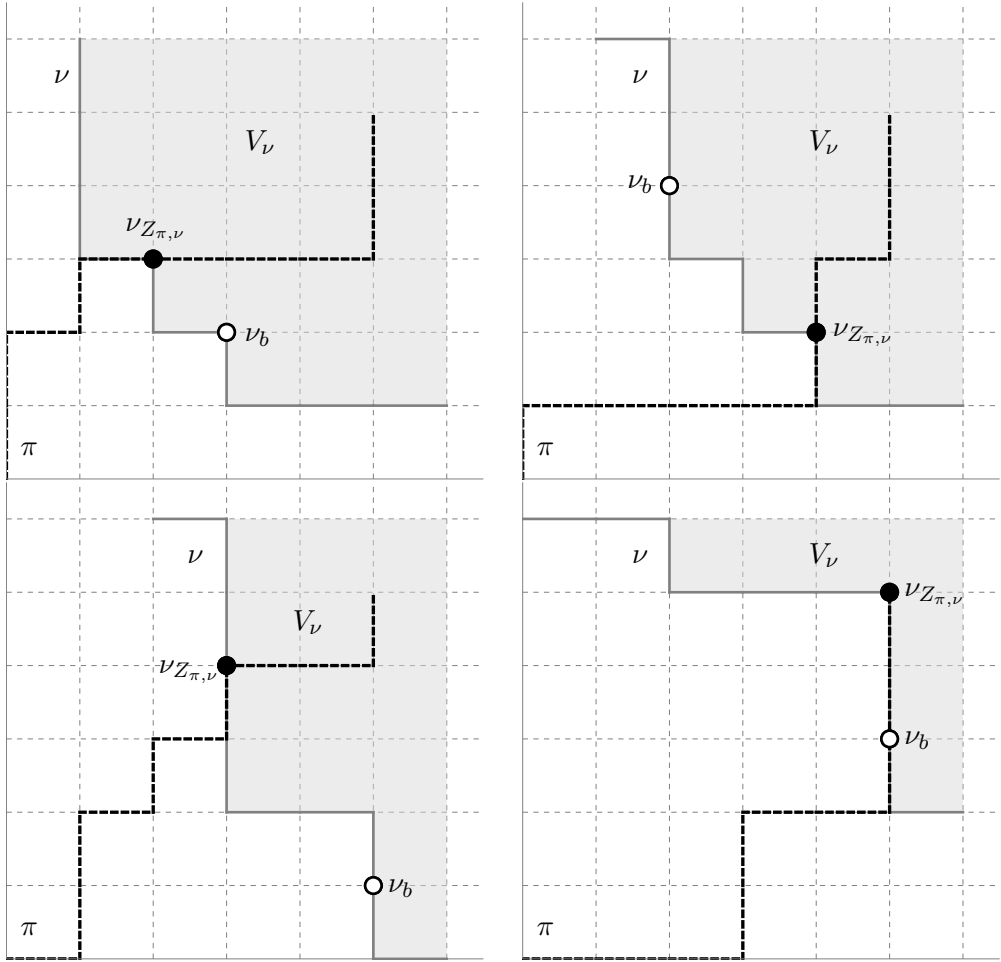
\begin{figure}
\centering
\begin{overpic}[scale=0.5]{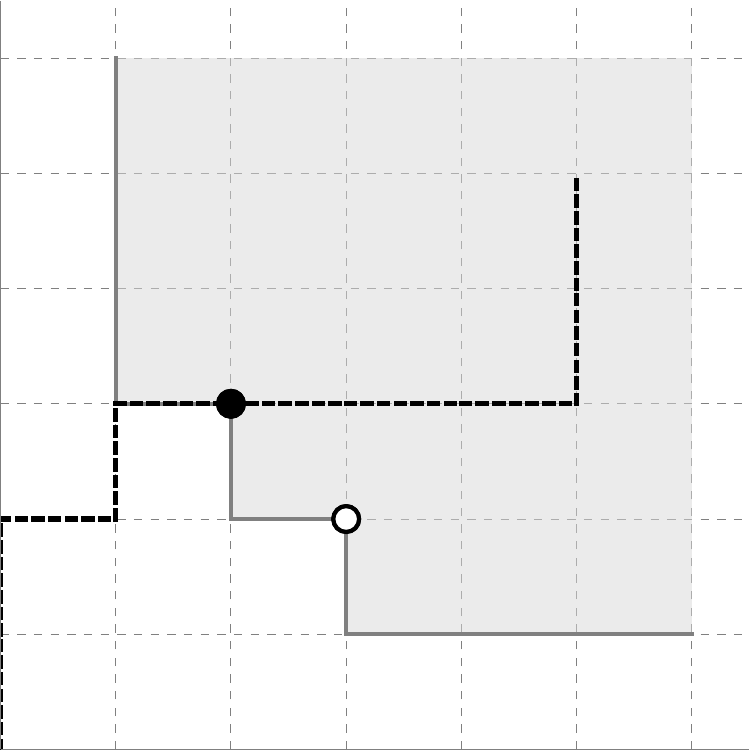}
\put(10, 83){$\nu$}
\put(3, 5){$\pi$}
\put(25, 52){$\nu_{Z_{\pi, \nu}}$}
\put(50, 29){$\nu_b$}
\put(50, 69){$V_{\nu}$}
\end{overpic}
\quad 
\begin{overpic}[scale=0.5]{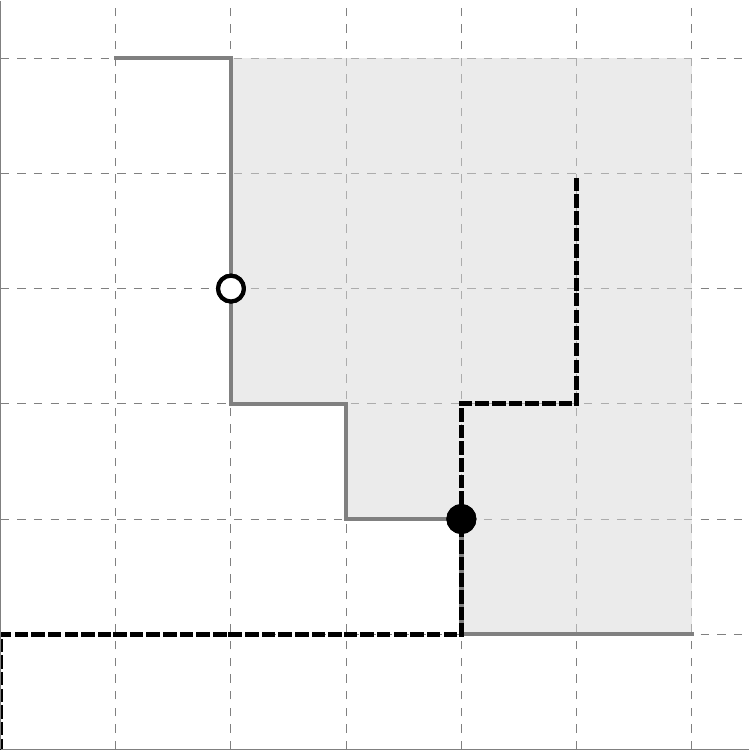}
\put(23, 83){$\nu$}
\put(3, 5){$\pi$}
\put(65, 30){$\nu_{Z_{\pi, \nu}}$}
\put(22, 61){$\nu_b$}
\put(60, 69){$V_{\nu}$}
\end{overpic}
\begin{overpic}[scale=0.5]{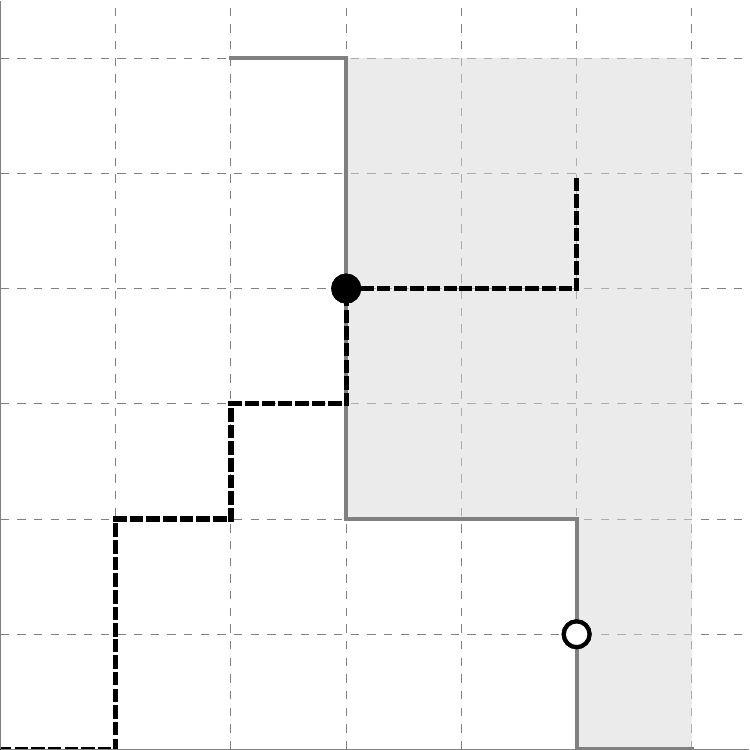}
\put(38, 83){$\nu$}
\put(3, 5){$\pi$}
\put(32, 60){$\nu_{Z_{\pi, \nu}}$}
\put(80, 14){$\nu_b$}
\put(60, 69){$V_{\nu}$}
\end{overpic}
\quad 
\begin{overpic}[scale=0.5]{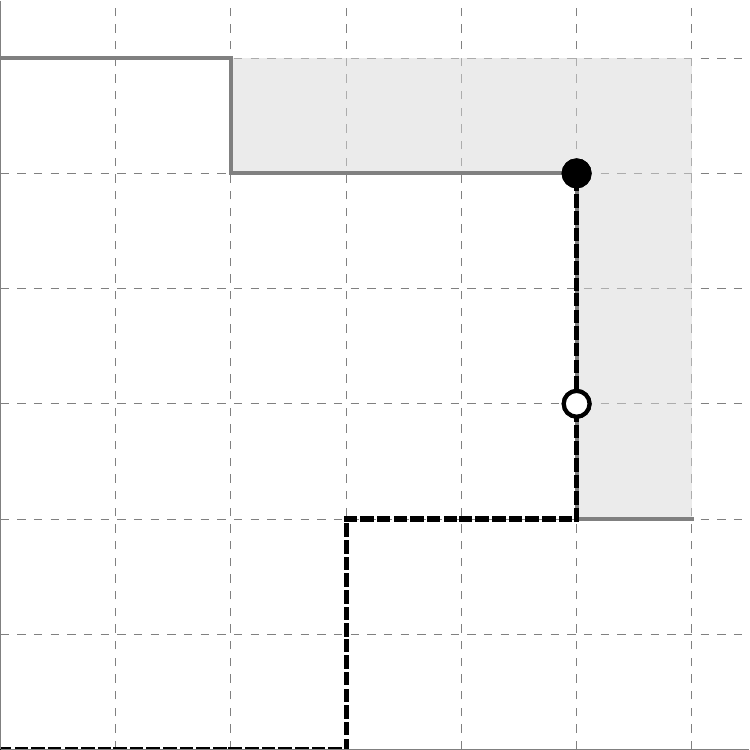}
\put(23, 83){$\nu$}
\put(3, 5){$\pi$}
\put(80, 76){$\nu_{Z_{\pi, \nu}}$}
\put(80, 45){$\nu_b$}
\put(60, 83){$V_{\nu}$}
\end{overpic}

\caption{Illustrations of the exit point of an up-right path $\pi$ (dashed, black) from a down-right path $\nu$ (gray). Vertex $\nu_{Z_{\pi, \nu}}$ (black dot), the base vertex $\nu_b = (i_0, j_0)$ (white dot) and the set $V_{\nu} \smallsetminus \nu$ (light gray) are indicated. The nonzero values $Z^{\pm} = Z_{\pi, \nu, i_0, j_0}^{\pm}$ above are as follows. Top left: $Z^- = 2$. Top right: $Z^+ = 4$. Bottom left: $Z^- = 5$. Bottom right: $Z^- = 2$.}
\label{FExit}
\end{figure}

Assume now that $\nu \subset \bbZ_{\ge 0}^2$. For each $(m, n) \in V_{\nu}$ and the choice of the sign $\square \in \{+, -\}$, define the (maximal) \emph{exit point} from $\nu$ of $\wb^{w, z}$-geodesics in $\Pi_{0, 0}^{m, n}$ by 
\begin{align}
\label{EEx}
\begin{split}
\Eh^{w, z, \square}_{\nu, i_0, j_0}(m, n) &= \max \{Z_{\pi, \nu, i_0, j_0}^\square: \pi \in \Pi_{0, 0}^{m, n} \text{ is a $\wb^{w, z}$-geodesic}\} \stackrel{\text{a.s.}}{=} Z_{\pi_{0, 0}^{w, z}(m, n), \nu, i_0, j_0}^\square. 
\end{split}
\end{align}
These are the main objects under study in this work. Definition \eqref{EEx} makes sense since $\pi \cap \nu \neq  \emptyset$ for each $\pi \in \Pi_{0, 0}^{m, n}$ by the restriction on $\nu$. The second equality above is due to the a.s.\ uniqueness of the geodesic $\pi_{0, 0}^{w, z}(m, n)$. 

When $\nu = L = L_{i_0, j_0}^{u, v}$ is the L-shaped path with lower left base vertex $\nu_b=(i_0, j_0)$ (where $b = v-j_0+1$),  \eqref{EEx} simplifies to 
\begin{align}
\label{EExL}
\begin{split}
\Eh^{w, z, +}_{L, i_0, j_0}(m, n) 
&\stackrel{\text{a.s.}}{=} [\max \{k \in \bbZ_{\ge 0}: (i_0+k, j_0) \in \pi^{w, z}_{0, 0}(m, n)\}]^+, \quad \text{ and } \\
\Eh^{w, z, -}_{L, i_0, j_0}(m, n) 
&\stackrel{\text{a.s.}}{=} [\max \{k \in \bbZ_{\ge 0}: (i_0, j_0+k) \in \pi^{w, z}_{0, 0}(m, n)\}]^+. 
\end{split}
\end{align}
To distinguish the case $i_0 = j_0 = 0$ where $L$ lies on the coordinate axes and contains the origin, we simplify the notation in  \eqref{EExL} to 
\begin{align}
\Eh^{w, z, \hor}(m, n) = \Eh^{w, z, +}_{L, 0, 0}(m, n) \quad \text{ and } \quad \Ev^{w, z, \ver}(m, n) = \Ev^{w, z, -}_{L, 0, 0}(m, n). \label{EExhv}
\end{align}
Equivalently,  $\Eh^{w, z, \hor}(m, n)$ is a.s.\ the distance that the geodesic $\pi_{0, 0}^{w, z}(m, n)$ spends on the horizontal axis, and $\Ev^{w, z, \ver}(m, n)$ the same for the vertical axis.

\subsection{Increment-stationary last-passage percolation}\label{SstatLPP}

A central role in our treatment belongs to the equilibrium (increment-stationary) versions of the exponential LPP and their characteristic directions. Let us briefly recall these notions.

For a down-right path $\nu$ in $\bbZ^2$, define the sets 
\begin{align}
\label{ERDsteps}
\begin{split}
R_\nu &= \{i \in [\ell(\nu)-1]: \nu_{i+1} = \nu_i + (1, 0)\}, \\ 
D_\nu &= \{j \in [\ell(\nu)] \smallsetminus \{1\}: \nu_{j-1} = \nu_j+(0, 1)\} 
\end{split}
\end{align}
that encode the right and down steps of $\nu$, respectively. See Figure \ref{FDRpath}. 

\begin{figure}
\centering
\begin{overpic}[scale=0.6]{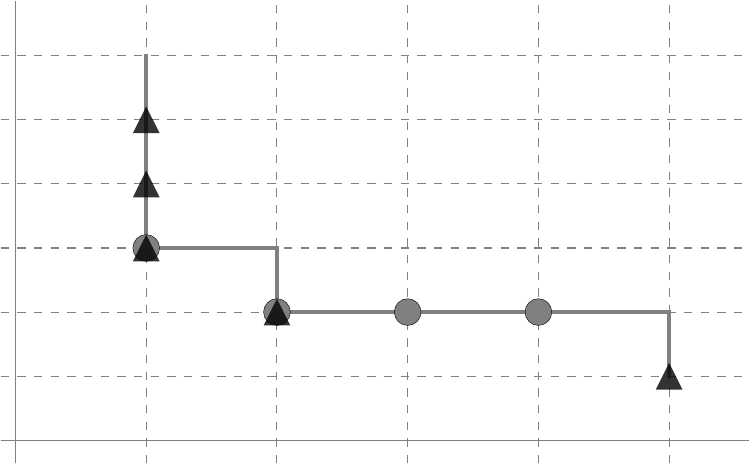}
\put(14, 48){$\nu$}
\end{overpic}
\caption{Illustrates a down-right path $\nu$ and the sets $R_{\nu}$ and $D_{\nu}$ defined at \eqref{ERDsteps}. Vertices in $R_{\nu}$ (dots) are the left endpoints of horizontal edges in $\nu$ while the vertices in $D_{\nu}$ (triangles) are the lower endpoints of the vertical edges in $\nu$. Two vertices lie in both $R_{\nu}$ and $D_{\nu}$.}
\label{FDRpath}
\end{figure}

By virtue of a version of Burke's theorem for LPP \cite[Lemma 4.2]{bala-cato-sepp}, for each $z \in (0, 1)$, the increments of the $\G^z$-process along any down-right path $\nu$ 
contained in $\bbZ_{\ge 0}^2$ enjoy this property: The collection 
\begin{align}
\label{EBurke}
\begin{split}
&\{\Gb^z(\nu_{i+1})-\Gb^{z}(\nu_{i}): i \in R_{\nu}\} \cup \{\Gb^z(\nu_{j-1})-\Gb^{z}(\nu_j): j \in D_{\nu}\} \\
&\text{ is jointly independent with the marginals given by } \\
&\Gb^z(\nu_{i+1})-\Gb^z(\nu_i) \sim \Exp\{z\} \quad \text{ for } i \in R_\nu \text{ and } \\
&\Gb^{z}(\nu_{j-1})-\Gb^{z}(\nu_j) \sim \Exp\{1-z\} \quad \text{ for } j \in D_\nu. 
\end{split}
\end{align}
As a consequence, the $\Gb^z$-process is increment-stationary in the sense that
\begin{align}
\Gb^z(m+p, n+q) - \Gb^z(p, q) \stackrel{\text{dist}}{=} \Gb^{z}(m, n) \quad \text{ for } m, n, p, q \in \bbZ_{\ge 0}. \label{Eincst}
\end{align}
Our exit point bounds (Theorems \ref{TExitUB}, \ref{TExitLB} and \ref{TPathFluc}, and Propositions \ref{PGeoInStep} and \ref{PGeoInStepLB}) are derived from \eqref{Eincst} while our applications of these results to the Busemann limits, competition interface and path-to-point exponential LPP (Theorems \ref{TBuse},  \ref{TCif} and \ref{TCifSt},  and Propositions \ref{PExitUB2} and \ref{PExitLB2}) rely on the stronger property \eqref{EBurke}. 

Introduce the function 
\begin{align}
\M^z(x, y) = \frac{x}{z} + \frac{y}{1-z} \quad \text{ for } x, y \in \bbR_{\ge0} \text{ and } z \in (0, 1). \label{EM}
\end{align}
It follows from \eqref{Eincst} that this function records the marginal means of the $\Gb^z$-process: $\E[\Gb^{z}(m, n)] = \M^z(m, n)$ for $m, n \in \bbZ_{\ge 0}$.
The curve $z \mapsto \M^z(x, y)$ for $z \in (0, 1)$ and some fixed $x, y \in \bbR_{>0}$ is plotted in Figure \ref{FMplot}. 

The \emph{shape function} of  the bulk $\G$-process can be expressed in terms of \eqref{EM} by 
\begin{align}
\Shp(x, y) = \inf_{z \in (0, 1)} \M^z(x, y) = (\sqrt{x}+\sqrt{y}\hspace{1pt})^2 \quad \text{ for } x, y \in \bbR_{\ge0}. \label{EShp}
\end{align}
A seminal result of H.\ Rost \cite{rost} identifies \eqref{EShp} as the limit $\lim \limits_{N\to\infty}N^{-1}\G(\lc Nx \rc, \lc Ny \rc) = \Shp(x, y)$ for $x, y \in \bbR_{>0}$ $\P$-a.s. 


The unique minimizer in \eqref{EShp} is given by 
\begin{align}
\Min(x, y) = \frac{\sqrt{x}}{\sqrt{x}+\sqrt{y}} \quad \text{ for } (x, y) \in \bbR_{\ge 0}^2 \smallsetminus \{(0, 0)\}. \label{EMin}
\end{align}
This function defines a bijection between  directions (unit vectors) in $\bbR_{\ge 0}^2$ and  the interval $[0, 1]$.  When  $z = \Min(x, y) \in (0, 1)$, the vector $(x, y)$ points in the   {\it characteristic direction}  of the $\Gb^z$-process.  
In this direction  the geodesics from the origin exit the boundary within a submacroscopic neighborhood of the origin. A precise version of the statement is contained in Theorem \ref{TExitUB} below. 

\begin{figure}
\centering
\begin{overpic}[scale=0.5]{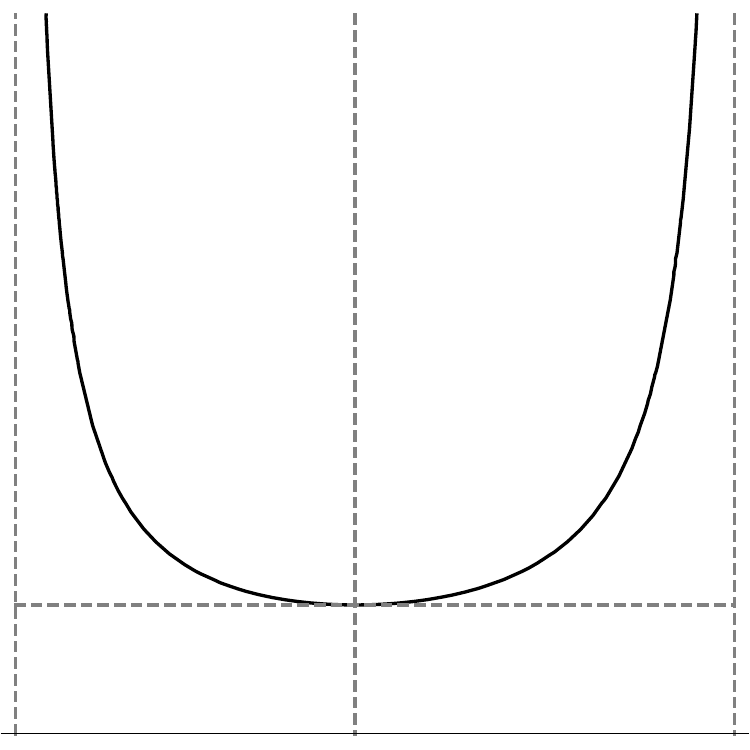}
\put (-16, 18){$\displaystyle \Shp(x, y)$}
\put (40, -3){$\displaystyle \Min(x, y)$}
\put (65, 60){$\displaystyle \M^{z}(x, y)$}
\put (0, -3){$0$}
\put (97, -3){$1$}
\end{overpic}
\caption{The plot of the curve $z \mapsto \M^z(x, y), z \in (0, 1)$ (black) with $x = 4$ and $y = 5$. The minimum value $\Shp(x, y)$ and the unique minimizer $\Min(x, y)$ are indicated.}
\label{FMplot}
\end{figure}
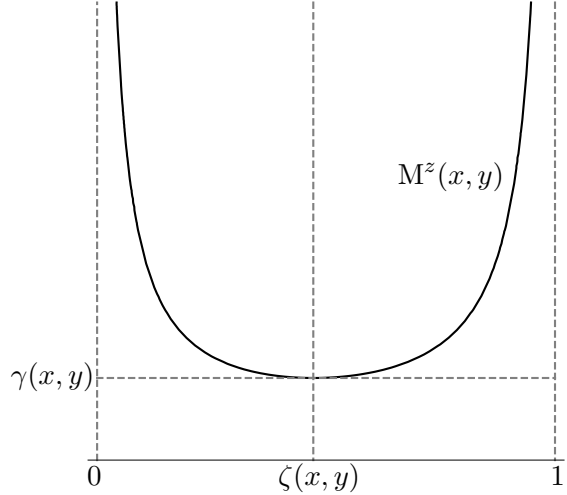

\subsection{The l.m.g.f.\ of the LPP process with boundary weights}

Define
\begin{align}
\Lb^{w, z}(x, y) = x\log \bigg(\frac{w}{z}\bigg)+y \log \bigg(\frac{1-z}{1-w}\bigg) \quad \text{ for } x, y \in \bbR_{\ge0} \text{ and } w, z \in (0, 1). \label{ELbd}
\end{align}
At the heart of our development is the following identity that links \eqref{ELbd} to the l.m.g.f.\ of the $\Gb^{w, z}$-process. 
\begin{prop}[\cite{rain-00}]
\label{PLMId}
Let $m, n \in \bbZ_{\ge 0}$ and $w, z \in (0, 1)$. Then 
\begin{align*}
\log \E[\exp\{(w-z) \Gb^{w, z}(m, n)\}] = \Lb^{w, z}(m, n). 
\end{align*}
\end{prop}
\begin{proof}
Recall \eqref{Ewbd}.  The second equality below  uses the product inside the expectation  as a Radon-Nikodym derivative.  This changes the rates of the exponential weights on the vertices $\{(i, 0): i \in [m]\}$ from $w$ to $z$. Then we use shift invariance \eqref{Eincst}. 
\begin{align*}
\log \E[e^{(w-z) \Gb^{w, z}(m, n)}]  
&=  \log \E\biggl[ \biggl(\;  \prod_{i=1}^m e^{(w-z) \wb^{w, z}(i,0)} \biggr)    e^{(w-z)(\Gb^{w,z}(m, n)-\Gb^{w,z}(m, 0))}\biggr] \\
&= m \log \bigg(\frac{w}{z}\bigg) + \log \E[e^{(w-z)(\Gb^{z}(m, n)-\Gb^{z}(m, 0))}] \\
&= m \log \bigg(\frac{w}{z}\bigg) + \log \E[e^{(w-z)\Gb^{z}(0, n)}] \\
&= m \log \bigg(\frac{w}{z}\bigg) + n \log \bigg(\frac{1-z}{1-w}\bigg) = \Lb^{w, z}(m, n). \qedhere
\end{align*}
\end{proof}

\begin{rem}
While the left-hand side of the identity makes sense for $w > 0$ and $z < 1$,  the restriction $w, z \in (0, 1)$ does not lose anything interesting.  Indeed,  if $w \ge 1$ and $n > 0$ then $w-z \ge 1-z \ge 0$ and 
\begin{align*}
\bfE[\exp\{(w-z)\mathrm{G}^{w,  z}(m, n)\}] \ge \bfE[\exp\{(1-z)\omega^{z}(0, 1)\}] = \infty
\end{align*}
because the weights are nonnegative and $\omega^z(0, 1) \sim \Exp(1-z)$.  Likewise,  the left-hand side is infinite if $z \le 0$ and $m > 0$.  Finally,  if either $m = 0$ or $n = 0$ then the left-hand side reduces to the l.m.g.f.\ of the sum of i.i.d.\  exponentials. 
\end{rem}
\begin{rem}A more general form of Proposition \ref{PLMId} appeared in a preprint of Rains \cite[Corollaries 3.3--3.4]{rain-00}. His version covers mixtures of the exponential and Poisson LPP, and mixtures of the geometric and Bernoulli LPP, and allows some inhomogeneity in parameters. \cite{rain-00} provides two proofs for the identity, both of which ultimately rely on exact determinantal formulas for the distribution of the last-passage times developed in \cite{baik-rain-01a}. 
The short argument above extends readily to the inhomogeneous exponential and geometric LPP but we have not attempted to verify this in the full setting of \cite{rain-00}.  Since the initial appearance of this work,  versions of Proposition \ref{PLMId} have been proved also in some integrable polymer models \cite{land-soso-22a, land-soso-22b, xie-22-phd} as well as for a nonintegrable model of interacting diffusions \cite{land-soso-22a} that generalizes the O'Connell-Yor polymer. 
\end{rem}
\begin{rem}\label{RmVarId}
The variance identity of Bal\'{a}zs-Cator-Sepp\"{a}l\"{a}inen \cite[Lemma 4.6]{bala-cato-sepp} can be recovered from Proposition \ref{PLMId}.  We give a formal calculation which can be made rigorous. Exponentiating, multiplying through by $\exp\{-(w-z)\M^z\}$ and Taylor expanding lead to 
\begin{align*}
\exp\{\Lb^{w, z}-(w-z)\M^z\} &= \E[\exp\{(w-z)(\Gb^{w, z}-\M^{z})\}] = \sum_{p = 0}^\infty \frac{(w-z)^p}{p!} \cdot \E[(\Gb^{w, z}-\M^z)^p], 
\end{align*}
where the vertex $(m, n)$ is dropped for brevity. Differentiating twice with respect to $w$ and setting $w = z$ yield 
\begin{align*}
\partial_w|_{w = z} \{\M^w(m, n)\} = 2 \partial_{w}|_{w = z} \{\E[\Gb^{w, z}(m, n)]\} + \Var[\Gb^z(m, n)]
\end{align*}
on account of the identity $\partial_w \Lb^{w, z} = \M^w$. Hence, with the left-hand side above written explicitly, one obtains that 
\begin{align}
\Var[\Gb^z(m, n)] = -\frac{m}{z^2} + \frac{n}{(1-z)^2} - 2\partial_w|_{w = z} \{\E[\Gb^{w, z}(m, n)]\}. \label{EVarIdDer}
\end{align}
This identity essentially appears within the proof of \cite[Lemma 4.6]{bala-cato-sepp}. The argument there computes the derivative in \eqref{EVarIdDer} in terms of 
the exit point, which results in the final form of the variance identity: 
\begin{align}
\Var[\Gb^z(m, n)] = -\frac{m}{z^2} + \frac{n}{(1-z)^2} + \frac{2}{z} \cdot \E\left[\sum_{i=0}^{\Eh^{z, \hor}(m, n)}\wb^z(i, 0)\right]. \label{EVar}
\end{align}
The analogue of \eqref{EVar} for the Poisson LPP was previously observed in \cite[Theorem 2.1]{cato-groe-06}. An early analogous identity relating the variance of the particle current to the expected position of the second-class particle in TASEP appeared in \cite{ferr-font-94a}.  Article \cite{bala-sepp-07JSP} generalized such identities to a broader class of stochastic particle processes. 

\end{rem}

\section{Main fluctuation bounds for exit points}
\label{SExBd}

We present our main results on the exit points. Theorems \ref{TExitUB} and \ref{TExitLB} below provide upper and lower fluctuation bounds in suitable regimes for the right tail of the exit points in \eqref{EEx}. A key point is that these bounds capture the correct cubic-exponential order of decay and are derived from the stationarity feature \eqref{Eincst} without integrable probability. Alternative formulations of the same results in terms of increment-stationary down-right-path-to-point exponential LPP are included in the appendix as Propositions \ref{PExitUB2} and \ref{PExitLB2}. 

To ensure the uniformity of various bounds, vertices are often restricted to the cone 
\begin{align}
\label{ESc}
S_\delta = \{(x, y) \in \bbR_{>0}^2: x \ge \delta y \text{ and } y \ge \delta x\}
\end{align}
for some fixed $\delta > 0$. 

\subsection{Main upper bounds}

The following theorem gives right-tail upper bounds on the exit point where  the geodesic  $\pi_{0,0}^{w, z}(m+i_0, n+j_0)$ 
leaves  a down-right path $\nu$, relative to the base vertex $(i_0, j_0)$ on $\nu$. The increment-stationary case is the one with $w = z$. 
Figure \ref{FExitUB} illustrates the statement.  The set $V_{\nu}$ was defined at \eqref{EVDef}.   


\begin{thm}
\label{TExitUB}
Fix $\delta > 0$. There exist finite positive constants $c_0 = c_0(\delta)$, $\epsilon_0 = \epsilon_0(\delta)$ and $N_0 = N_0(\delta)$ such that the following statements hold whenever $w > 0$, $z <1$, $(m, n) \in S_\delta \cap \bbZ_{\ge N_0}^2$, $s \ge (m+n)^{-2/3}$, $(i_0, j_0) \in \bbZ_{\ge 0}^2$, and $\nu$ is a down-right path on $\bbZ^2_{\ge 0}$ with $(i_0, j_0) \in \nu$ and $(m+i_0, n+j_0) \in V_{\nu}$. 
\begin{enumerate}[\normalfont (a)]
\item If $\min \{w, z\} \ge \Min(m, n)-\epsilon_0 s(m+n)^{-1/3}$ then 
\begin{align*}
\P\{\Eh^{w, z, +}_{\nu, i_0, j_0}(m+i_0, n+j_0) > s(m+n)^{2/3}\} \le \exp\{-c_0 \min \{s^3, m+n\}\}. 
\end{align*}
\item If $\max \{w, z\} \le \Min(m, n)+\epsilon_0 s(m+n)^{-1/3}$ then 
\begin{align*}
\P\{\Eh^{w, z, -}_{\nu, i_0, j_0}(m+i_0, n+j_0) > s(m+n)^{2/3}\} \le \exp\{-c_0 \min \{s^3, m+n\}\}. 
\end{align*}
\end{enumerate}
\end{thm}
\begin{rem}
\label{RPathLen}
Assume further that $\ell(\nu) \le C (m+n)$ for some constant $C > 0$. Then the probabilities above vanish for $s > C (m+n)^{1/3}$. Therefore, the bounds in the theorem can be replaced with $\exp\{-c_1 s^3\}$ where $c_1 = c_0 \min \{C^{-3}, 1\}$. 
\end{rem}



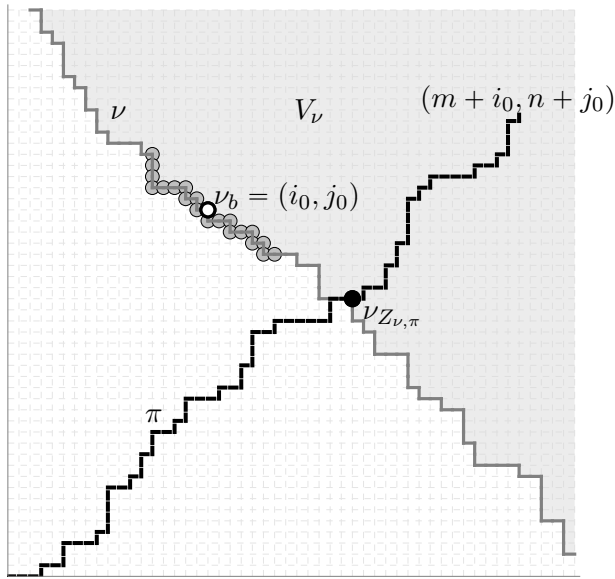
\begin{figure}
\centering
\begin{overpic}[scale=0.6]{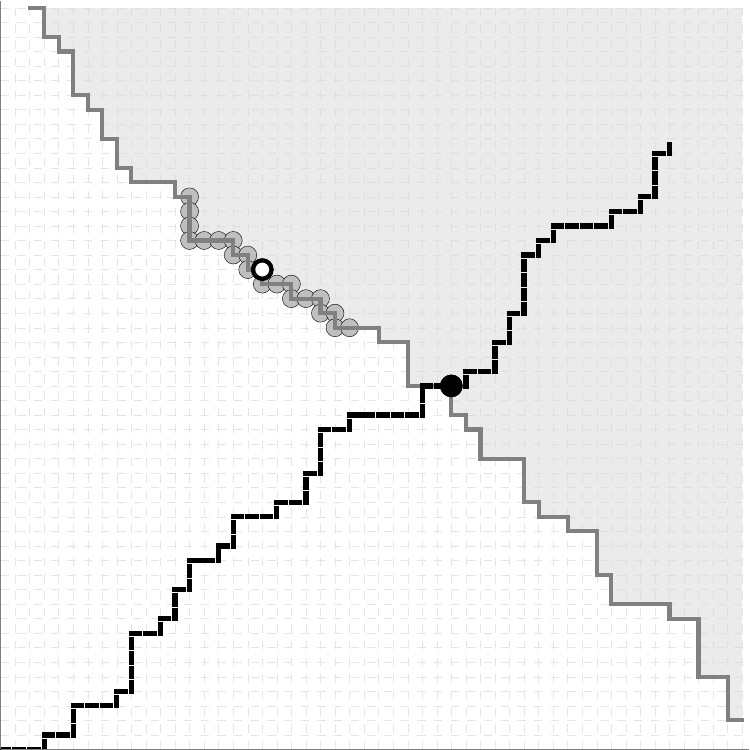}
\put(18, 80){$\nu$}
\put(50, 80){$V_\nu$}
\put(36, 65){$\nu_b = (i_0, j_0)$}
\put(72, 82){$(m+i_0, n+j_0)$}
\put(24, 27){$\pi$}
\put(62, 45){$\nu_{Z_{\nu, \pi}}$}
\end{overpic}
\caption{Illustrates Theorem \ref{TExitUB} in the case $w = z = \Min(m, n)$. A possible occurrence of the event that the geodesic $\pi = \pi^z_{0, 0}(m+i_0, n+j_0)$ to a vertex $(m+i_0, n+j_0)$ in $V_\nu$ exits a down-right path $\nu$ at a vertex (black dot) outside and below the $s(m+n)^{2/3}$ neighborhood (gray dots) of the base vertex $\nu_b = (i_0, j_0)$ (white dot) is shown. By part (a) of the theorem, the probability of this event is at most $\exp\{-c_0\min \{s^3, m+n\}\}$.}
\label{FExitUB}
\end{figure}

The following corollary for the L-shaped down-right path on the coordinate axes strengthens an exit point bound from the earlier version of this article \cite[Theorem 2.5]{emra-janj-sepp-20-md}.  A similar bound also appeared independently in \cite[Theorem 2.5]{Bha-20-JSP}. 
\begin{cor}
\label{CExitUB}
Fix $\delta > 0$. There exist finite positive constants $c_0 = c_0(\delta)$, $\epsilon_0 = \epsilon_0(\delta)$ and $N_0 = N_0(\delta)$ such that the following statements hold whenever $w > 0$, $z < 1$, $(m, n) \in S_\delta \cap \bbZ_{\ge N_0}^2$ and $s \ge (m+n)^{-2/3}$. 
\begin{enumerate}[\normalfont (a)]
\item If $\min \{w, z\} \ge \Min(m, n) - \epsilon_0 s(m+n)^{-1/3}$ then $$\P\{\Eh^{w, z, \hor}(m, n) > s(m+n)^{2/3}\} \le \exp\{-c_0s^3\}.$$ 
\item If $\max \{w, z\} \le \Min(m, n) + \epsilon_0 s(m+n)^{-1/3}$ then $$\P\{\Eh^{w, z, \ver}(m, n) > s(m+n)^{2/3}\} \le \exp\{-c_0s^3\}.$$ 
\end{enumerate}
\end{cor}
\begin{proof}
Apply Theorem \ref{TExitUB} with $(i_0, j_0) = (0, 0)$ and $\nu = L_{0, 0}^{m, n}$. The result then follows from Remark \ref{RPathLen} since $\ell(\nu) = m+n+1 \le 2(m+n)$. 
\end{proof}

\subsection{Upper bounds for the first step probabilities}

The following proposition provides upper bounds of optimal order for the probability of a geodesic from the origin taking the less likely initial step. 
The complementary lower bounds come  in Proposition \ref{PGeoInStepLB} below. 
The statement is a main ingredient in the proofs of several results in this work and is of independent interest. 

\begin{prop}
\label{PGeoInStep}
Fix $\delta > 0$. Let $(m, n) \in S_\delta \cap \bbZ_{>0}^2$, $\Min = \Min(m, n)$ and $z \in (0, 1)$. There exists a constant $c_0 = c_0(\delta) > 0$ such that the following statements hold. 
\begin{enumerate}[\normalfont (a)]
\item If $z > \Min$ then 
\begin{align*}
\P\{\Eh^{z, \hor}(m, n) > 0\} \le \exp\{-c_0(m+n)(z-\Min)^3\}. 
\end{align*}
\item If $z < \Min$ then
\begin{align*}
\P\{\Ev^{z, \ver}(m, n) > 0\} \le \exp\{-c_0(m+n)(\Min-z)^3\}. 
\end{align*}
\end{enumerate}
\end{prop}
\begin{rem}
\label{RConst}
By tweaking the proof of Proposition \ref{PGeoInStep}, it is possible to restate the preceding bounds with a precise constant and error term when $z$ is sufficiently close to $\Min$. For example, fixing $\delta > 0$ and setting $z = \Min + \dfrac{s}{\curv}$ where $\curv = \curv(m, n)$ is the scaling factor defined at \eqref{Ecurv} below, one has 
\begin{align}
\P\{\Ev^{z, \hor}(m, n) > 0\} \le \exp\bigg\{-\frac{s^3}{6} + \frac{C_0 s^4}{(m+n)^{1/3}}\bigg\} \label{EExitConst}
\end{align}
whenever $(m, n) \in S_\delta \cap \bbZ_{>0}^2$ and $0 \le s \le \epsilon (m+n)^{1/3}$ for some constants $C_0 = C_0(\delta) > 0$ and $\epsilon = \epsilon(\delta) > 0$. The leading order term in \eqref{EExitConst} arises from some optimal choices in the proof.  See Remark \ref{ROptConst} for more details.  We do not know whether the constant $-1/6$ is indeed sharp. 
\end{rem}

\subsection{Main lower bounds}

The next theorem states the lower bounds complementary to Theorem \ref{TExitUB}. Note that, unlike the situation in Theorem \ref{TExitUB}, the $s$ parameter is now bounded from below by some fixed $\epsilon > 0$ and the constants $N_0$ and $C_0$ depend on $\epsilon$.  

\begin{thm}
\label{TExitLB}
Fix $\delta > 0$,  $\epsilon > 0$ and $K \ge 0$. There exist finite positive constants $c_0 = c_0(\delta,  K)$, $C_0 = C_0(\delta, \epsilon,  K)$ and $N_0 = N_0(\delta, \epsilon,  K)$ such that the following statements hold whenenever $(m, n) \in S_\delta \cap \bbZ_{\ge N_0}^2$, $s \in [\epsilon, c_0(m+n)^{1/3}]$, $w > 0$,  $z < 1$, $(i_0, j_0) \in \bbZ_{\ge 0}^2$, and $\nu$ is a down-right path on $\bbZ^2_{\ge 0}$ with $(i_0, j_0) \in \nu$ and $(m+i_0, n+j_0) \in V_\nu$. 
\begin{enumerate}[\normalfont (a)]
\item If $\max \{w, z\} \le \Min(m, n) + K s (m+n)^{-1/3}$ then 
\begin{align*}
\P\{\Eh^{w, z, +}_{\nu, i_0, j_0}(m+i_0, n+j_0) > s(m+n)^{2/3}\} \ge \exp\{-C_0s^3\}. 
\end{align*}
\item If $\min \{w, z\} \ge \Min(m, n)-Ks(m+n)^{-1/3}$ then 
\begin{align*}
\P\{\Eh^{w, z, -}_{\nu, i_0, j_0}(m+i_0, n+j_0) > s(m+n)^{2/3}\} \ge \exp\{-C_0s^3\}. 
\end{align*}
\end{enumerate}
\end{thm}

The special case of Theorem \ref{TExitLB} where $(i_0, j_0) = (0, 0)$ and $\nu = L_{0, 0}^{m, n}$ gives the lower bounds that match the upper bounds in Corollary \ref{CExitUB}.

\begin{cor}
\label{CExitLB}
Fix $\delta > 0$,  $\epsilon > 0$ and $K \ge 0$. There exist finite positive constants $c_0 = c_0(\delta,  K)$, $C_0 = C_0(\delta, \epsilon,  K)$ and $N_0 = N_0(\delta, \epsilon,  K)$ such that the following statements hold whenever $(m, n) \in S_\delta \cap \bbZ_{\ge N_0}^2$, $s \in [\epsilon, c_0(m+n)^{1/3}]$, $w > 0$ and $z < 1$. 
\begin{enumerate}[\normalfont (a)]
\item If $\max \{w, z\} \le \Min(m, n) + K s(m+n)^{-1/3}$ then $$\P\{\Eh^{w, z, \hor}(m, n) > s(m+n)^{2/3}\} \ge \exp\{-C_0s^3\}.$$ 
\item If $\min \{w, z\} \ge \Min(m, n) - Ks(m+n)^{-1/3}$ then $$\P\{\Eh^{w, z, \ver}(m, n) > s(m+n)^{2/3}\} \ge \exp\{-C_0s^3\}.$$ 
\end{enumerate}
\end{cor}

\begin{rem}
\label{RExitLB}
The result implies the following lower bound in the increment-stationary case: Given $\delta > 0$ and $\epsilon > 0$, there exist positive constants $c_0 = c_0(\delta)$, $C_0 = C_0(\delta, \epsilon)$ and $N_0 = N_0(\delta, \epsilon)$ such that
\begin{align}
\P\{\max \{\Eh^{z, \hor}(m, n), \Ev^{z, \ver}(m, n)\} > s(m+n)^{2/3}\} \ge \exp\{-C_0 s^3\} \label{Exit_LB}
\end{align}
whenever $(m, n) \in S_\delta \cap \bbZ_{\ge N_0}^2$, $z \in (0, 1)$ and $s \in [\epsilon, c_0(m+n)^{1/3}]$. This bound has essentially the same content as a recent result in \cite[Theorem 4.1]{shen-sepp-19} that was also proved without integrable probability, by an adaptation of a  change-of-measure argument from \cite{bala-sepp-aom, sepp-cgm-18}. Our proof of Theorem \ref{TExitLB} is in a similar spirit. 
\end{rem}


\subsection{Lower bounds for the first step probabilities}

The next result gives lower bounds complementing the upper bounds in Proposition \ref{PGeoInStep}. The result serves as a main step in our proof of Theorem \ref{TExitLB}. 
\begin{prop}
\label{PGeoInStepLB}
Fix $\delta > 0$ and $\epsilon > 0$. There exist finite positive constants $C_0 = C_0(\delta, \epsilon)$,  $\epsilon_0 = \epsilon_0(\delta)$ and $N_0 = N_0(\delta, \epsilon)$ such that the following statements hold for $(m, n) \in S_\delta \cap \bbZ_{\ge N_0}^2$ and $z \in (0, 1)$ with $|z-\Min| \le \epsilon_0$ where $\Min= \Min(m, n)$. 
\begin{enumerate}[\normalfont (a)]
\item If $z \ge \Min+\epsilon (m+n)^{-1/3}$ then 
\begin{align*}
\P\{\Ev^{z, \hor}(m, n) > 0\} \ge \exp\{-C_0(m+n)(z-\Min)^3\}. 
\end{align*}
\item If $z \le \Min - \epsilon (m+n)^{-1/3}$ then 
\begin{align*}
\P\{\Ev^{z, \ver}(m, n) > 0\} \ge \exp\{-C_0(m+n)(\Min-z)^3\}. 
\end{align*}
\end{enumerate}
\end{prop}

\section{Applications of main exit point bounds}
\label{SApp}

We apply some results 
from Section \ref{SExBd} to obtain fluctuation upper bounds for the geodesics in the bulk, and speed 
bounds for the distributional convergence of the LPP increments and the competition interface direction. 

\subsection{Exit point upper bounds for geodesics in the bulk}
\label{SsBlkGeo}

In our framework, by virtue of the identity 
\begin{align}
\{\Gb^{1, 0}(m, n): (m, n) \in \bbZ^2_{\ge 0}\} \stackrel{\text{dist.}}{=} \{\G(m+1, n+1)-\w(1, 1): (m, n) \in \bbZ^2_{\ge 0}\},  \label{EBlkDistId}
\end{align}
bounds on transversal fluctuations of bulk geodesics can be expressed as right-tail bounds for the exit points in \eqref{EEx} when $w = 1$ and $z = 0$. The exit point bounds from Section \ref{SExBd} do not apply in this case due to the restrictions there on the  parameters. Nevertheless, utilizing the ordering of geodesics with a common endpoint, the upper bound in Theorem \ref{TPathFluc} below can be deduced from Theorem \ref{TExitUB} and Proposition \ref{PGeoInStep} in a fairly straightforward manner. 

\begin{thm}
\label{TPathFluc}
Fix $\delta > 0$. There exist finite positive constants $c_0 = c_0(\delta), \epsilon_0 = \epsilon_0(\delta), N_0 = N_0(\delta)$ and $s_0 = s_0(\delta)$ such that 
\begin{align*}\P\{\Eh^{1, 0, \square}_{\nu, i_0, j_0}(m+i_0, n+j_0) > s(m+n)^{2/3}\} \le \exp\{-c_0\min \{s^3, m+n\}\}\end{align*}
whenever $\square \in \{+, -\}$, $s \ge s_0$, $(m, n) \in S_\delta \cap \bbZ_{\ge N_0}^2$, $(i_0, j_0) \in \bbZ^2_{\ge 0}$ subject to 
\begin{align}
\label{EBaseCond}
\begin{split}
&|\Min(m+i_0+1, n+j_0+1)-\Min(m, n)| \le \epsilon_0s(m+n)^{-1/3}, 
\end{split}
\end{align} 
and $\nu$ is a down-right path in $\bbZ_{\ge 0}^2$ with $(i_0, j_0) \in \nu$ and $(m+i_0, n+j_0) \in V_{\nu}$. 
\end{thm}

\begin{rem}
Note from \eqref{EMin} that the left-hand side of \eqref{EBaseCond} equals zero if and only if the vertices $(0, 0)$, $(m, n)$ and $(m+i_0+1, n+j_0+1)$ are colinear. On the grounds of this and the continuity of the function $\Min$, condition \eqref{EBaseCond} can be interpreted as indicating the approximate colinearity of the preceding vertices. 
\end{rem}
\begin{rem}
\label{RPathFluc}
Uniform bounds similar to \cite[Theorem 2.8]{busa-ferr-22} but weaker by a logarithmic factor can be readily obtained from Theorem \ref{TPathFluc} via the union bound. To demonstrate, let $(x, y) \in S_\delta$ and $(p_k, q_k) = (\lf kx \rf, \lf ky \rf)$ for $k \in \bbZ_{\ge 0}$. Pick $N \in \bbZ_{>0}$, and consider the L-shaped paths $L_{k, N} = L_{p_k, 0}^{p_N, q_N}$ for $k \in [N] \cup \{0\}$. Theorem \ref{TPathFluc} implies the existence of positive constants $c_0, N_0$ and $s_0$ depending only on $\delta$ such that 
\begin{align*}
\begin{split}
\P\{\Eh^{1, 0, +}_{L_{k, N}, p_k, q_k}(p_N, q_N) > t_k (N-k)^{2/3} \text{ for } N-M \le k \le N-N_0\} \le \sum_{k=N-M}^{N-N_0} \exp\{-c_0t_{k}^3\} 
\end{split}
\end{align*}
whenever $M \in [N]$ with $M \ge N_0$ and $t_{k} \ge s_0$ for $N-M \le k \le N-N_0$. In particular, setting $t_k = s (\log M)^{1/3}$ above for some $s \ge s_0$, and choosing $s_0$ and $N_0$ sufficiently large yields 
\begin{align*}
\begin{split}
&\P\{\Eh^{1, 0, +}_{L_{k, N}, p_k, q_k}(p_N, q_N) > s (\log M)^{1/3} (N-k)^{2/3} \text{ for } N-M \le k \le N-N_0\} \\
&\le \exp\{-c_0 s^3 \log M\}
\end{split}
\end{align*}
after renaming the constant $c_0$.  
\end{rem}

\subsection{Speed of the distributional convergence to Busemann functions}

We now turn to an application of Proposition \ref{PGeoInStep} to bound from above the speed of distributional convergence of the bulk LPP increments to the Busemann functions. 


Denote the increments of the bulk LPP process with respect to the initial point by 
\begin{align}
\label{Einc}
\begin{split}
\B_{i, j}^{\hor}(m, n) &= \G_{i, j}(m, n)-\G_{i+1, j}(m, n) \\ 
\B_{i, j}^{\ver}(m, n) &= \G_{i, j}(m, n)-\G_{i, j+1}(m, n)  
\end{split}
\end{align}
for $m, n \in \bbZ_{>0}$, $i \in [m]$ and $j \in [n]$. By definition \eqref{Elpp} (and the convention that $\max \emptyset = -\infty$), these increments are equal to  $+\infty$ when $i = m$ and $j =n$, respectively. 

It is known \cite{coup-11,  ferr-pime-05} that, for any given direction vector $(x, y) \in \bbR_{>0}^2$,  there exists a stationary stochastic process $\{ \b_{i, j}^{\hor}(x, y), \b_{i, j}^{\ver}(x, y): i,j\in\bbZ_{>0}\}$ and  an event of full probability  on which the    limits   
\begin{align}
\lim_{N \rightarrow \infty}\B_{i, j}^{\hor}(m_N, n_N) &= \b_{i, j}^{\hor}(x, y) \quad \text{ and } \quad \lim_{N \rightarrow \infty}\B_{i, j}^{\ver}(m_N, n_N) = \b_{i, j}^{\ver}(x, y) 
\label{EBuse}
\end{align}
hold for all $(i, j) \in \bbZ^2_{>0}$ and 
sequences  $\{(m_N, n_N)\}_{N\ge 1}\subset\bbZ^2_{>0}$ such that 
$\min \{m_N, n_N\} \rightarrow \infty$ and $m_N/n_N\to x/y$. The limits \eqref{EBuse} are examples of \emph{Busemann functions} evaluated, respectively, at pairs $((i, j), (i+1, j))$ and $((i, j), (i, j+1))$ of adjacent vertices. 


The following distributional properties of the Busemann functions were obtained in \cite[Lemma 3.3]{cato-pime-13}. The marginal distributions are given by 
\begin{align}
\b^{\hor}_{i, j}(x, y) \sim \Exp\{\Min(x, y)\} \quad \text{ and } \quad \b^{\ver}_{i, j}(x, y) \sim \Exp\{1-\Min(x, y)\} \label{EBuseDis}
\end{align}
for $(i, j) \in \bbZ_{>0}^2$ and $(x, y) \in \bbR_{>0}^2$. Furthermore, for any down-right path $\nu$ in $\bbZ_{>0}^2$ (recalling the definition of $R_\nu$ and $D_\nu$ in \eqref{ERDsteps})
 the collection 
\begin{align}
\label{EBuseInd}
\begin{split}
&\{\b_{\nu_i}^{\hor}(x, y): i \in R_\nu\} \cup \{\b_{\nu_j}^{\ver}(x, y): j \in D_\nu\} \text{ is jointly independent. } 
\end{split}
\end{align} 
Comparing with \eqref{EBurke}, one recognizes that the Busemann functions in \eqref{EBuseInd} have the same joint distribution as the (absolute) increments of the $\Gb^{z}$-process along 
$\nu$, namely, the collection 
$
\{\Gb^{z}(\nu_{i+1})-\Gb^{z}(\nu_{i}): i \in R_{\nu}\} \cup \{\Gb^{z}(\nu_{j-1})-\Gb^{z}(\nu_{j}): j \in D_{\nu}\}
$
when $z = \Min(x, y)$. 

For each $p, q \in \bbZ_{\ge 0}$ and $z \in (0, 1)$, introduce the functions $\bcdf_{p, q}^{z, \hor}, \bcdf_{p, q}^{z, \ver}: \bbR^p \times \bbR^q \to [0, 1]$ by 
\begin{align}
\label{Ebcdf}
\begin{split}
\bcdf_{p, q}^{z, \hor}(s, t) &= \prod_{i \in [p]}\exp\{-s_i^-z\} \prod_{j \in [q]}(1-\exp\{-t_j^+(1-z)\}) \\
\bcdf_{p, q}^{z, \ver}(s, t) &= \prod_{i \in [p]}(1-\exp\{-s_i^+z\}) \prod_{j \in [q]}\exp\{-t_j^-(1-z)\} 
\end{split}
\end{align}
for $s = (s_i)_{i \in [p]} \in \bbR^p$ and $t = (t_j)_{j \in [q]} \in \bbR^q$.  Suppose $p = \hash R_\nu$ and $q = \hash D_{\nu}$ equal the number of down and right steps in $\nu$.  Then the functions $\bcdf_{p, q}^{z, \ver}$ and $\bcdf_{p, q}^{z, \hor}$ give two representations of the joint distribution 
of the increments $(\Gb^z(\nu_{k+1})-\Gb^z(\nu_{k}))_{k \in [\ell(\nu)-1]}$. Indeed, by definition \eqref{Ebcdf} and \eqref{EBurke}, 
\begin{align*}
\bcdf_{p, q}^{z, \ver}(s, t) &= \P\{\wb^z_{i, 0} \le s_i \text{ for } i \in [p] \text{ and } \wb^z_{0, j} \ge -t_j \text{ for } j \in [q]\} \\
&= \P\{\Gb^{z}(\nu_{r_{i}+1})-\Gb^{z}(\nu_{r_i}) \le s_i \text{ for } i \in [p] \text{ and } \\
&\qquad \qquad \ \ \Gb^z(\nu_{d_{j}})-\Gb^{z}(\nu_{d_j-1}) \le t_j \text{ for } j \in [q]\}, 
\end{align*}
where $(r_i)_{i \in [p]}$ and $(d_j)_{j \in [q]}$ denote the elements of $R_\nu$ and $D_\nu$ in increasing order. Similarly, $\bcdf_{p, q}^{z, \hor}$ gives the joint distribution 
of the
reversed
 increments $(\Gb^z(\nu_{k})-\Gb^z(\nu_{k+1}))_{k \in [\ell(\nu)-1]}$. In particular, from the discussion in the preceding paragraph, one has 
\begin{align}
\label{EBuseCdf}
\begin{split}
\bcdf_{p, q}^{\Min(x, y), \hor}(s, t) &= \P\{\b_{\nu_{r_i}}^{\hor}(x, y) \ge -s_i \text{ for } i \in [p] \text{ and } \b_{\nu_{d_j}}^{\ver}(x, y) \le t_j \text{ for } j \in [q]\} \\
\bcdf_{p, q}^{\Min(x, y), \ver}(s, t) &= \P\{\b_{\nu_{r_i}}^{\hor}(x, y) \le s_i \text{ for } i \in [p] \text{ and } \b_{\nu_{d_j}}^{\ver}(x, y) \ge -t_j \text{ for } j \in [q]\}. 
\end{split}
\end{align}

Let $m, n \in \bbZ_{>0}$ be sufficiently large such that $m \ge u$ and $n \ge v$ for any $(u, v) \in \pi$. In the same vein as in \eqref{EBuseCdf}, define the
pre-limiting functions 
$\Bcdf_{\nu}^{m, n, \hor}, \Bcdf_{\nu}^{m, n, \ver}: \bbR^p \times \bbR^q \to [0, 1]$ by 
\begin{align}
\label{EIncCdf}
\begin{split}
\Bcdf_{\nu}^{m, n, \hor}(s, t) &= \P\{\B_{\nu_{r_i}}^{\hor}(m, n) \ge -s_i \text{ for } i \in [p] \text{ and } \B_{\nu_{d_j}}^{\ver}(m, n) \le t_j \text{ for } j \in [q]\} \\
\Bcdf_{\nu}^{m, n, \ver}(s, t) &= \P\{\B_{\nu_{r_i}}^{\hor}(m, n) \le s_i \text{ for } i \in [p] \text{ and } \B_{\nu_{d_j}}^{\ver}(m, n) \ge -t_j \text{ for } j \in [q]\}
\end{split}
\end{align}
for $s \in \bbR^p$ and $t \in \bbR^q$. 
In light of \eqref{EBuse}, \eqref{EBuseDis} and \eqref{EBuseInd}, for fixed $\nu$ and $(x, y)$,
\begin{align}
\lim_{N \rightarrow \infty} \Bcdf_{\nu}^{\lc Nx \rc, \lc Ny \rc, \square}(s, t) = f_{p, q}^{\Min(x, y), \square}(s, t) \quad \text{ for each } s \in \bbR^p, t \in \bbR^q \text{ and } \square \in \{\hor, \ver\}. \label{EcdfLim}
\end{align}
A natural problem is then the speed of convergence. The next result provides some bounds in this direction. 

\begin{thm}
\label{TBuse}
Let $\delta > 0$ and $\epsilon > 0$. There exist constants $N_0 = N_0(\delta, \epsilon) > 0$ and $C_0 = C_0(\delta) > 0$ such that 
\begin{align*}
|\,\Bcdf_{\nu}^{m, n, \hor}(s, t)-\bcdf_{p, q}^{\Min(m, n), \hor}(s, t)\,| &\le C_0 (1+\one_{\{q \ge1\}}\log q) \bigg\{\frac{\log (m+n)}{(m+n)}\bigg\}^{1/3}\\
|\,\Bcdf_{\nu}^{m, n, \ver}(s, t)-\bcdf_{p, q}^{\Min(m, n), \ver}(s, t)\,| &\le C_0 (1+\one_{\{p \ge1\}}\log p) \bigg\{\frac{\log (m+n)}{(m+n)}\bigg\}^{1/3}
\end{align*}
whenever $(m, n) \in S_\delta \cap \bbZ_{\ge N_0}^2$, $\nu$ is a down-right path contained in $[1, \epsilon (m+n)^{2/3}]^2$, $p = \hash R_{\nu}$, $q = \hash D_{\nu}$, $s \in \bbR^p$ and $t \in \bbR^q$. 
\end{thm}
\begin{rem}
\label{RBBS}
Let us compare Theorem \ref{TBuse} with a related previous result from \cite{bala-busa-sepp-21}. Let $\epsilon > 0$, $N \in \bbZ_{>0}$ and $R_N = \bbZ^2_{>0} \cap [0, \epsilon N^{2/3}]^2$. Consider the total variation distance between the joint distributions of the two collections 
\begin{align*}
\{\B^{\square}_{i, j}(N, N): (i, j, \square) \in R_N \times \{\hor, \ver\}\} \quad \text{ and } \quad \{\b^{\square}_{i, j}(1, 1): (i, j, \square) \in R_N \times \{\hor, \ver\}\}. 
\end{align*}
\cite[Theorem 2.1]{bala-busa-sepp-21} shows that this distance is at most $C\epsilon^{3/8}$ for all $\epsilon \le c$  and $N \ge 1$ for some positive constants $c$ and $C$. (Their result covers all directions not just $(1, 1)$). The total variation is stronger as a metric than the c.d.f.\ distance, which is the notion of distance considered in Theorem \ref{TBuse}. Therefore, taking $\epsilon = N^{-2/3}$ above yields  
\begin{align*}
|\P\{\B^{\square}_{1, 1}(N, N) \le x\}-\P\{\b^{\square}_{1, 1}(N, N) \le x\}| \le CN^{-1/4} \quad \text{ for } x \in \bbR \text{ and } \square \in \{\hor, \ver\}.  
\end{align*}
This is weaker than the order $(\log N)^{1/3}N^{-1/3}$ bound provided by Theorem \ref{TBuse}. 
%
%
\end{rem}
\begin{rem}
Our expectation is that the upper bounds of Theorem \ref{TBuse} are optimal up to logarithmic factors but we are unable to verify this at the moment.  We are also unaware of \emph{any} lower bounds for the speed of convergence to the Busemann functions in \emph{any} LPP setting.  It would be nice to have lower bounds complementary to Theorem \ref{TBuse}; we leave this interesting problem as a topic of future works. 
\end{rem}

\subsection{Speed of the distributional convergence of the competition interface direction} \label{SsCif}
We describe one more application of Proposition \ref{PGeoInStep}, of a flavor similar to  Theorem \ref{TBuse}.  This time we bound from above the speed of distributional convergence  of the competition interface to its limiting direction. 

For the definitions in this section, restrict to the full probability event on which the geodesic $\pi^{m, n}$ from $(1, 1)$ to $(m, n)$ is unique for all $m, n \in \bbZ_{>0}$. Partition $\bbZ_{>0}^2 \smallsetminus \{(1, 1)\}$ into the subsets
\begin{align}
\tree^{\hor} &= \{(m, n) \in \bbZ_{>0}^2: (2, 1) \in \pi^{m, n}\} = \{(m, n) \in \bbZ_{>0}^2: \G_{2, 1}(m, n) > \G_{1, 2}(m, n)\} \label{ETh} \\
\tree^{\ver} &= \{(m, n) \in \bbZ_{>0}^2: (1, 2) \in \pi^{m, n}\} = \{(m, n) \in \bbZ_{>0}^2: \G_{2, 1}(m, n) < \G_{1, 2}(m, n)\}. \label{ETv}
\end{align}
As a consequence of planarity and the uniqueness of geodesics,  the sets above  enjoy the following structure:
\begin{align}
(k, l) \in \tree^{\hor} &\text{ implies that } \bbZ_{\ge k} \times [l] \subset \tree^{\hor}, \text{ and } 
\label{EThmon}\\
(k, l) \in \tree^{\ver} &\text{ implies that } [k] \times \bbZ_{\ge l} \subset \tree^{\ver}
\label{ETvmon}
\end{align}
See Figure \ref{Fcif}. 

The \emph{competition interface} 
is a notion of a boundary between $\tree^{\hor}$ and $\tree^{\ver}$ introduced by P.\ A.\ Ferrari and L.\ Pimentel in \cite{ferr-pime-05}. One precise definition of it is as the unique sequence $\cif = (\cif_n)_{n \in \bbZ_{>0}} = (\cif_n^{\hor}, \cif_n^{\ver})_{n \in \bbZ_{>0}}$ in $\bbZ_{>0}^2$ such that, for all $n \in \bbZ_{>0}$, 
\begin{align}
(\cif_n^{\hor}+1, \cif_n^{\ver}) \in \tree^{\hor}, \quad (\cif_n^{\hor}, \cif_n^{\ver}+1) \in \tree^{\ver} \quad \text{ and } \quad \cif_n^{\hor}+\cif_n^{\ver} = n+1. \label{Ecif} 
\end{align} 
The existence and uniqueness of $\cif$ can be seen from properties \eqref{EThmon}--\eqref{ETvmon}. The original definition from \cite{ferr-pime-05} describes the competition interface recursively as follows: 
\begin{align}
\cif_1^{\hor} =1, \quad \cif_n^{\hor} = \cif_{n-1}^{\hor} + \one\{\G(\cif_{n-1}^{\hor}+1, \cif_{n-1}^{\ver}) < \G(\cif_{n-1}^{\hor}, \cif_{n-1}^{\ver}+1)\} \label{Ecifh}\\
\cif_1^{\ver} = 1, \quad \cif_n^{\ver} =  \cif_{n-1}^{\ver} + \one\{\G(\cif_{n-1}^{\hor}+1, \cif_{n-1}^{\ver}) > \G(\cif_{n-1}^{\hor}, \cif_{n-1}^{\ver}+1)\} \label{Ecifv}
\end{align}
for $n \in \bbZ_{>1}$.  This says that $\cif$ always moves in the direction of the minimal increment of $\G$.  The equivalence of \eqref{Ecif} and \eqref{Ecifh}--\eqref{Ecifv} can be verified by induction.
With $\mrR_t = \{(m, n) \in \bbZ_{>0}^2: \G(m, n) \le t\}$, one can view the sets $\tree^{\hor} \cap \mrR_t$ and $\tree^{\ver} \cap \mrR_t$ as the states of two competing growth processes on $\bbZ_{>0}^2$, and the path $\{\varphi_n: n \in \bbZ_{>0} \text{ and } \G(\cif_n) \le t\}$ as the interface between them at time $t \ge 0$.  The study of competing growth began with 
\cite{hagg-pema-98} in the context of first-passage percolation with exponential weights. 

\begin{figure}
\centering
\begin{overpic}[scale=0.8]{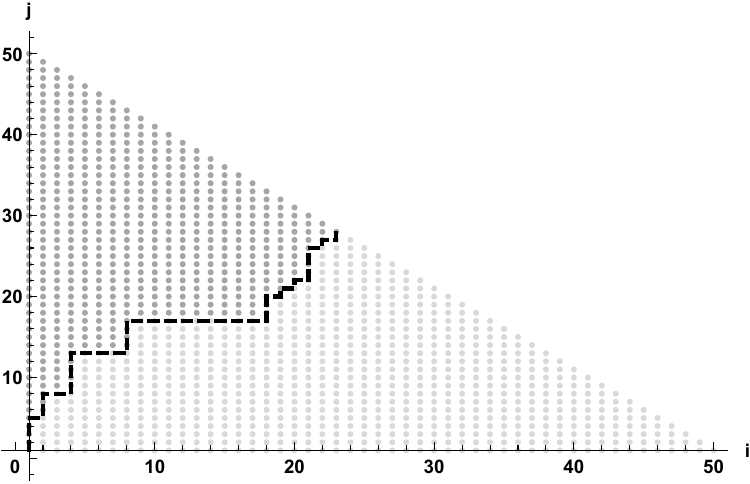}
\end{overpic}
\caption{A simulation of the first $N = 50$ steps of the competition interface (black). The vertices in $\tree^{\hor}$ (lighter gray) and $\tree^{\ver}$ (darker gray) below the antidiagonal $i+j = N+1$ are shown.}
\label{Fcif}
\end{figure}

The distributional limit of $\cif$ computed in \cite[Theorem 1]{ferr-pime-05} (see also \cite{ferr-kipn-95,  moun-guio-05}) can be phrased in our notation as follows: For $x \in [0, 1]$, 
\begin{align}
\lim_{n \rightarrow \infty} \P\{\cif_n^{\hor} \le nx \} = \Min(x, 1-x) = \frac{\sqrt{x}}{\sqrt{x}+\sqrt{1-x}} = \lim_{n \rightarrow \infty} \P\{\cif_n^{\ver} \le nx\}. \label{EcifLimDis}
\end{align}
The next result bounds the speed of convergence in \eqref{EcifLimDis} from above.  Similarly to the situation with Theorem \ref{TBuse},  we predict the upper bound of Theorem \ref{TCif} to be optimal up to logarithms but are unable to produce a matching-order lower bound at this time.  
\begin{thm}
\label{TCif}
Let $\delta > 0$. There exists a constant $C_0 = C_0(\delta) > 0$ such that 
\begin{align*}
|\P\{\cif_n^{\hor} \le nx\}-\Min(x, 1-x)| &\le C_0 \bigg(\frac{\log n}{n}\bigg)^{1/3} 
\end{align*}
for $x \in [\delta, 1-\delta]$ and $n \in \bbZ_{>1}$.
\end{thm} 

We next take an initial step towards generalizing Theorem \ref{TCif} to the exponential LPP with two-sided boundary.  Fix $w > 0$ and $z < 1$. Let $\wh{\cif}^{w, z} = (\wh{\cif}_n^{w, z})_{n \in \bbZ_{\ge 0}} = (\wh{\cif}_n^{w, z, \hor}, \wh{\cif}_n^{w, z, \ver})_{n \in \bbZ_{\ge 0}}$ denote the competition interface associated with the $\Gb^{w, z}$-process.  More precisely, let 
\begin{align}
\label{Etreebd}
\begin{split}
\wh{\sT}^{w, z, \hor} &= \{(m, n) \in \bbZ_{\ge 0}^2: \Gb_{1, 0}^{w, z}(m, n) > \Gb_{0, 1}^{w, z}(m, n)\}, \\
\wh{\sT}^{w, z, \ver} &= \{(m, n) \in \bbZ_{\ge 0}^2: \Gb_{1, 0}^{w, z}(m, n) < \Gb_{0, 1}^{w, z}(m, n)\}
\end{split}
\end{align} 
The geodesic $\pi_{0, 0}^{w, z}(m, n)$ is a.s.\ unique for each $(m, n) \in \bbZ_{\ge 0}^2$. On this event of probability one, define $\wh{\cif}^{w, z}$ as the unique sequence in $\bbZ^2_{\ge 0}$ such that 
\begin{align}
\label{Ecifbd}
\begin{split}
(\wh{\cif}_n^{w, z, \hor} + 1, \wh{\cif}_n^{w, z, \ver}) &\in \wh{\sT}^{w, z, \hor}, \quad (\wh{\cif}_n^{w, z, \hor}, \wh{\cif}_n^{w, z, \ver}+1) \in \wh{\sT}^{w, z, \ver}, \quad \text{ and } \\
&\wh{\cif}_n^{w, z, \hor} + \wh{\cif}_n^{w, z, \ver} = n \quad \text{ for } n \in \bbZ_{\ge 0}. 
\end{split}
\end{align}
Definition \eqref{Ecifbd} coincides with \eqref{Ecif} for $w = 1$ and $z = 0$ in the sense that the sequences  $(\wh{\cif}_n^{1, 0})_{n \in \bbZ_{\ge 0}}$ and $(\cif_{n+1}-(1, 1))_{n \in \bbZ_{\ge 0}}$ are equal in distribution. 

For any $w \in (0, 1]$ and $z \in [0, 1)$, the limit distribution of $\wh{\cif}^{w, z}$ has also been  computed explicitly \cite[Theorem 2]{ferr-mart-pime-09}, and is given by 
\begin{align}
\label{EcifLimDis2}
\begin{split}
\lim_{n \to \infty}\P\{\wh{\cif}^{w, z, \hor}_n \le n x\} &= \frac{\one\{w > z\}}{w-z}\left(\min \{w, \Min(x, 1-x)\}-\min \{z, \Min(x, 1-x)\}\right) \\
&+\one\{w \le z\}\one\left\{\Min(x, 1-x) > \rho^{w, z}\right\} \quad \text{ for } x \in [0, 1]
\end{split}
\end{align}
except at the point of discontinuity $\Min(x, 1-x) = \rho^{w, z} = \Min(wz, (1-w)(1-z))$ in the case $w \le z$.  

The next result provides matching-order upper and lower speed bounds for the distributional convergence of $\wh{\cif}^z$ in the equilibrium case $w = z$. The main point is to contrast the cubic-exponential decay below with the cube-root decay in Theorem \ref{TCif}.  
It would be nice to have optimal-order speed bounds for the full range of the $w$ and $z$ parameters illuminating the transition in the speed of convergence from cube-root to cubic-exponential decay.  We leave this to future works. 

\begin{thm}
\label{TCifSt}
Fix $\delta > 0$ and $\epsilon > 0$. The following statements hold for all $x \in [\delta, 1-\delta]$ and $z \in (0, 1)$ subject to the indicated assumptions. 
\begin{enumerate}[\normalfont (a)]
\item There exist positive constants $c_0 = c_0(\delta)$ and $A_0 = A_0(\delta)$ such that 
\begin{align*}
|\P\{\wh{\cif}_n^{z, \hor} \le n x\}-\one\{\Min(x, 1-x) > z\}| \le \exp\{-c_0 n |\Min(x, 1-x)-z|^3\}
\end{align*} 
whenever $n \in \bbZ_{>0}$ and $|\Min(x, 1-x)-z| \ge A_0n^{-1}$. 
\item There exist positive constants $\epsilon_0 = \epsilon_0(\delta)$,  
$C_0 = C_0(\delta, \epsilon)$ and $N_0 = N_0(\delta, \epsilon)$ such that 
\begin{align*}
|\P\{\wh{\cif}_n^{z, \hor} \le n x\}-\one\{\Min(x, 1-x) > z\}| \ge \exp\{-C_0 n |\Min(x, 1-x)-z|^3\}
\end{align*}
whenever $n \in \bbZ_{\ge N_0}$ and $|\Min(x, 1-x)-z| \in [\epsilon n^{-1/3},  \epsilon_0]$. 
\end{enumerate}
\end{thm}

\section{Proofs of the exit point bounds}\label{SPf}

We begin to prove our main results.  This part is divided into Sections \ref{SsPfInStepUB}-\ref{SsPfPathFluc} devoted to the proofs of Proposition \ref{PGeoInStep},  Theorem \ref{TExitUB},  Proposition \ref{PGeoInStepLB},  and Theorems \ref{TExitLB} and \ref{TPathFluc},  respectively.  The proofs of Theorems \ref{TBuse},  \ref{TCif} and \ref{TCifSt} will appear in Section \ref{SPfApp}.  

\subsection{Proof of the upper bounds for the first step probabilities}
\label{SsPfInStepUB}

For the proof of Proposition \ref{PGeoInStep}, let us first record a suitable Taylor approximation of the l.m.g.f.\ in \eqref{ELbd}. From definitions \eqref{EM} and \eqref{ELbd}, one has the identity 
\begin{align}
\Lb^{w, z}(x, y) = \int_{z}^w \M^t(x, y) \dd t \quad \text{ for } w, z \in (0, 1) \text{ with } w \ge z \text{ and } x, y \in \bbR_{\ge 0}. \label{ELbdMId}
\end{align}
Define the function  
\begin{align}
\curv(x, y) = \bigg(\frac{\Shp(x, y)}{\Min(x, y) (1-\Min(x, y))}\bigg)^{1/3} = \frac{(\sqrt{x}+\sqrt{y}\hspace{1pt})^{4/3}}{x^{1/6}y^{1/6}} \quad \text{ for } x, y \in \bbR_{>0}.  
 \label{Ecurv}
\end{align}
This is connected to \eqref{EM} through 
\begin{align}
\curv(x, y)^3 = \frac{1}{2}\partial_z^2\bigg|_{z = \Min(x, y)} \bigg\{\M^{z}(x, y)\bigg\} \quad \text{ for } x, y \in \bbR_{>0}. \label{Ecurv2}
\end{align}
 
\begin{lem}
\label{LLIBdEst}
Let $x, y \in \bbR_{>0}$ and $w, z \in (0, 1)$ with $w \ge z$. Abbreviate $\Shp = \Shp(x, y)$, $\Min = \Min(x, y)$ and $\curv = \curv(x, y)$. 
Fix $\delta > 0$ and $\epsilon > 0$. There exists a constant $C_0 = C_0(\delta, \epsilon) > 0$ such that 
\begin{align*}
\bigg|\Lb^{w, z}(x, y)-(w-z) \Shp- \dfrac{\curv^3}{3} \{(w-\Min)^3-(z-\Min)^3\}\bigg| &\le C_0 (x+y)\{(w-\Min)^4 + (z-\Min)^4\} \end{align*}
whenever $(x, y) \in S_\delta$ and $w, z \in (\epsilon, 1-\epsilon)$.
\end{lem}
\begin{proof}
This follows from \eqref{ELbdMId} and Lemma \ref{LMShpId}. 
\end{proof}

\begin{proof}[Proof of Proposition \ref{PGeoInStep}]
By symmetry, it suffices to prove (a). Assume $z > \Min$ and write $\lambda = (z-\Min)/4 > 0$. In the computations below, the arguments of the LPP values and various functions are fixed at the vertex $(m, n)$ and omitted. Using definitions \eqref{Elppbd} and \eqref{EExhv}, monotonicity, the Cauchy-Schwarz inequality and Proposition \ref{PLMId}, one arrives at 
\begin{align}
\label{E-2}
\begin{split}
\P\{\Ev^{z, \hor} > 0\} &= \P\{\Gb_{1, 0}^z \ge \Gb_{0, 1}^z\} \le \P\{\Gb_{1, 0}^{z-2\lambda} \ge \Gb_{0, 1}^{z}\} \\
&= \E[\exp\{\lambda\Gb_{1, 0}^{z-2\lambda}-\lambda\Gb^{z-2\lambda, z}\}\one\{\Gb_{1, 0}^{z-2\lambda} \ge \Gb_{0, 1}^{z}\}] \\
&\le \E[\exp\{\lambda \Gb_{1, 0}^{z-2\lambda}-\lambda \Gb^{z-2\lambda, z}\}] \\ 
&\le \E[\exp\{2\lambda \Gb_{1, 0}^{z-2\lambda}\}]^{1/2}\,\E[\exp\{-2\lambda \Gb^{z-2\lambda, z}\}]^{1/2} \\
&= \E[\exp\{2\lambda \Gb_{1, 0}^{z-2\lambda}\}]^{1/2}\exp\bigg\{\frac{1}{2}\Lb^{z-2\lambda, z}\bigg\} \\
&\le \E[\exp\{2\lambda \Gb_{}^{z-2\lambda, z-4\lambda}\}]^{1/2}\exp\bigg\{\frac{1}{2}\Lb^{z-2\lambda, z}\bigg\}\\
&= \exp\bigg\{\frac{1}{2}\Lb^{z-2\lambda, z-4\lambda} + \frac{1}{2}\Lb^{z-2\lambda, z}\bigg\} \\
&= \exp\bigg\{\frac{1}{2}\Lb^{z-2\lambda, z-4\lambda} - \frac{1}{2}\Lb^{z, z-2\lambda}\bigg\}. 
\end{split}
\end{align}
The minus sign in the final step comes from switching the order of the parameters $z-2\lambda$ and $z$ in the superscript, see definition \eqref{ELbd}. 

Since $(m, n) \in S_\delta$, the last exponent in \eqref{E-2} can be bounded by means of Lemmas \ref{LLIBdEst} and \ref{LShpMinBnd} as follows (see Figure \ref{FPrGeoUB}): For some constants $C_0, c_0, \epsilon > 0$ depending only on $\delta$, 
\begin{align*}
\Lb^{z-2\lambda, z-4\lambda} - \Lb^{z, z-2\lambda} &= \Lb^{\Min+2\lambda, \Min} - \Lb^{\Min+4\lambda, \Min+2\lambda} \\
&\le \bigg(2\lambda \Shp + \frac{8\lambda^3\curv^3}{3}\bigg) - \bigg(2\lambda \Shp + \frac{1}{3}( 64\lambda^3-8\lambda^3)\curv^3\bigg) + C_0 (m+n) \lambda^4\\
& = -16 \lambda^3 \curv^3+ C_0 (m+n) \lambda^4 \\
&\le -2c_0 (m+n)\lambda^3  + C_0(m+n) \lambda^4 \\ 
&\le -c_0(m+n)\lambda^3 
\end{align*}
provided that $\lambda \le \epsilon$. This completes the proof in the case $\lambda \le \epsilon$. 

When $\lambda \in (\epsilon, 1)$, the claimed bound also holds after adjusting $c_0$ by a constant factor dependent only on $\epsilon = \epsilon(\delta)$. 
\end{proof}

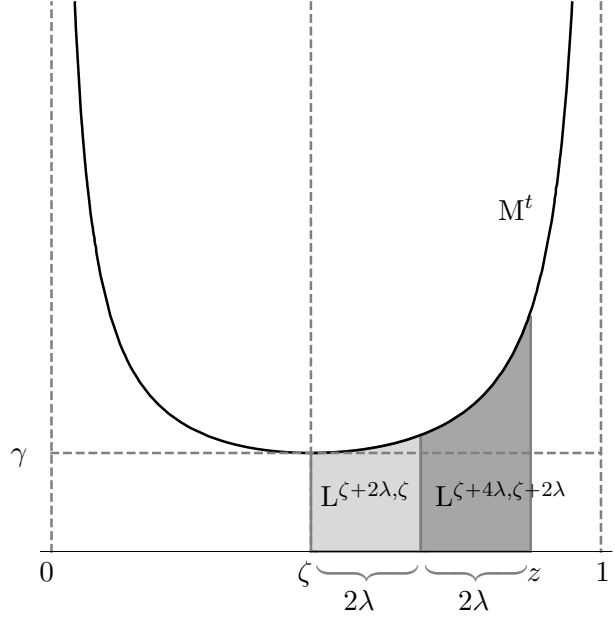
\begin{figure}
\centering
\begin{overpic}[scale=0.6]{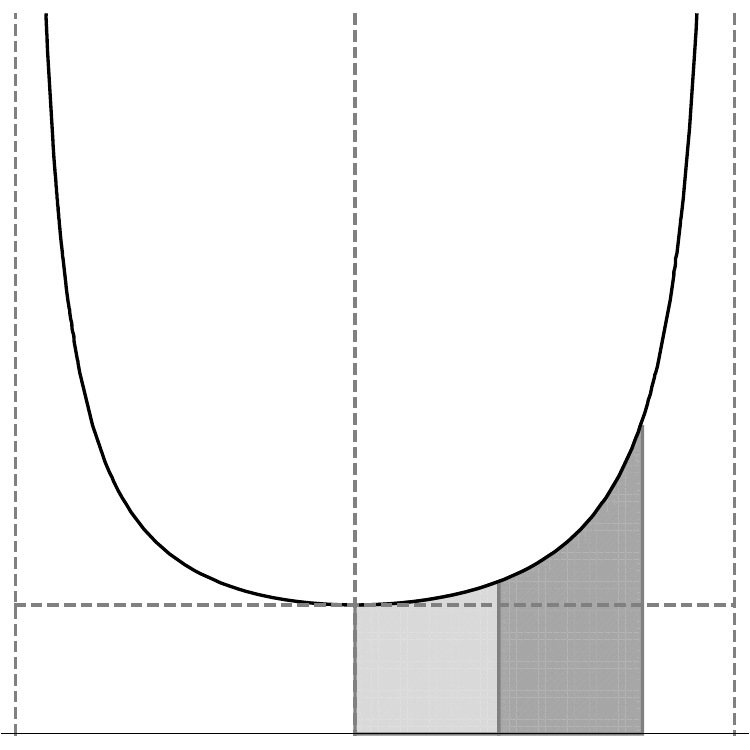}
\put (-5, 18){$\displaystyle \Shp$}
\put (45, -3){$\displaystyle \Min$}
\put (85, -3){$\displaystyle z$}
\put (80, 60){$\displaystyle \M^{t}$}
\put (0, -3){$0$}
\put (97, -3){$1$}
\put (48, 0){{\color{gray}$\displaystyle \underbrace{\hspace{0.55in}}$}}
\put (67, 0){{\color{gray}$\displaystyle \underbrace{\hspace{0.55in}}$}}
\put (53, -8){$2\lambda$}
\put (73, -8){$2\lambda$}
\put (49, 10){$\Lb^{\Min+2\lambda, \Min}$}
\put (69, 10){$\Lb^{\Min+4\lambda, \Min+2\lambda}$}
\end{overpic}
\vspace{0.2in}
\caption{Illustrates the last display in the proof of Proposition \ref{PGeoInStep}. The bound $\Lb^{\Min+2\lambda \Min}- \Lb^{\Min+4\lambda, \Min+2\lambda}$ measures the difference of the areas of the regions shaded in gray and darker gray. This quantity is of order $(z-\Min)^3$ because, for fixed $x, y > 0$, the curve $t \mapsto \M^t(x, y)$ (black) is approximately a parabola with the vertex at $(\Min(x, y), \Shp(x, y))$. 
}
\label{FPrGeoUB}
\end{figure}

\begin{rem}
\label{ROptConst}
We sketch how \eqref{EExitConst} can be obtained by modifying the preceding proof. Choosing $\epsilon > 0$ small ensures that $z = \Min + s/\curv \in (0, 1)$ in view of Lemma \ref{LShpMinBnd}(c). Let $p, q > 1$ with $1/p + 1/q = 1$. Using H\"{o}lder's inequality instead of the Cauchy-Schwarz in display \eqref{E-2} leads to the bound 
\begin{align}
\log \P\{\Eh^{z, \hor} > 0\} \le -\frac{1}{p}\Lb^{z-q\lambda, z-(p+q)\lambda} + \frac{1}{q}\Lb^{z, z-q\lambda}, \label{E-3}
\end{align}
where $\lambda \ge 0$ is a parameter to be chosen comparable to $z-\Min$. It can be seen from Lemma \ref{LLIBdEst} and some algebra that the right-hand side of \eqref{E-3} is at most 
\begin{align*}
\frac{\curv^3}{3}[3p^2\lambda^2(z-\Min) - p^3 \lambda^3 - 3p^2 q \lambda^2 (z-\Min) + p^3 q^2 \lambda^3] + \frac{C_0s^4}{(m+n)^{1/3}}. 
\end{align*}
Then calculus shows that the leading order term above attains its optimal value $-s^3/6$ when $p = 3/2, q = 3$ and $\lambda = (z-\Min)/3$. 
\end{rem}


\subsection{Proof of the main upper bound}

In preparation for the proof of Theorem \ref{TExitUB}, the next lemma states a simple geometric property of the exit points.  See Figure \ref{FExitStr} for an illustration.  The top right picture in particular exemplifies the case of equality in part (b).  The lemma uses the notation for the L-shaped down-right paths from Section \ref{SsEx}.  


\begin{lem}
\label{LExitStr}
Let $\pi$ be an up-right path from $(p, q)$ to $(m, n)$. Let $\nu$ be a down-right path  such that $\nu_1 \in \bbZ_{\ge p} \times \{n\}$ and $\nu_{\ell(\nu)} \in \{m\} \times \bbZ_{\ge q}$. Pick some $b \in [\ell(\nu)]$ and write $(i_0, j_0) = \nu_{b}$. 
Let $r \in \bbZ_{\ge 0}$. Then the following statements hold for each $\square \in \{+, -\}$. 
\begin{enumerate}[\normalfont (a)]
\item If $Z^\square_{\pi, \nu, i_0, j_0} > r$ then $b\square r \in [\ell(\nu)]$ and $Z^\square_{\pi, L, i_r, j_r} > 0$ where $(i_r, j_r) = \nu_{b\square r}$ and $L = L_{i_r, j_r}^{m, n}$. 
\item If $b\square r \in [\ell(\nu)]$ and $Z^\square_{\pi, L, i_r, j_r} > 0$ where $i_r, j_r$ and $L$ are as in {\rm(a}{\rm)} then $Z^\square_{\pi, \nu, i_0, j_0} \ge r$. 
\end{enumerate}
\end{lem}

\begin{proof}
By symmetry, it suffices to verify the claims for $\square = +$. 

To obtain (a), assume that $Z^+_{\pi, \nu, i_0, j_0} > r$. By definition \eqref{EExitPts}, $Z_{\pi, \nu} = b+s$ for some $s > r$ with $b+s \in [\ell(\nu)]$. Since $\nu$ is down-right, writing $(i_s, j_s) = \nu_{b+s}$, one has $i_r \le i_s \le m$ and $j_s \le j_r \le n$. Therefore, and also because $\pi$ is up-right and contains $\{(i_s, j_s), (m, n)\}$, $\pi$ also contains $(u, j_r)$ for some $u \in \bbZ \cap [i_s, m]$. If $u = i_r$ then necessarily $i_r = i_s$ and the vertical segment $\{(i_r, j): j \in \bbZ \cap [j_s, j_r]\} \subset \pi \cap \nu$ but this contradicts the present assumption that $\pi$ exits $\nu$ at $(i_s, j_s) \neq (i_r, j_r)$. Hence, $u > i_r$. Then, by definition \eqref{EExitPts}, $Z^+_{\pi, L, i_r, j_r} > 0$ as claimed. 

Now (b). Assume that $b+r \in [\ell(\nu)]$ and $Z^{+}_{\pi, L, i_r, j_r} > 0$. Then $(v, j_r) \in \pi$ for some $v \in \bbZ \cap [i_r+1, m]$ by \eqref{EExitPts}. Because $\pi$ is up-right path with $\{(p, q), (v, j_r)\} \subset \pi$ while $\nu$ is a down-right path with $(i_r, j_r) \in \nu$ and $\nu_{\ell(\nu)} \in \{m\} \times \bbZ$, from the inequalities $p \le i_r \le v \le m$, 
one concludes that $\nu_{b+t} = (i_t, j_t) \in \pi$ for some $t \ge r$ with $b+t \in [\ell(\nu)]$. Due to the strict inequality $v > i_r$ and that $\pi$ is up-right with $(v, j_r) \in \pi$, one has $\pi \cap (\{i_r\} \times \bbZ_{>j_r}) = \emptyset$. Hence, $\nu_k \not \in \pi$ for $k \in [b+r-1]$ because $\nu$ is down-right and $\nu_{b+r} = (i_r, j_r)$. Since also $(i_t, j_t) \in \pi \cap \nu$, it follows from definition \eqref{EExitDef} that $Z_{\pi, \nu} \ge b+r$, and therefore $Z^+_{\pi, \nu, i_0, j_0} \ge r$ by \eqref{EExitPts}.
\end{proof}

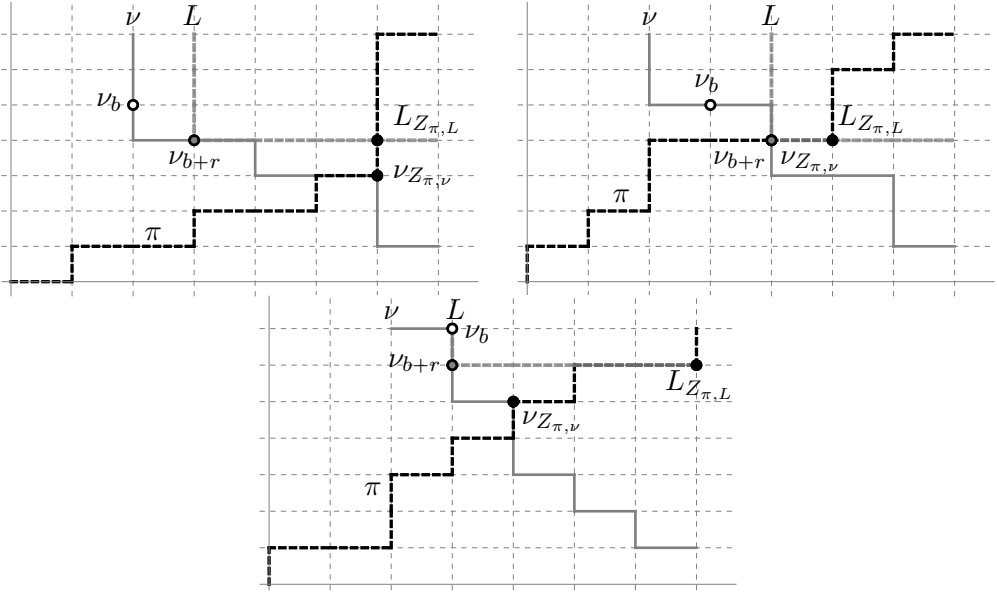
\begin{figure}
\centering
\begin{overpic}[scale=0.5]{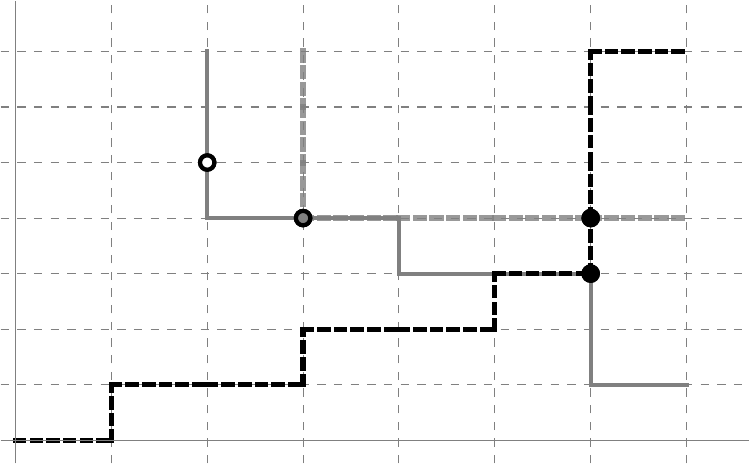}
\put (26, 57){$\nu$}
\put (38, 57){$L$}
\put (30, 12){$\pi$}
\put (20, 40){$\nu_b$}
\put (35, 28){$\nu_{b+r}$}
\put (82, 25){$\displaystyle \nu_{Z_{\pi, \nu}}$}
\put (82, 36){$\displaystyle L_{Z_{\pi, L}}$}
\end{overpic}
\quad 
\begin{overpic}[scale=0.5]{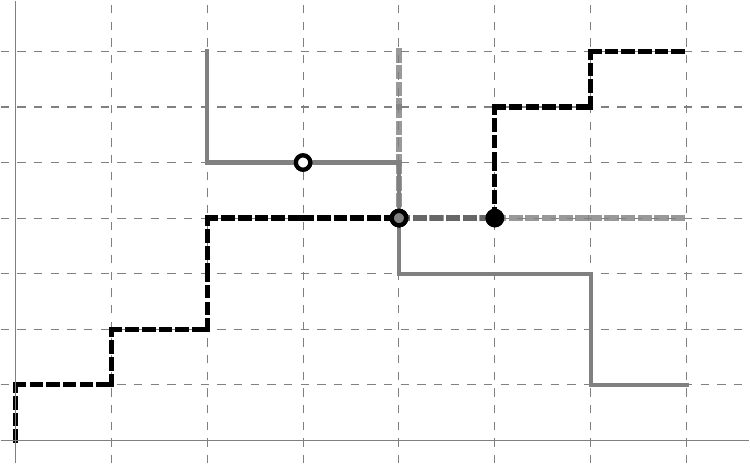}
\put (26, 57){$\nu$}
\put (51, 57){$L$}
\put (20, 20){$\pi$}
\put (37, 44){$\nu_b$}
\put (41, 28){$\nu_{b+r}$}
\put (55, 28){$\displaystyle \nu_{Z_{\pi, \nu}}$}
\put (67, 36){$\displaystyle L_{Z_{\pi, L}}$}
\end{overpic}
\quad 
\begin{overpic}[scale=0.5]{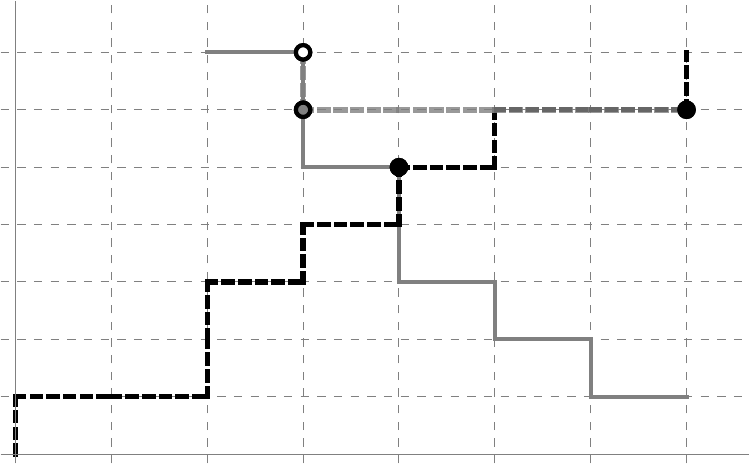}
\put (26, 57){$\nu$}
\put (39, 57){$L$}
\put (22, 20){$\pi$}
\put (43, 53){$\nu_b$}
\put (27, 47){$\nu_{b+r}$}
\put (55, 35){$\displaystyle \nu_{Z_{\pi, \nu}}$}
\put (85, 42){$\displaystyle L_{Z_{\pi, L}}$}
\end{overpic}
\caption{Illustrates Lemma \ref{LExitStr} for the case $\square = +$ in three situations. Paths $\pi$ (dashed black), $\nu$ (gray) and $L$ (dashed gray), and the vertices $\nu_b = (i_0, j_0)$ (white dot) and $\nu_{b+r} = (i_r, j_r)$ (gray dot) are shown. The last vertices $\nu_{Z_{\pi, \nu}}$ and $L_{Z_{\pi, L}}$ that $\pi$ visits on $\nu$ and $L$, respectively, are also marked (with black dots unless already indicated). Informally, part (a) of the lemma says that if $\nu_{Z_{\pi, \nu}}$ is strictly after $\nu_{b+r}$ along $\nu$ then $\pi$ exits $L$ from its horizontal segment. Part (b) says that if $\pi$ exits $L$ from its horizontal segment then $\nu_{Z_{\pi, \nu}}$ is strictly after or the same as $\nu_{b+r}$. On top left: $Z_{\pi, \nu, i_0, j_0}^+ = 6 > r = 2$ and $Z_{\pi, L, i_r, j_r}^+ = 4 > 0$. On top right: $Z_{\pi, \nu, i_0, j_0}^+ = 2 = r$ and $Z_{\pi, L, i_r, j_r}^+ = 1 > 0$. On the bottom: $Z_{\pi, \nu, i_0, j_0}^+ = 3 > r = 1$ and $Z_{\pi, L, i_r, j_r}^+ = 4 > 0$. These statements are consistent with the claims of the lemma. In particular, the top right picture demonstrates the case of equality in part (b). }
\label{FExitStr}
\end{figure}

Next a monotonicity property for the geodesic exit points defined at \eqref{EEx}. 
\begin{lem}
\label{LExMon2}
The exit points in \eqref{EEx} satisfy the following properties for each $(m, n) \in V_{\nu}$. 
\begin{enumerate}[\normalfont (a)]
\item $\Eh^{w, z, +}_{\nu, i_0, j_0}(m, n)$ is nonincreasing in $w$ and $z$. 
\item $\Eh^{w, z, -}_{\nu, i_0, j_0}(m, n)$ is nondecreasing in $w$ and $z$. 
\end{enumerate}
\end{lem}
\begin{proof}
Note from \eqref{ECplBlkBd} that $\wb^{w}(i,  0)$ is decreasing in $w$,  and $\wb^{z}(0, j)$ is increasing in $z$ for each $i, j \in \bbZ_{>0}$.  The claimed monotonicities are then special cases of Lemma \ref{LExMon}. 
\end{proof}

The following distributional identity connects the exit points from L-shaped paths and to those from the axes.  A statement to the same effect appeared previously in \cite[Lemma A.2]{sepp-cgm-18}.  The lemma can also be derived as a corollary of Proposition \ref{PExitDistId}. 
\begin{lem}
\label{LExitDistId}
Let $m, n, p, q \in \bbZ_{\ge 0}$, 
$L = L_{p, q}^{m+p, n+q}$ and $z \in (0, 1)$. Then 
\begin{align*}
\Eh^{z, +}_{L, p, q}(m+p, n+q) \stackrel{\text{\rm{dist.}}}{=} \Eh^{z, \hor}(m, n) \quad \text{ and } \quad \Eh^{z, -}_{L, p, q}(m+p, n+q) \stackrel{\text{\rm{dist.}}}{=} \Eh^{z, \ver}(m, n). 
\end{align*}
\end{lem}

To establish Theorem \ref{TExitUB}, one needs one more lemma comparing the values of the minimizer \eqref{EMin} at different vertices. 

\begin{lem}
\label{LMinEst}
Let $x, y \in \bbR_{>0}$ and $\delta \in \bbR_{\ge 0}$. Then  
\begin{align*}
\Min(x+\delta, y)-\Min(x, y) &= \frac{\delta (1-\Min(x, y))}{(\sqrt{x+\delta}+\sqrt{y})(\sqrt{x+\delta} + \sqrt{x})},  \\ 
\Min(x, y+\delta)-\Min(x, y) &= -\frac{\delta\Min(x, y)}{(\sqrt{x}+\sqrt{y+\delta})(\sqrt{y} + \sqrt{y+\delta})}. 
\end{align*}
\end{lem}
\begin{proof}
These can be readily verified from \eqref{EShp} and \eqref{EMin}. 
\end{proof}

\begin{proof}[Proof of Theorem \ref{TExitUB}]

Let $\epsilon_0 = \epsilon_0(\delta)$ and $c = c(\delta)$ denote positive constants to be chosen sufficiently small below. Take $N_0 = N_0(\delta) \ge 1/c$. Let $(m, n) \in S_\delta \cap \bbZ_{\ge N_0}^2$, $(m+n)^{-2/3} \le s \le c(m+n)^{1/3}$ and $k = \lf s(m+n)^{2/3}\rf$. Note that the interval for $s$ is nonempty and $k \ge 1$.  The additional restriction on $s$ due to the upper bound will be lifted at the end of our argument.  

Our aim is to verify the bound in (a).  If $w > 0$ and $z < 1$ such that $\min\{w, z\} \ge \Min(m, n)-\epsilon_0 s (m+n)^{-1/3}$ then,  by Lemma \ref{LShpMinBnd}(b) and since $(m, n) \in S_\delta$,    
\begin{align*}
1 > \min\{w, z\} \ge \Min(m, n)-\epsilon_0 c > 0
\end{align*}
provided that $\epsilon_0 c$ is sufficiently small.  Then on account of the monotonicity of the exit points recorded in Lemma \ref{LExMon2},  it suffices to obtain (a) for the exit point $\Eh_{\nu, i_0, j_0}^{\min\{w, z\}, \min\{w, z\}, +}$,  which is well-defined since $\min\{w, z\} \in (0, 1)$.  This reduces (a) to the case $w = z$.  To treat this case,  pick $z \in (0, 1)$ with $z \ge \Min(m, n) - \epsilon_0 s(m+n)^{-1/3}$.  

Let $(i_0, j_0) \in \bbZ_{\ge 0}^2$ and consider a down-right path $\nu$ on $\bbZ^2_{\ge 0}$ such that $(m+i_0, n+j_0) \in V_{\nu}$ and $\nu_b = (i_0, j_0)$ for some $b \in [\ell(\nu)]$. Assuming that $b+k \le \ell(\nu)$ for now, let $(i_k, j_k) = \nu_{b+k}$. Then $j_k \le j_0 \le n+j_0$. Assuming further that $i_k \le m+i_0$, let $L_k = L_{i_k, j_k}^{m+i_0, n+j_0}$ denote the L-shaped down-right path (see Section \ref{SsEx}) from $(i_k, n+j_0)$ to $(m+i_0, j_k)$ passing through the vertex $(i_k, j_k)$. 
See Figure \ref{FPrExitUB}. 

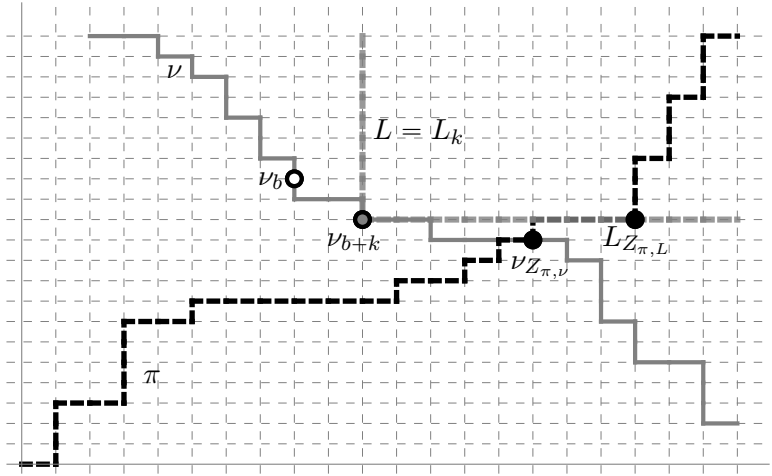
\begin{figure}
\centering
\begin{overpic}[scale=0.8]{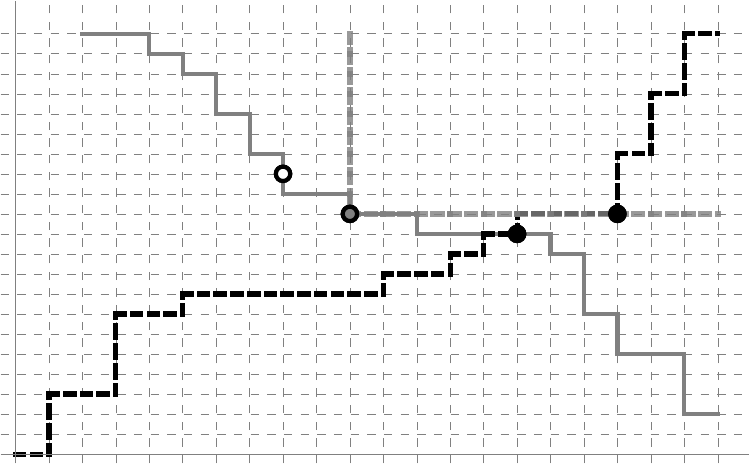}
\put (21, 52){$\nu$}
\put (18, 12){$\pi$}
\put (48, 44){$L = L_k$}
\put(33, 38){$\nu_b$}
\put(42, 30){$\nu_{b+k}$}
\put(66, 27){$\nu_{Z_{\pi, \nu}}$}
\put(78, 30){$L_{Z_{\pi, L}}$}
\end{overpic}
\caption{Illustrates the setup prior to display \eqref{E-22}. Path $\nu$ (gray), base vertex $(i_0, j_0) = \nu_b$ (white dot), the geodesic $\pi = \pi_{0, 0}^{z}(m+i_0, n+j_0)$ (black, dashed), vertex $(i_k, j_k) = \nu_{b+k}$ (gray dot) and path $L = L_k = L_{i_k, j_k}^{m+i_k, n+j_k}$ (gray, dashed) are shown. The exit points $\nu_{Z_{\pi, \nu}}$ and $L_{Z_{\pi, L}}$ (black dots) are also marked.}
\label{FPrExitUB}
\end{figure}

In the next display, the inequality comes from an application of Lemma \ref{LExitStr}(a) with the geodesic $\pi = \pi_{0, 0}^z(m+i_0, n+j_0)$. The subsequent equality holds by virtue of Lemma \ref{LExitDistId}. 
\begin{align}
\label{E-22}
\begin{split}
\P\{\Eh^{z, +}_{\nu, i_0, j_0}(m+i_0, n+j_0) > s(m+n)^{2/3}\}  &= \P\{\Eh^{z, +}_{\nu, i_0, j_0}(m+i_0, n+j_0) > k\} \\
&\le \P\{\Eh^{z, +}_{L_k, i_k, j_k}(m+i_0, n+j_0) > 0\} \\
&= \P\{\Eh^{z, \hor}(m+i_0-i_k, n+j_0-j_k) > 0\}. 
\end{split}
\end{align}

Because $k \le c(m+n)$ and $(m, n) \in S_\delta$,  choosing $c < 1$ sufficiently small (depending on $\delta$) ensures that $(m-k, n+k) \in S_{\delta/2}$.  Then $(m+i_0-i_k, n)$,  $(m, n+j_0-j_k) \in S_{\delta/2}$ as well since both $i_k-i_0,  j_0-j_k \in [0, k]$.  Combining these with Lemmas \ref{LMinEst} and \ref{LShpMinBnd}(b) (applying the latter to bound the factors involving $\zeta$) 
leads to the bound 
\begin{align}
\label{E-23}
\begin{split}
&\Min(m + i_0 - i_k, n+j_0-j_k)-\Min(m, n) \\
&= \Min(m + i_0 - i_k, n+j_0-j_k)-\Min(m, n+j_0-j_k) \\
&+ \Min(m, n+j_0-j_k)-\Min(m, n) \\
&= - \frac{(i_k-i_0)(1-\Min(m, n+j_0-j_k))}{(\sqrt{m+i_0-i_k} + \sqrt{n+j_0-j_k})(\sqrt{m+i_0-i_k}+\sqrt{m})} \\
&-  \frac{(j_0-j_k)\Min(m, n)}{(\sqrt{m} + \sqrt{n+j_0-j_k})(\sqrt{n} + \sqrt{n+j_0-j_k})} \\
&\le -\frac{a_0 (i_k-i_0 + j_0-j_k)}{m+n} = -\frac{a_0 k}{(m+n)} \le -\frac{1}{2}a_0 s (m+n)^{-1/3}
\end{split}
\end{align} 
for some constant $a_0 = a_0(\delta) > 0$. For the last step, recall that $k = \lf s(m+n)^{2/3} \rf \ge 1$. Now choosing $\epsilon_0 \le a_0/4$ gives  
\begin{align*}
z-\Min(m + i_0 - i_k, n+j_0-j_k) \ge \frac{1}{4}a_0 s(m+n)^{-1/3}
\end{align*}
in view of \eqref{E-23} and the assumption on $z$.  Note also that $(m+i_0-i_k,  n+j_0-j_k) \in S_{\delta/2}$ since $(m-k, n+k) \in S_{\delta/2}$ and $(m, n) \in S_\delta$. Therefore, appealing to Proposition \ref{PGeoInStep}(a) for the last probability in \eqref{E-22}, one obtains that 
\begin{align}
\label{E-24}
\P\{\Eh^{z, +}_{\nu, i_0, j_0}(m+i_0, n+j_0) > s(m+n)^{2/3}\} \le \exp\{-c_0 s^3\} 
\end{align}
for $s \in [(m+n)^{-2/3}, c(m+n)^{1/3}]$ and some constant $c_0 = c_0(\delta) > 0$.  The bound in \eqref{E-24} also holds trivially when $i_k > m+i_0$ or $b+k > \ell(\nu)$ because in both cases the event on the left-hand side is empty. 

Assume now that $s > c(m+n)^{1/3}$. Then, by \eqref{E-24}, 
\begin{align}
\label{E-25}
\begin{split}
\P\{\Eh^{z, +}_{\nu, i_0, j_0}(m+i_0, n+j_0) > s(m+n)^{2/3}\} &\le \P\{\Eh^{z, +}_{\nu, i_0, j_0}(m+i_0, n+j_0) > c(m+n)\}\\
&\le \exp\{-c_0 c^3(m+n)\}. 
\end{split}
\end{align} 

Combining \eqref{E-24} and \eqref{E-25}, and redefining $c_0$ suitably (for example,  as $c_0c^3$) completes the proof of part (a).  Part (b) is treated similarly.  
\end{proof}

\subsection{Proof of the lower bound for the first step probabilities}

The main idea behind the proof of Proposition \ref{PGeoInStepLB} is 
a suitable change of the rates on the boundaries. 
This type of change-of-measure argument originated in \cite{bala-sepp-aom} and has been employed recently in \cite{sepp-cgm-18, shen-sepp-19}. 
Changing the measure 
brings in weights with \emph{mixed} boundary rates. In our proof, these weights can be defined from the i.i.d.\ $\Exp(1)$-distributed weights $\{\eta(i, j): i, j \in \bbZ_{\ge 0}\}$ by 
\begin{align}
\label{EMixEnv}
\begin{split}
\wt{\w}^{w, z, k}(i, j) = \eta(i, j) \cdot \bigg(\one_{\{i, j > 0\}} &+ \one_{\{j = 0\}}\bigg\{\frac{\one_{\{i > k\}}}{w} + \frac{\one_{\{0 < i \le k\}}}{z}\bigg\} \\
&+ \one_{\{i = 0\}}\bigg\{\frac{\one_{\{j > k\}}}{1-w} + \frac{\one_{\{0 < j \le k\}}}{1-z}\bigg\}\bigg)
\end{split}
\end{align}
for $i, j, k \in \bbZ_{\ge 0}$ and $w, z \in (0, 1)$. 
For clarity, let us indicate the last-passage times and exit points defined from the weights $\wt{\w}^{w, z, k}$ also with the decoration $\wt{\ }$ and superscript $w, z, k$ as in $\wt{\Ev}^{w, z, k, \hor}$ for example.  

\begin{proof}[Proof of Proposition \ref{PGeoInStepLB}]
We prove only (a) leaving out the similar argument for (b). 

Let $(m, n) \in S_\delta \cap \bbZ_{>0}^2$. By Lemma \ref{LShpMinBnd}(b), $\Min = \Min(m, n) \in (\epsilon_0, 1-\epsilon_0)$ for some constant $\epsilon_0 = \epsilon_0(\delta) > 0$. Choose $N_0 = N_0(\delta, \epsilon) > 0$ sufficiently large such that $\epsilon N_0^{-1/3} \le \epsilon_0/8$. Assuming $m, n \ge N_0$ from now on, pick $z \in [\Min+\epsilon(m+n)^{-1/3}, \Min+\epsilon_0/8]$. Then put $\lambda = z-\Min \in [\epsilon (m+n)^{-1/3}, \epsilon_0/8]$ and $w = \Min - \lambda \in [7\epsilon_0/8, 1-\epsilon_0)$. Hence, $z-w = 2\lambda$. Introduce another constant $r = r(\delta, \epsilon) \ge 1$ to be specified below and let $k = \lf r \lambda (m+n)\rf$. Since $k \ge \epsilon N_0^{2/3}-1$, choosing $N_0$ large enough ensures that $k \ge 1$. 

In the development of the bound below, the first step changes the measure such that the underlying weights $\wh{\w}^w$ are replaced with $\wt{\w}^{w, z, k}$. Note that the associated Radon-Nikodym derivative is given by 
\begin{align*}
&\bigg(\frac{w}{z}\bigg)^k \prod_{i \in [k]} \exp\{2\lambda\wt{\w}^{w, z, k}(i, 0)\} \cdot \bigg(\frac{1-w}{1-z}\bigg)^{k} \prod_{j \in [k]} \exp\{-2\lambda \wt{\w}^{w, z, k}(0, j)\} \\
&= \bigg(\frac{w}{z}\bigg)^k \exp\{2\lambda \Gb^{z}(k, 0)\} \cdot \bigg(\frac{1-w}{1-z}\bigg)^{k} \exp\{-2\lambda \Gb^{z}(0, k)\}. 
\end{align*}
The second step below applies the Cauchy-Schwarz inequality. For the last equality, use independence and recall that $4\lambda \le \epsilon_0/2 < \epsilon_0 \le \Min < z$. 
\begin{align}
\begin{split}
&\P\{\Ev^{w, \hor}(m, n) > 0\}^2 \\
&= \bigg(\frac{w}{z}\bigg)^{2k} \bigg(\frac{1-w}{1-z}\bigg)^{2k} \E\left[\one\{\wt{\Ev}^{w, z, k, \hor}(m, n) > 0\}\exp\{2\lambda[\Gb^{z}(k, 0)-\Gb^{z}(0, k)]\}\right]^2 \\
&\le \bigg(\frac{w}{z}\bigg)^{2k} \bigg(\frac{1-w}{1-z}\bigg)^{2k} \E\left[\exp\{4\lambda[\Gb^{z}(k, 0)-\Gb^{z}(0, k)]\}\right]\P\{\wt{\Ev}^{w, z, k, \hor}(m, n) > 0\} \\
&= \bigg(\frac{w}{z}\bigg)^{2k} \bigg(\frac{1-w}{1-z}\bigg)^{2k} \bigg(\frac{z}{z - 4\lambda}\bigg)^k \bigg(\frac{1-z}{1-z +4\lambda}\bigg)^k\P\{\wt{\Ev}^{w, z, k, \hor}(m, n) > 0\}. 
\end{split}
\label{E-4}
\end{align}

By virtue of Proposition \ref{PGeoInStep}(b) and the choices of $w$ and $\lambda$, for some constant $c_0 = c_0(\delta) > 0$, the first probability in \eqref{E-4} obeys the bound 
\begin{align}
\label{E-5}
\begin{split}
\P\{\Ev^{w, \hor}(m, n) > 0\} &= 1-\P\{\Eh^{w, \ver}(m, n) > 0\} \\
&\ge 1-\exp\{-c_0(m+n)\lambda^3\} \\
&\ge 1-\exp\{-c_0\epsilon^3\} = p_0 > 0. 
\end{split}
\end{align}

Next bound the logarithm of the product of the first four factors on the last line of \eqref{E-4} as follows. For some constant $C_1 = C_1(\delta) > 0$, 
\begin{align}
\begin{split}
&2k \log\bigg(1 - \frac{2\lambda}{z}\bigg) - k \log\bigg(1 - \frac{4\lambda}{z}\bigg) + 2k \log\bigg(1 + \frac{2\lambda}{1-z}\bigg) - k \log \bigg(1+\frac{4\lambda}{1-z}\bigg) \\
&\le \frac{4k\lambda^2}{z^2} + \frac{4k\lambda^2}{(1-z)^2} + C_1k \lambda^3\\ 
&\le C_1 r \lambda^3 (m+n).
\end{split}
\label{E-6}
\end{align}
For the first inequality in \eqref{E-6}, apply the estimate $|\log (1+t)-t + t^2/2| \le |t|^3$ for $t \in [-1/2, 1/2]$ recalling that $\lambda \le \epsilon_0/8$ and $z \in (\epsilon_0, 1-\epsilon_0/2)$. The second inequality inserts the definition of $k$, 
and uses the bounds on $z$ once more. 

Now turn to the last probability in \eqref{E-4}.  Recall that the weights in \eqref{ECplBlkBd} and \eqref{EMixEnv} are all coupled through the $\eta$-variables.  Let $E = E_{m, n}^{w, z, k}$ denote the event on which the inequalities 
\begin{align}
\label{E-8}
\begin{split}
&\wt{\Ev}^{w, z, k, \hor}(m, n) > 0 \text{ and } \\
&\max \{\Ev^{w, \hor}(m, n), \Eh^{z, \hor}(m, n), \Eh^{z, \ver}(m, n)\} \le k, 
\end{split}
\end{align}
all hold.  Also, write $F = F_{m, n}^{w, z, k}$ for the event of the second inequality in \eqref{E-8}. By virtue of Corollary \ref{CExitUB}, the probability of the complementary event is at most 
\begin{align}
\label{E-9}
\begin{split}
\P\{F^c\} \le 3\exp\{-c_1 r^3 \lambda^3 (m+n)\}
\end{split}
\end{align}
for some constant $c_1 = c_1(\delta) > 0$ provided that $N_0$ and $r$ are sufficiently large. From $r$ specifically, it would suffice to require $\min \{r\epsilon, r\wt{\epsilon}_0\} \ge 1$ where $\wt{\epsilon}_0 = \wt{\epsilon}_0(\delta)$ refers to the constant denoted with $\epsilon_0$ in the statement of the corollary. (Then $s = r\lambda (m+n)^{1/3} \ge r\epsilon \ge 1 \ge (m+n)^{-2/3}$ and $z-\Min = \Min-w = \lambda \le \wt{\epsilon}_0 r \lambda = \wt{\epsilon_0} s (m+n)^{-1/3}$ as needed). Let us next claim the containment
\begin{align}
E \subset \{\Ev^{z, \hor}(m, n) > 0\} \label{E-10}
\end{align}
whose verification is deferred to the end of the proof.  Combining \eqref{E-9} and \eqref{E-10} via a union bound yields 
\begin{align}
\label{E-7}
\begin{split}
\P\{\wt{\Ev}^{w, z, k, \hor}(m, n) > 0\} &\le \P\{\Ev^{z, \hor}(m, n) > 0\} + 3\exp\{-c_1 r^3 \lambda^3 (m+n)\}. 
\end{split}
\end{align}

Putting together \eqref{E-4},  \eqref{E-5}, \eqref{E-6} and \eqref{E-7} results in 
\begin{align*}
\begin{split}
\P\{\Ev^{z, \hor}(m, n) > 0\} &\ge p_0^2 \exp\{-C_1 r \lambda^3 (m+n)\}-3\exp\{-c_1 r^3 \lambda^3 (m+n)\} \\ 
&\ge \frac{p_0^2}{2} \exp\{-C_1 r \lambda^3 (m+n)\} \ge \exp\{-2C_1 r \lambda^3 (m+n)\}. 
\end{split}
\end{align*}
The inequalities on the last line hold for sufficiently large $r$ since $\lambda^3 (m+n) \ge \epsilon^3$. Then the conclusion of (a) holds with $C_0 = 2C_1r$. 

It remains to verify \eqref{E-10} to complete the proof of (a). Restrict to the event $E$ below. Then 
$0 < \wt{\Ev}^{w, z, k, \hor}(m, n) \le \Ev^{w, \hor}(m, n) \le k$ where the middle inequality comes from 
Lemma \ref{LExMon} (the first monotonicity in part (a)) and since $w \le z$. Consequently, 
\begin{align}
\label{E-12}
\begin{split}
\max_{i \in [k]} \{\Gb^z(i, 0) + \G_{i, 1}(m, n)\} &= \wt{\G}^{w, z, k}(m, n) = \wt{\G}^{w, z, k}_{1, 0}(m, n) \ge \wt{\G}^{w, z, k}_{0, 1}(m, n) \\
&\ge \max_{j \in [k]}\{\Gb^z(0, j) + \G_{1, j}(m, n)\}. 
\end{split}
\end{align}
Since also $\max \{\Eh^{z, \hor}(m, n), \Ev^{z, \ver}(m, n)\} \le k$, one concludes from \eqref{E-12} that $\Gb_{1, 0}^z(m, n) \ge \Gb_{0, 1}^z(m, n)$. Hence, \eqref{E-10} holds. 
%
\end{proof}

\subsection{Proof of the main lower bound}

Our proof of Theorem \ref{TExitLB} is a suitable modification of the proof of Theorem \ref{TExitUB}. 

\begin{proof}[Proof of Theorem \ref{TExitLB}]
Let $c_0 = c_0(\delta,  K) > 0$ and $N_0 = N_0(\delta, \epsilon, K) > 0$ denote constants to be chosen below.  Let $(m, n) \in S_\delta \cap \bbZ_{\ge N_0}^2$ and $s \in [\epsilon,  c_0 (m+n)^{1/3}]$ taking $N_0 \ge (\epsilon/c_0)^3$ to ensure that the preceding interval is nonempty.  Let $k = \lf s (m+n)^{2/3} \rf \le c_0 (m+n)$.  After decreasing $c_0$ and increasing $N_0$ if necessary,  one has $1 \le k < m$.  

If $w > 0$ and $z < 1$ with $\max \{w, z\} \le \Min(m, n) + K s (m+n)^{-1/3}$ then it follows from Lemma \ref{LShpMinBnd}(b) that 
\begin{align*}
0 < \max\{w, z\} \le \Min(m, n) + Kc_0 < 1
\end{align*}
for sufficiently small $c_0$.  Then,  appealing to the monotonicity in Lemma \ref{LExMon2},  it suffices to prove the bound in (a) for the exit point $\Eh^{\max\{w, z\}, \max\{w, z\}, +}_{\nu, i_0, j_0}$,  which makes sense according to definition  \eqref{EEx} since $\max\{w, z\} \in (0, 1)$.  Consequently,  the bound in (a) reduces to the case $w = z$.  Hence,  pick $z \in (0, 1)$ such that $z \le \Min(m, n) + Ks(m+n)^{-1/3}$.  Via another appeal to Lemma \ref{LExMon2},  it suffices to restrict to the case $z \ge \Min(m, n)$ in proving (a). 

Let $(i_0, j_0) \in \bbZ_{\ge 0}^2$, and $\nu$ be a down-right path on $\bbZ_{\ge 0}^2$ such that $(m+i_0, n+j_0) \in V_{\nu}$ and $\nu_b = (i_0, j_0)$ for some $b \in [\ell(\nu)]$. Then $k+b < \ell(\nu)$ since the path $\nu$ takes at least $m$ steps from $\nu_b \in \{i_0\} \times \bbZ$ to $\nu_{\ell(\nu)} \in \bbZ_{\ge m+i_0} \times \bbZ$. Consider the L-shaped path $L_{k} = L_{i_{k}, j_{k}}^{m+i_0, n+j_0}$ where $(i_k, j_k)$ denotes the coordinates of $\nu_{b+k+1}$. (The picture is the same as in Figure \ref{FPrExitUB} except that the vertex $\nu_{b+k+1}$ is used now instead of $\nu_{b+k}$). 

Now appeal to Lemma \ref{LExitStr}(b) and then Lemma \ref{LExitDistId} (in the last step) to obtain 
\begin{align}
\label{E-26}
\begin{split}
\P\{\Eh^{z, +}_{\nu, i_0, j_0}(m+i_0, n+j_0) > s(m+n)^{2/3}\}  &= \P\{\Eh^{z, +}_{\nu, i_0, j_0}(m+i_0, n+j_0) > k\} \\
&= \P\{\Eh^{z, +}_{\nu, i_0, j_0}(m+i_0, n+j_0) \ge k+1\} \\
&\ge \P\{\Eh^{z, +}_{L_{k}, i_{k}, j_{k}}(m+i_0, n+j_0) > 0\} \\
&= \P\{\Eh^{z, \hor}(m+i_0-i_{k}, n+j_0-j_{k}) > 0\}. 
\end{split}
\end{align}
The computation in \eqref{E-23} still gives  
\begin{align}
\label{E-23a}
\begin{split}
\Min(m+i_0-i_k, n+j_0-j_k)-\Min(m, n) \le -a_0 s (m+n)^{-1/3}
\end{split}
\end{align}
for some constant $a_0 = a_0(\delta) > 0$.  In the same vein,  one also obtains the lower bound 
\begin{align}
\label{E-23b}
\begin{split}
&\Min(m+i_0-i_k, n+j_0-j_k)-\Min(m, n) \\
&\ge -\frac{A_0(i_k-i_0+j_0-j_k)}{2(m+n)} = -\frac{A_0 (k+1)}{2(m+n)} \ge -A_0 s (m+n)^{-1/3}
\end{split}
\end{align}
for some constant $A_0 = A_0(\delta) > 0$.  Due to \eqref{E-23a},  and the assumptions that $z \ge \Min(m, n)$ and $s \ge \epsilon$,  one has
\begin{align}
\label{E-23c}
\begin{split}
z-\Min(m+i_0-i_k, n+j_0-j_k) &\ge a_0 s (m+n)^{-1/3} \ge a_0 \epsilon (m+n)^{-1/3} 
\end{split}
\end{align}
Also,  by \eqref{E-23b} and the assumptions that $z \le \Min(m, n)+K s (m+n)^{-1/3}$ and $s \le c_0(m+n)^{1/3}$,  
\begin{align}
\label{E-23d}
z-\Min(m+i_0-i_k, n+j_0-j_k) \le (K + A_0) s (m+n)^{-1/3} \le (K+A_0) c_0. 
\end{align}
With $c_0$ chosen sufficiently small,  $(m+i_0-i_k,  n+j_0-j_k) \in S_{\delta/2}$ and the last expression in \eqref{E-23d} can be made $(K+A_0) c_0 \le \epsilon_0(\delta/2)$ where $\epsilon_0$ refers to the constant in Proposition \ref{PGeoInStepLB}.  Then Proposition \ref{PGeoInStepLB}(a) applied to the last probability in \eqref{E-26} yields 
\begin{align}
\label{E-23e}
\begin{split}
&\P\{\Eh^{z, +}_{\nu, i_0, j_0}(m+i_0, n+j_0) > s(m+n)^{2/3}\}  \\
&\ge \exp\{-C_1 (m+i_0-i_k+n+j_0-j_k)(z-\Min(m+i_0-i_k,  n+j_0-j_k))^3\} \\
&\ge \exp\{-C_0 s^3\}
\end{split}
\end{align}
for some constants $C_1 = C_1(\delta, \epsilon) > 0$ and $C_0 = C_0(\delta, \epsilon,  K) > 0$ provided that $N_0$ is sufficiently large.  The final inequality in \eqref{E-23e} relies on the first bound in \eqref{E-23d}.  This finishes the proof of (a),  and the proof of (b) is completely analogous. 
\end{proof}

\subsection{Proof of the exit point upper bounds for bulk geodesics}
\label{SsPfPathFluc}

We conclude this section with the proof of Theorem \ref{TPathFluc}. 

\begin{proof}[Proof of Theorem \ref{TPathFluc}]
Let $c, \epsilon_0, s_0$ and $N_0$ denote positive constants depending only on $\delta$ and to be chosen in the course of the proof. Let $(m, n) \in S_\delta \cap \bbZ_{\ge N_0}^2$, $s \in [s_0, c(m+n)^{1/3}]$ and $k = \lf s(m+n)^{2/3}\rf$ where $N_0$ is taken sufficiently large to ensure the existence of $s$. Pick $(i_0, j_0) \in \bbZ^2_{\ge 0}$ subject to \eqref{EBaseCond}, and a down-right path $\nu$ on $\bbZ^2_{\ge 0}$ such that $(i_0, j_0) \in \nu$ and $(m+i_0, n+j_0) \in V_\nu$. Write $b \in [\ell(\nu)]$ for the unique index for which $\nu_b = (i_0, j_0)$, and $\wt{\nu} = \nu + (1, 1)$ for the down-right path obtained by shifting the vertices in $\nu$ by $(1, 1)$. 

Abbreviate $M = m+i_0+1$ and $N = n+j_0+1$. Condition \eqref{EBaseCond} and the upper bound on $s$ imply that $|\Min(M, N) - \Min(m, n)| \le c\epsilon_0$. Combining this with the fact that $(m, n) \in S_\delta$, recalling definition \eqref{EMin} and appealing to Lemma \ref{LShpMinBnd}(b), one concludes that $(M, N) \in S_{\varepsilon}$ for some constant $\varepsilon = \varepsilon(\delta) > 0$ provided that $c$ is sufficiently small given $\epsilon_0$.  (Importantly,  $\varepsilon$ does not depend on the choice of $c$,  which will be used in a moment). 

Let $E = \{Z_{\pi, \wt{\nu}} > b+k\}$ where $\pi$ denotes the geodesic $\pi_{1, 1}(M, N)$. Let 
\begin{align*}
z = \Min(M, N) - \epsilon_0 s(M+N)^{-1/3} \ge \Min(M, N)-c\epsilon_0. 
\end{align*}
Because $(M, N) \in S_{\varepsilon}$, by Lemma \ref{LShpMinBnd}(b) and after decreasing $c$ if necessary depending on $\varepsilon$ and $\epsilon_0$,  one guarantees that $z > 0$.  Then $z \in (0, 1)$ is a legitimate parameter below. 

The first step in the next displayed argument is due to the choice of $k$. The second equality holds on the account of \eqref{EBlkDistId}. To justify the final step, observe that if $E$ occurs and the geodesic $\pi_{0, 0}^z(M, N)$ contains $(1, 0)$ then necessarily $\Eh^{z, +}_{\wt{\nu}, i_0+1, j_0+1}(M, N) > k$ by \eqref{EEx}. See Figure \ref{FPathFluc}. 
\begin{align}
\label{E-33}
\begin{split}
&\P\{\Eh^{1, 0, +}_{\nu, i_0, j_0}(m+i_0, n+j_0) > s(m+n)^{2/3}\} = \P\{\Eh^{1, 0, +}_{\nu, i_0, j_0}(m+i_0, n+j_0) > k\} = \P\{E\} \\
&\le \P\{\Eh^{z, \ver}(M, N) > 0\} +\P\{\Eh^{z, +}_{\wt{\nu}, i_0+1, j_0+1}(M, N) > k\}. 
\end{split}
\end{align}

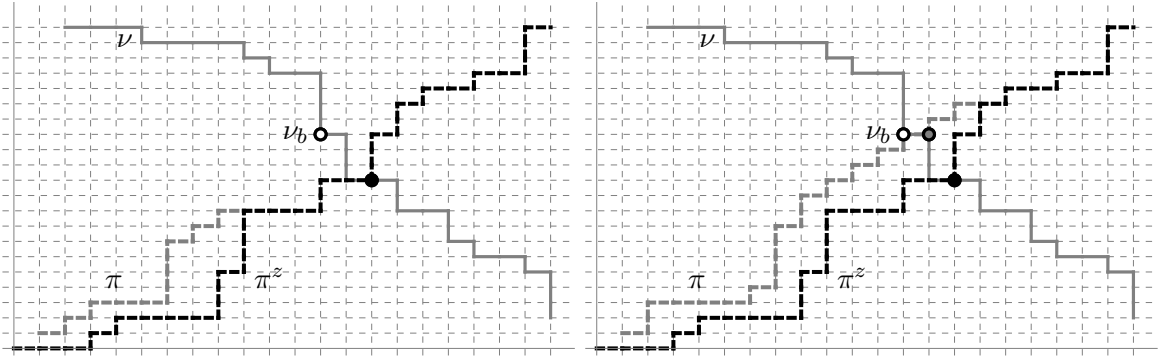
\begin{figure}
\centering
\begin{overpic}[scale=0.6]{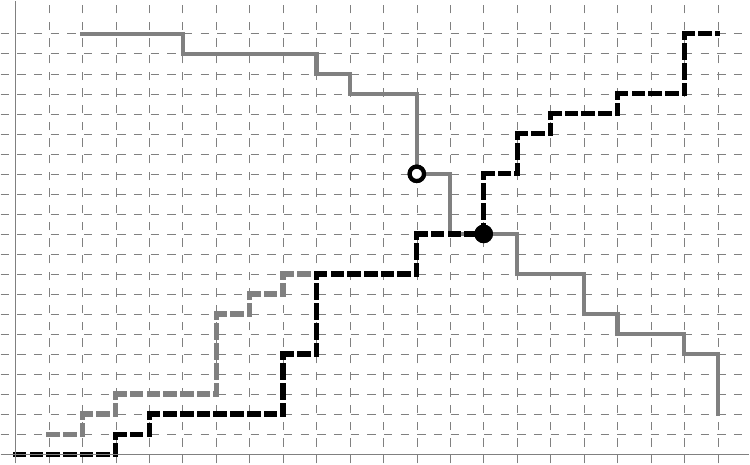}
\put (20, 52){$\wt{\nu}$}
\put (18, 12){$\pi$}
\put (44, 12){$\pi^z$}
\put(49, 38){$\wt{\nu}_b$}
\end{overpic}
\begin{overpic}[scale=0.6]{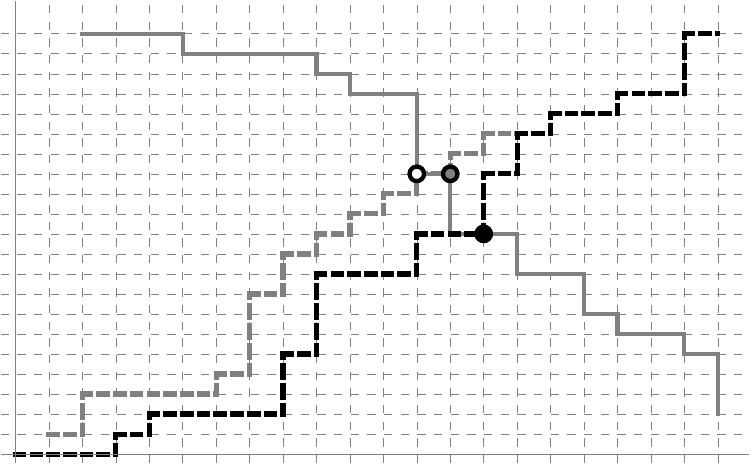}
\put (20, 52){$\wt{\nu}$}
\put (18, 12){$\pi$}
\put (44, 12){$\pi^z$}
\put(49, 38){$\wt{\nu}_b$}
\end{overpic}
\caption{Illustrates the justification for the inequality in \eqref{E-33} in two situations. The down-right path $\wt{\nu}$ (gray) and  the geodesics $\pi^z = \pi^z_{0, 0}(M, N)$ (dashed black) and $\pi = \pi_{1, 1}(M, N)$ (dashed gray until it meets $\pi^z$) are shown. The exit points of $\pi^z$ (black dot) and $\pi$ (gray dot unless already indicated) from $\wt{\nu}$ are marked. $\pi^z$ passes through $(1, 0)$. Then $Z_{\pi^z, \wt{\nu}} \ge Z_{\pi, \wt{\nu}}$ a.s.}
\label{FPathFluc}
\end{figure}

Now bound the last two probabilities in \eqref{E-33} as follows. By virtue of Proposition \ref{PGeoInStep}(b), 
\begin{align}
\P\{\Eh^{z, \ver}(M, N) > 0\} \le \exp\{-c_0 s^3\} \label{E-32}
\end{align}
for some constant $c_0 = c_0(\delta) > 0$. Note also from \eqref{EBaseCond} and the choice of $z$ that 
\begin{align*}
z-\Min(m, n) &= -\epsilon_0 s(M+N)^{-1/3} + \Min(M, N)-\Min(m, n) \ge -2\epsilon_0 s(m+n)^{-1/3}. 
\end{align*} 
Therefore, for sufficiently small $\epsilon_0$, one obtains from Theorem \ref{TExitUB}(a) that 
\begin{align}
\P\{\Eh^{z, +}_{\wt{\nu}, i_0+1, j_0+1}(M, N) > k\} \le \exp\{-c_0 \min \{s^3, m+n\}\} \label{E-34}
\end{align}
after increasing $N_0$ and decreasing $c_0$ if necessary. 

Putting together \eqref{E-33}, \eqref{E-32} and \eqref{E-34} and choosing $s_0$ sufficiently large establish the claimed upper bound for $s \in [s_0, c(m+n)^{1/3}]$ when $\square = +$. The upper bound extends to $s > c(m+n)^{1/3}$ after modifying $c_0$ suitably (as done previously in \eqref{E-25}). The case $\square = -$ is handled similarly. 
\end{proof}

\section{Proofs of the speed bounds}\label{SPfApp}

This section proves Theorems \ref{TBuse},  \ref{TCif} and \ref{TCifSt}.  The proofs of first two results utilize Proposition \ref{PGeoInStep}, monotonicity properties of last-passage times, planarity, and increment-stationary LPP processes with \emph{northeast} boundary weights. 
To introduce these processes, first define the weights 
\begin{align}
\label{Ewne}
\wne^{w, z, m, n}(i, j) = \eta(i, j) \bigg(\one_{\{i \le m, j \le n\}} + \frac{\one_{\{i \le m, j = n+1\}}}{w} + \frac{\one_{\{i = m+1, j \le n\}}}{1-z}\bigg)
\end{align}
for $w > 0$, $z < 1$, $m, n \in \bbZ_{>0}$, $i \in [m+1]$ and $j \in [n+1]$. In particular, $\wne^{w, z, m, n}(m+1, n+1) = 0$. Then let 
\begin{align}
\label{ELppne}
\Gne^{w, z, m, n}_{p, q}(k, l) = \max \limits_{\pi \in \Pi_{k, l}^{p, q}} \sum_{(i, j) \in \pi} \wne^{w, z, m, n}(i, j) \quad \text{ for } p, k \in [m+1] \text{ and } q, l \in [n+1]. 
\end{align}
A comparison of \eqref{Ewbd}--\eqref{Elppbd} with \eqref{Ewne}--\eqref{ELppne} shows the distributional identity
\begin{align}
\label{ElppDisId}
\begin{split}
&\{\Gne^{w, z, m, n}_{p, q}(k, l): p, k \in [m+1], q, l \in [n+1]\} \\
&\stackrel{\text{\rm{dist.}}}{=} \{\Gb^{w, z}_{m+1-p, n+1-q}(m+1-k, n+1-l): p, k \in [m+1], q, l \in [n+1]\}
\end{split}
\end{align}
Below we  use \eqref{ELppne} only in the stationary case $w = z$ and, as before, write $z$ only once in the superscript.

\subsection{Proofs of the bounds for the Busemann limits}

As the first step towards the proof of Theorem \ref{TBuse}, let us bound the c.d.f.s in \eqref{EIncCdf} via the c.d.f.s in \eqref{Ebcdf}. 

\begin{lem}
\label{LSqz}
Let $\nu$ be a down-right path in $\bbZ_{>0}^2$. Let $(u, l)$ and $(k, v)$ denote the first and last points on $\nu$ {\rm(}in the down-right direction{\rm)}. Let $\delta > 0$, $(m, n) \in S_\delta \cap (\bbZ_{\ge k} \times \bbZ_{\ge l})$, $p = k-u$, $q = l-v$, $s \in \bbR^p$, $t  \in \bbR^q$ and $z \in (0, 1)$. Write $\underline{\Min} = \Min(m-k+1, n-v+1)$ and $\overline{\Min} = \Min(m-u+1, n-l+1)$. Then there exist constants $c = c(\delta) > 0$ and $\epsilon = \epsilon(\delta) > 0$ such that
\begin{align*}
\Bcdf_{\nu}^{m, n, \hor}(s, t) &\le \bcdf_{p, q}^{z, \hor}(s, t) + \exp\{-c(m+n)(\underline{\Min}-z)^3\} \qquad \text{ if } z < \underline{\Min}, \\
\Bcdf_{\nu}^{m, n, \hor}(s, t) &\ge  \bcdf_{p, q}^{z, \hor}(s, t) - \exp\{-c(m+n)(z-\overline{\Min})^3\} \qquad \text{ if } z > \overline{\Min}, \\
\Bcdf_{\nu}^{m, n, \ver}(s, t) &\le \bcdf_{p, q}^{z, \ver}(s, t) + \exp\{-c(m+n)(z-\overline{\Min})^3\} \qquad \text{ if } z > \overline{\Min}, \\
\Bcdf_{\nu}^{m, n, \ver}(s, t) &\ge  \bcdf_{p, q}^{z, \ver}(s, t) - \exp\{-c(m+n)(\underline{\Min}-z)^3\} \qquad \text{ if } z < \underline{\Min}
\end{align*}
provided that $k, l \le \epsilon (m+n)$. 
\end{lem}
\begin{proof}
We prove the first two inequalities. The remaining two can be obtained in a similar manner. 

Recall \eqref{ERDsteps}. The number of right and down steps of $\nu$ are given by $\hash R_\nu = p$ and $\hash D_\nu = q$. Let $(r_i)_{i \in [p]}$ and $(d_j)_{j \in [q]}$ denote the enumerations of $R_\nu$ and $D_\nu$, respectively, in increasing order. Write $E_0$ for the event 
\begin{align*}
\Gne^{z, m, n}_{m+1, n+1}(\nu_{r_{i}})-\Gne^{z, m, n}_{m+1, n+1}(\nu_{r_i+1})  \ge -s_i \quad \text{ for } i \in [p] \\
\text{and}\quad \Gne^{z, m, n}_{m+1, n+1}(\nu_{d_j})-\Gne^{z, m, n}_{m+1, n+1}(\nu_{d_j-1}) \le t_j \quad \text{ for } j \in [q]. 
\end{align*}
Noting that $\nu$ is a sequence of length $\ell(\nu) = L = p+q+1$, define another sequence $\wt{\nu} = (\wt{\nu}_i)_{i \in [L]}$ via $\wt{\nu}_i = (m+1, n+1)-\nu_{L+1-i}$ for $i \in [L]$. Then $\wt{\nu}$ is also a down-right path with $R_{\wt{\nu}} = \{L-r_i: i \in [p]\}$ and $D_{\wt{\nu}} = \{L+2-d_j: j \in [q]\}$. 
Now on account of \eqref{ElppDisId}, \eqref{EBurke} and definition \eqref{Ebcdf}, the probability of $E_0$ can be computed exactly: 
\begin{align}
\label{E25}
\begin{split}
\P\{E_0\} &= \P\{\Gb^z((m+1, n+1)-\nu_{r_i})-\Gb^z((m+1, n+1)-\nu_{r_i+1}) \ge -s_i \text{ for } i \in [p], \\
&\quad \quad \ \ \Gb^z((m+1, n+1)-\nu_{d_j})-\Gb^z((m+1, n+1)-\nu_{d_{j}-1}) \le t_j \text{ for } j \in [q]\} \\
&= \P\{\Gb^z(\wt{\nu}_{L+1-r_i})-\Gb^z(\wt{\nu}_{L-r_i}) \ge -s_i \text{ for } i \in [p], \\
&\quad \quad \ \ \Gb^z(\wt{\nu}_{L+1-d_j})-\Gb^z(\wt{\nu}_{L+2-d_j}) \le t_j \text{ for } j \in [q]\} \\
&= \prod_{i \in [p]} \exp\{-s_i^-z\} \prod_{j \in [q]} (1-\exp\{-t_j^+ (1-z)\})\\
&= \bcdf_{p, q}^{z, \hor}(s, t). 
\end{split}
\end{align}

Define the events $E_1$ and $E_2$ exactly as $E_0$ but  replace the base point $(m+1, n+1)$ with $(m, n+1)$ and $(m+1, n)$, respectively.  That is,  $E_1$ denotes the event 
\begin{align*}
\Gne^{z, m, n}_{m, n+1}(\nu_{r_{i}})-\Gne^{z, m, n}_{m, n+1}(\nu_{r_i+1})  \ge -s_i \quad \text{ for } i \in [p] \\
\text{and}\quad \Gne^{z, m, n}_{m, n+1}(\nu_{d_j})-\Gne^{z, m, n}_{m, n+1}(\nu_{d_j-1}) \le t_j \quad \text{ for } j \in [q], 
\end{align*}
while $E_2$ denotes the event 
\begin{align*}
\Gne^{z, m, n}_{m+1, n}(\nu_{r_{i}})-\Gne^{z, m, n}_{m+1, n}(\nu_{r_i+1})  \ge -s_i \quad \text{ for } i \in [p] \\
\text{and}\quad \Gne^{z, m, n}_{m+1, n}(\nu_{d_j})-\Gne^{z, m, n}_{m+1, n}(\nu_{d_j-1}) \le t_j \quad \text{ for } j \in [q]. 
\end{align*}
From a union bound and \eqref{ElppDisId}, one obtains that 
\begin{align}
\label{E26}
\begin{split}
\P\{E_1\} &= \P\{E_1 \cap \{\Gne^{z, m, n}_{m, n+1}(k, v) \ge \Gne^{z, m, n}_{m+1, n}(k, v)\}\} \\
&+ \P\{E_1 \cap \{\Gne^{z, m, n}_{m, n+1}(k, v) < \Gne^{z, m, n}_{m+1, n}(k, v)\}\}\\
&= \P\{E_0 \cap \{\Gne^{z, m, n}_{m, n+1}(k, v) \ge \Gne^{z, m, n}_{m+1, n}(k, v)\}\} \\
&+ \P\{E_1 \cap \{\Gne^{z, m, n}_{m, n+1}(k, v) < \Gne^{z, m, n}_{m+1, n}(k, v)\}\}\\
&\le \P\{E_0\} + \P\{\Gne^{z, m, n}_{m, n+1}(k, v) < \Gne^{z, m, n}_{m+1, n}(k, v)\} \\ 
&= \P\{E_0\} + \P\{\Gb^{z}_{1, 0}(m+1-k, n+1-v) < \Gb^{z}_{0, 1}(m+1-k, n+1-v)\} \\
&=\P\{E_0\}+\P\{\Ev^{z, \ver}(m+1-k, n+1-v) > 0\}. 
\end{split}
\end{align}
In the second step above, we utilized this consequence of planarity and a.s.\ uniqueness of geodesics: If the geodesic from $(m+1, n+1)$ to $(k, v)$ visits $(m, n+1)$ then a.s.\ so does the geodesic from $(m+1, n+1)$ to any point in $[k] \times ([n+1] \smallsetminus [v-1])\supset \nu$. See Figure \ref{FPlan}. 
Similar reasoning also gives 
\begin{align}
\label{E27}
\begin{split}
\P\{E_2\} &= \P\{E_2 \cap \{\Gne^{z, m, n}_{m+1, n}(u, l) \ge \Gne^{z, m, n}_{m, n+1}(u, l)\}\} \\
&+ \P\{E_2 \cap \{\Gne^{z, m, n}_{m+1, n}(u, l) < \Gne^{z, m, n}_{m, n+1}(u, l)\}\}\\
&= \P\{E_0 \cap \{\Gne^{z, m, n}_{m+1, n}(u, l) \ge \Gne^{z, m, n}_{m, n+1}(u, l)\}\} \\
&+ \P\{E_2 \cap \{\Gne^{z, m, n}_{m+1, n}(u, l) < \Gne^{z, m, n}_{m, n+1}(u, l)\}\}\\
&\ge \P\{E_0\}-\P\{\Gne^{z, m, n}_{m+1, n}(u, l) < \Gne^{z, m, n}_{m, n+1}(u, l)\} \\
&= \P\{E_0\}- \P\{\Gb^{z}_{0, 1}(m+1-u, n+1-l) < \Gb^{z}_{1, 0}(m+1-u, n+1-l)\} \\
&=\P\{E_0\}- \P\{\Eh^{z, \hor}(m+1-u, n+1-l) > 0\}. 
\end{split}
\end{align}

Assume now that $k, l \le \epsilon (m+n)$ for some constant $\epsilon = \epsilon(\delta) > 0$ chosen sufficiently small to have $(m+1-k, n), (m, n+1-l) \in S_{\delta/2}$. The first two steps of the next display recall definitions \eqref{EIncCdf} and \eqref{Einc}. The third equality follows from definition \eqref{ELppne} and since $k \le m$ and $l \le n$. The subsequent inequality is an application of Lemma \ref{LCros}. The final inequality holds for some constant $c = c(\delta) > 0$ and $z < \underline{\Min}$ by virtue of \eqref{E26} and Proposition \ref{PGeoInStep}, and because $(m+1-k, n+1-v), (m+1-u, n+1-l) \in S_{\delta/2}$.  
\begin{align*}
\Bcdf_{\nu}^{m, n, \hor}(s, t) &= \P\{\B_{\nu_{r_i}}^{\hor}(m, n) \ge -s_i \text{ for } i \in [p] \text{ and } \B_{\nu_{d_j}}^{\ver}(m, n) \le t_j \text{ for } j \in [q]\} \\
&= \P\{\G_{\nu_{r_i}}(m, n)-\G_{\nu_{r_i+1}}(m, n) \ge -s_i \text{ for } i \in [p], \\ 
&\quad \quad \ \ \G_{\nu_{d_j}}(m, n)-\G_{\nu_{d_j-1}}(m, n) \le t_j \text{ for } j \in [q]\} \\
&= \P\{\Gne_{m, n}^{m, n, z}(\nu_{r_i})-\Gne_{m, n}^{m, n, z}(\nu_{r_i+1}) \ge -s_i \text{ for } i \in [p], \\ 
&\quad \quad \ \ \Gne_{m, n}^{m, n, z}(\nu_{d_j})-\Gne_{m, n}^{m, n, z}(\nu_{d_j-1}) \le t_j \text{ for } j \in [q]\} \\
&\le \P\{\Gne_{m, n+1}^{m, n, z}(\nu_{r_i})-\Gne_{m, n+1}^{m, n, z}(\nu_{r_i+1}) \ge -s_i \text{ for } i \in [p], \\ 
&\quad \quad \ \ \Gne_{m, n+1}^{m, n, z}(\nu_{d_j})-\Gne_{m, n+1}^{m, n, z}(\nu_{d_j-1}) \le t_j \text{ for } j \in [q]\} \\
&= \P\{E_1\} \\
&\le \P\{E_0\} + \exp\{-c(m+n)(\underline{\Min}-z)^3\} \quad \text{ when } z < \underline{\Min}. 
\end{align*}
A similar sequence of steps using \eqref{E27} also yields 
\begin{align*}
\Bcdf_{\nu}^{m, n, \hor}(s, t) &\ge \P\{E_2\} \ge \P\{E_0\} - \exp\{-c(m+n)(z-\overline{\Min})^3\} \quad \text{ when } z > \overline{\Min}.
\end{align*}
Then the proof is complete in view of \eqref{E25}. 
\end{proof}

\begin{figure}
\centering
\begin{overpic}[scale=0.5]{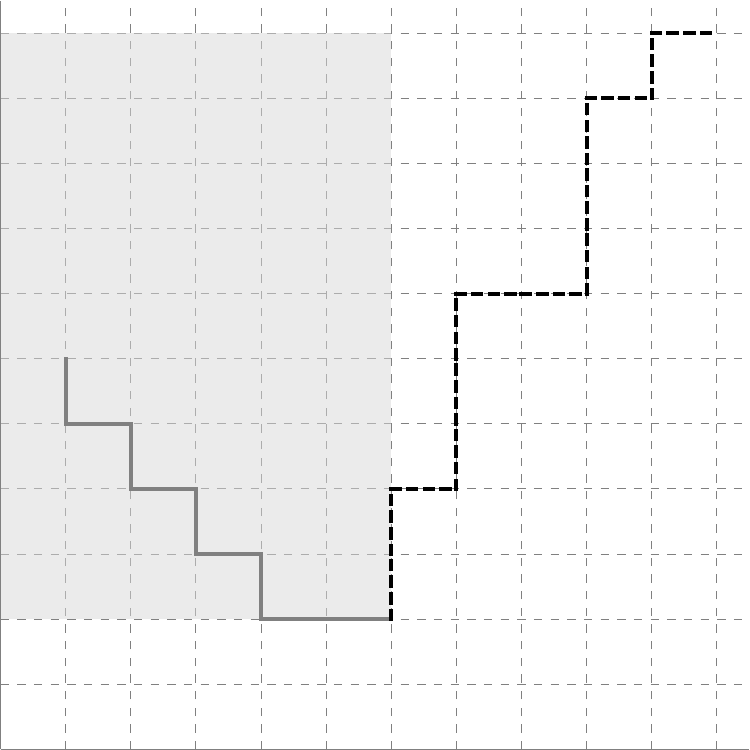}
\put(19, 36){$\nu$}
\put(97, 94){$(m+1, n+1)$}
\put(45, 12){$(k, v)$}
\put(3, 54){$(u, l)$}
\end{overpic}
\caption{Illustrates the justification for the second step in \eqref{E26}. If the geodesic (dashed black) from $(k, v)$ to $(m+1, n+1)$ visits $(m, n+1)$ then a.s.\ the geodesic from any vertex in $[k] \times ([n+1] \smallsetminus [v-1])$ (shaded region) to $(m+1, n+1)$ also visits $(m, n+1)$. In particular, this is true for any vertex on $\nu$ (gray). }
\label{FPlan}
\end{figure}
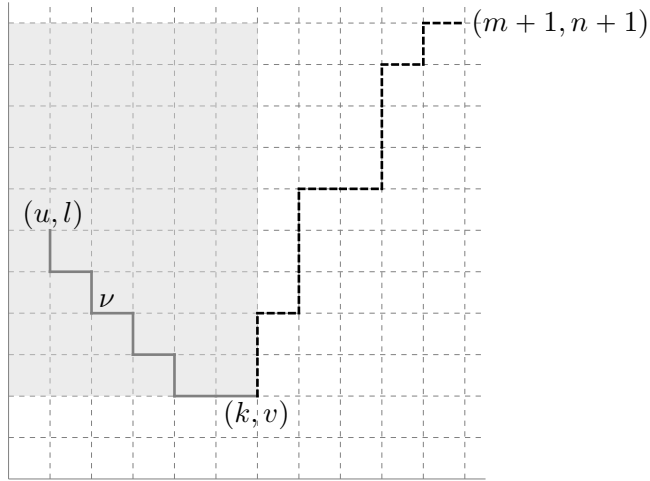

For optimal use of Lemma \ref{LSqz} ahead we bound the variations of the functions $\{\bcdf_{p, q}^{z, \square}: p, q \in \bbZ_{\ge 0}, \square \in \{\hor, \ver\}\}$ with respect to the $z$-parameter. 
\begin{lem}
\label{LcdfLip}
Let $\delta > 0$, $p, q \in \bbZ_{\ge 0}$, $s \in \bbR^p$, $t  \in \bbR^q$ and $w, z \in (\delta, 1-\delta)$ with $w \ge z$. There exists a constant $C_0 = C_0(\delta) > 0$ such that 
\begin{align*}
0 &\le \bcdf_{p, q}^{z, \hor}(s, t) - \bcdf_{p, q}^{w, \hor}(s, t) \le C_0 (1+\one_{\{q > 0\}}\log q) (w-z) \\
0 &\le \bcdf_{p, q}^{w, \ver}(s, t) - \bcdf_{p, q}^{z, \ver}(s, t) \le C_0 (1+\one_{\{p > 0\}}\log p) (w-z).  
\end{align*}
\end{lem}
\begin{proof}
Writing $s = (s_i)_{i \in [p]}$ and $t = (t_j)_{j \in [q]}$, the $z$-derivative of $\bcdf^z = \bcdf_{p, q}^{z, \hor}(s, t)$ is given by 
\begin{align}
\partial_z f^z = -f^z \sum_{i \in [p]} s_i^- - f^z \sum_{j \in [q]}\frac{t_j^+e^{-t_j^+(1-z)}}{1-e^{-t_j^+(1-z)}}, \label{E108}
\end{align}
where the  $j$th term in the second sum is interpreted as zero when $t_j \le 0$. Since $f^z > 0$, \eqref{E108} shows that $f^z$ is nonincreasing in $z$ proving the first inequality asserted in the lemma. 

To obtain the second inequality, first note that the absolute value of the first term in the right-hand side of \eqref{E108} is at most 
\begin{align}
f^z\sum_{i \in [p]}s_i^-  \le \sum_{i \in [p]}s_i^- \exp\bigg\{-\sum_{i \in [p]}s_i^- z\bigg\} \le \frac{1}{z}\sup_{t \ge 0} \{te^{-t}\} = \frac{1}{ez}. \label{E110}
\end{align}
Next bound the second term in the right-hand side of \eqref{E108} in absolute value from above by the function 
\begin{align}
\varphi_q(t) = \sum_{k \in [q]} t_k^+e^{-t_k^+(1-z)} \prod_{\substack{j \in [q] \\ j \neq k}} (1-e^{-t_j^+(1-z)}). \label{E109}
\end{align}
If $t_k \le 0$ for some $k \in [q]$ then all term vanish on the right-hand side. In the case $q \in \{0, 1\}$, one has $\varphi_q(t) \le e^{-1}(1-z)^{-1}$ similarly to \eqref{E110}. Assume that $q > 1$ and $t_k > 0$ for $k \in [q]$ from here until the last paragraph. Our objective is to maximize $\varphi_q$ over $\bbR^q_{>0}$. To aid the next computation, change the variables via $u_k = 1-\exp\{-t_k(1-z)\} \in (0, 1)$ for $k \in [q]$. Then \eqref{E109} turns into the following function of $u = (u_k)_{k \in [q]} \in (0, 1)^q$: 
\begin{align}
\psi_q(u) = -\frac{1}{1-z}\sum_{k \in [q]} (1-u_k) \log (1-u_k) \prod_{\substack{j \in [q] \\ j \neq k}} u_j. \label{E114}
\end{align} 
Note that $\psi_q$ extends continuously to $[0, 1]^q$, and the boundary values are given by 
\begin{align}
\psi_q(u)|_{u_j = 0} = 0 \quad \text{ and } \quad \psi_q(u)|_{u_j = 1} = \psi_{q-1}(u^j) \quad \text{ for } j \in [q], \label{E117}
\end{align}
where $u^j \in \bbR^{q-1}$ is obtained from $u$ by deleting the $j$th coordinate. 

The partial $u$-derivatives of $\psi_q$ are given by 
\begin{align}
\label{E111} 
\begin{split}
(1-z)\partial_r \psi_q(u) &= (1+\log (1-u_r))\prod_{\substack{j \in [q] \smallsetminus \{r\}}} u_j \\
& -\sum_{k \in [q] \smallsetminus \{r\}}(1-u_k) \log (1-u_k)\prod_{\substack{j \in [q] \smallsetminus \{k, r\}}} u_j   \\
&= \prod_{\substack{j \in [q] \smallsetminus \{r\}}} u_j \cdot \bigg(1 + \log (1-u_r) - \sum_{k \in [q] \smallsetminus \{r\}}\bigg(\frac{1}{u_k}-1\bigg)\log(1-u_k)\bigg). 
\end{split}
\end{align}
Note from \eqref{E111} that $u \in (0, 1)^q$ is a zero of the gradient $\nabla \psi_q$  if and only if 
\begin{align}
1+\frac{\log (1-u_r)}{u_r} = \sum_{k \in [q]}\bigg(\frac{1}{u_k}-1\bigg)\log(1-u_k) \quad \text{ for each } r \in [q].  \label{E112}
\end{align}
Since the function $x \mapsto x^{-1}\log (1-x)$ is strictly decreasing on $(0, 1)$, \eqref{E112} holds if and only if the coordinates $u_k = v$ for $k \in [q]$ for some $v \in (0, 1)$ such that   
\begin{align}
1 + \bigg(q-\frac{q-1}{v}\bigg) \log (1-v) = 0. \label{E113}
\end{align}
The left-hand side of \eqref{E113} defines a continuous function $g = g(v)$ on $(0, 1)$ that is decreasing since its derivative
\begin{align*}
g'(v) = \frac{(q-1)}{v^2}\log (1-v)-\bigg(q-\frac{q-1}{v}\bigg)\frac{1}{1-v} = -\frac{1}{1-v} + \frac{(q-1)}{v^2}\bigg(\log (1-v) + v\bigg) < 0, 
\end{align*}
where the last inequality can be seen from the expansion $\log (1-t) = -\sum_{i=1}^\infty t^i/i$. Furthermore, the limits of $g$ at the endpoints $0$ and $1$ can be computed as $q$ and $-\infty$, respectively. Therefore, there exists a unique $v_q \in (0, 1)$ such that \eqref{E113} holds. The next step is to verify that 
\begin{align}
v_q \le 1-\frac{1}{e^2q}. \label{E116}
\end{align}
Arguing by contradiction, suppose that \eqref{E116} is false. Then, by the monotonicity of $g$, 
\begin{align}
0 = g(v_q) < g\bigg(1-\frac{1}{e^2q}\bigg) = 1 - \frac{1-\dfrac{1}{e^2}}{1-\dfrac{1}{e^2q}} (2 + \log q) \le 1 - 2\bigg(1-\frac{1}{e^2}\bigg) < 0, 
\end{align}
a contradiction. Therefore, \eqref{E116} holds. 

Now, setting $u_j = v_q$ for $j \in [q]$ in \eqref{E114} leads to 
\begin{align}
(1-z)\psi_q((v_q)_{j \in [q]}) &= -q(1-v_q) \log(1-v_q) v_q^{q-1} \nonumber \\
&\le -(1-v_q) \log(1-v_q)-(q-1)(1-v_q) \log(1-v_q) v_q^{q-1} \nonumber \\
&= -(1-v_q) \log(1-v_q)-(1+\log(1-v_q)) v_q^{q} \label{E115}\\
&\le \frac{1}{e} +(1+\log q)v_q^q \le 2 + \log q. \label{E118}
\end{align}
Line \eqref{E115} above comes from \eqref{E113}. The first inequality in \eqref{E118} bounds the two terms of \eqref{E115} separately and uses \eqref{E116}. 

From \eqref{E118}, the structure of the boundary values \eqref{E117} and the positivity of \eqref{E114}, one concludes that 
\begin{align}
\sup_{t \in \bbR^q} \varphi_q(t) = \sup_{t \in \bbR_{>0}^q} \varphi_q(t) = \sup_{u \in [0, 1]^q} \psi_q(u) \le \frac{2+\log q}{1-z}. \label{E119}
\end{align}
Combining \eqref{E119} with the bounds in the case $q \in \{0, 1\}$ and \eqref{E110} gives 
\begin{align*}
\partial_z f^z \le C_0 (1 + \one_{\{q > 0\}}\log q) \quad \text{ for } z \in (\delta, 1-\delta)
\end{align*}
for some constant $C_0 = C_0(\delta) > 0$. Then the second inequality of the lemma follows from the mean value theorem. 

The proofs of the remaining inequalities are similar,  and are therefore omitted. 
\end{proof}

\begin{proof}[Proof of Theorem \ref{TBuse}]
Let $(m, n) \in S_\delta \cap \bbZ_{\ge N_0}^2$ where $N_0 = N_0(\delta, \epsilon) > 0$ is a constant to be chosen below. Let $\nu$ be a down-right path contained in $[1, \epsilon (m+n)^{2/3}]^2$ taking $N_0$ large enough to ensure that the preceding set is nonempty.  Write $(u, l) = \nu_1$ and $(k, v) = \nu_{\ell{(\nu)}}$ for the first and last vertices, respectively, on $\nu$. Let 
\begin{align*}
w = \overline{\Min} + \bigg(\frac{r \log (m+n)}{m+n}\bigg)^{1/3} \quad \text{ and } \quad z = \underline{\Min}-\bigg(\frac{r \log (m+n)}{m+n}\bigg)^{1/3}, 
\end{align*}
where $\underline{\Min} = \Min(m+1-k, n+1-v)$, $\overline{\Min} = \Min(m+1-u, n+1-l)$ and $r = r(\delta) > 0$ is another constant to be specified below. Lemmas \ref{LShpMinBnd}(b) and \ref{LMinEst} combined with the fact that $k, l \le \epsilon (m+n)^{2/3}$ imply that $w, z \in (\epsilon_0, 1-\epsilon_0) > 0$ for some constant $\epsilon_0 = \epsilon(\delta) > 0$, provided that $N_0$ is sufficiently large. Let $p = k-u = \hash R_{\nu}$, $q = l-v = \hash D_{\nu}$, $s \in \bbR^p$ and $t \in \bbR^q$. From the choice of $z$ and the first inequalities in Lemmas \ref{LSqz} and \ref{LcdfLip}, one obtains that
\begin{align*}
\Bcdf_{\nu}^{m, n, \hor}(s, t) &\le \bcdf_{p, q}^{z, \hor}(s, t) + \exp\{-c_0(m+n)(\underline{\Min}-z)^3\} \\
&= \bcdf_{p, q}^{z, \hor}(s, t) + \frac{1}{(m+n)^{c_0r}} \\
&\le \bcdf_{p, q}^{\Min, \hor}(s, t) + C_0(1+\one_{\{q > 0\}}\log q) |z-\Min| + \frac{1}{(m+n)^{c_0 r}} \\
&\le \bcdf_{p, q}^{\Min, \hor}(s, t) + C_0(1+\one_{\{q > 0\}}\log q) \bigg\{|\underline{\Min}-\Min| + \bigg(\frac{r\log (m+n)}{m+n}\bigg)^{1/3}\bigg\} + \frac{1}{(m+n)^{c_0 r}} \\
&\le \bcdf_{p, q}^{\Min, \hor}(s, t) + 2C_0 (1+\one_{\{q > 0\}}\log q) \bigg(\frac{r\log (m+n)}{m+n}\bigg)^{1/3} + \frac{1}{(m+n)^{c_0 r}}
\end{align*}
for some constants $c_0 = c_0(\delta) > 0$, $C_0 = C_0(\delta) > 0$ and sufficiently large $N_0$. The last inequality above holds because  
\begin{align*}
|\underline{\Min}-\Min| \le \frac{C(k+v)}{m+n} \le \frac{C (k+l)}{m+n} \le \frac{2C\epsilon}{(m+n)^{1/3}}
\end{align*}
for some constant $C = C(\delta) > 0$ by virtue of Lemma \ref{LMinEst} and that $k, l \le \epsilon (m+n)^{2/3}$. Now choosing $r = \dfrac{1}{3c_0}$ yields 
\begin{align*}
\Bcdf_{\nu}^{m, n, \hor}(s, t)  \le \bcdf_{p, q}^{\Min, \hor}(s, t) + C_0 (1+\one_{\{q > 0\}}\log q) \bigg(\frac{\log (m+n)}{m+n}\bigg)^{1/3}
\end{align*}
after adjusting $C_0$. The complementary lower bound is established similarly using $w$ instead of $z$. The second set of bounds in the theorem are also proved similarly. 
\end{proof}

\subsection{Proof of the speed bounds for the competition interface}


\begin{proof}[Proof of Theorem \ref{TCif}]
Let $n \in \bbZ_{>1}$ and $k \in [n-1]$. Restrict to the full probability event on which the competition interface $\cif$ is well-defined. In the next display, the first equality follows from \eqref{ETvmon} and definition \eqref{Ecif}, and the subsequent equalities are due to definitions \eqref{ETv} and \eqref{Einc}.  
\begin{align}
\label{E121}
\begin{split}
\{\cif^{\hor}_n > k\} &= \{(k+1, n-k+1) \in \tree^{\ver}\} \\
&= \{\G_{2, 1}(k+1, n-k+1) < \G_{1, 2}(k+1, n-k+1)\} \\
&= \{\B^{\hor}_{1, 1}(k+1, n-k+1) > \B^{\ver}_{1, 1}(k+1, n-k+1\}. 
\end{split}
\end{align}

Put $\Min = \Min(k+1, n-k)$ and let $z > \Min$ to be chosen below.  The next derivation begins with \eqref{E121}.  The second step writes the increments $\B^{\square}_{1, 1}(k+1, n-k+1)$ for $\square \in \{\hor,  \ver\}$ defined at \eqref{Einc} in terms of the process $\Gne^{z} = \Gne^{z, k+1, n-k+1}$ defined at \eqref{ELppne}.  The inequality at the end holds by virtue of Lemma \ref{LCros}. 
\begin{align}
&\P\{\cif^{\hor}_n \le k\} \nonumber\\
&= \P\{\B^{\hor}_{1, 1}(k+1, n-k+1) \le \B^{\ver}_{1, 1}(k+1, n-k+1)\} \nonumber\\
&= \P\{\Gne_{k+1, n-k+1}^{z}(1, 1)-\Gne_{k+1, n-k+1}^z(2, 1) \le \Gne_{k+1, n-k+1}^{z}(1, 1)-\Gne_{k+1, n-k+1}^z(1, 2)\}\nonumber\\
&\le \P\{\Gne_{k+2, n-k+1}^{z}(1, 1)-\Gne_{k+2, n-k+1}^z(2, 1) \le \Gne_{k+2, n-k+1}^{z}(1, 1)-\Gne_{k+2, n-k+1}^z(1, 2)\}.\label{E1000}
\end{align}
The next display applies a union bound using the following implication of planarity and the uniqueness of geodesics (as in the proof of Lemma \ref{LSqz}): If the vertex $(k+2, n-k+1)$ is on the geodesic from $(1, 2)$ to $(k+2, n-k+2)$ then it must also be on the two geodesics from $(1, 1)$ and $(2, 1)$ to $(k+2, n-k+2)$.  In terms of the northeast LPP process, 
this means that the inequality $\Gne_{k+1, n-k+2}^{z}(1, 2) < \Gne_{k+2, n-k+1}^{z}(1, 2)$ implies that $\Gne_{k+2, n-k+1}^{z}(i, j) = \Gne_{k+2, n-k+2}^{z}(i, j)$ for $(i, j) \in \{(1, 2), (1, 1), (2, 1)\}$.  Therefore, 
\begin{align}
&\text{the right-hand side of \eqref{E1000}} \nonumber\\
&\le \P\{\Gne_{k+2, n-k+2}^{z}(1, 1)-\Gne_{k+2, n-k+2}^z(2, 1) \le \Gne_{k+2, n-k+2}^{z}(1, 1)-\Gne_{k+2, n-k+2}^z(1, 2)\}\label{E1001}\\
&\qquad+ \P\{\Gne_{k+1, n-k+2}^{z}(1, 2) > \Gne_{k+2, n-k+1}^{z}(1, 2)\} \nonumber\end{align}
By virtue of \eqref{ElppDisId} and then \eqref{EBurke},  the right-hand side of \eqref{E1001} can be written as 
\begin{align}
&= \Gb^{z}(k+1, n-k+1)-\Gb^z(k+1, n-k)\}\nonumber\\
&\qquad + \P\{\Gb_{1, 0}^{z}(k+1, n-k) > \Gb_{0, 1}^{z}(k+1, n-k)\} \nonumber\\
&= \int_0^\infty ze^{-zx}\int_x^\infty (1-z)e^{-(1-z)y} \dd y \dd x + \P\{\Eh^{z, \hor}(k+1, n-k) > 0\}\nonumber\\
&= z + \P\{\Eh^{z, \hor}(k+1, n-k) > 0\}. \label{E122}
\end{align}

To avoid a vacuous statement, assume that $\delta \in (0, 1/2)$. Work with $k \in [\delta n, (1-\delta) n]$ and $n \ge N_0$ for some sufficiently large $N_0 = N_0(\delta) > 0$ that ensures that the preceding interval contains some integers. Then $(k+1, n-k) \in S_{\delta/2}$. Therefore, by the assumption $z>\Min$ and Proposition \ref{PGeoInStep},  
\begin{align}
\P\{\Eh^{z, \hor}(k+1, n-k) > 0\} \le \exp\{-c_0n(z-\Min)^3\}  \label{E123}
\end{align}
for some constant $c_0 = c_0(\delta) > 0$.    Set $z = \Min + \bigg(\dfrac{\log n}{3c_0n}\bigg)^{1/3}$ after increasing $N_0$ if necessary to have $z \in (0, 1)$. Resuming from \eqref{E122} and using \eqref{E123}, one obtains that 
\begin{align}
\P\{\cif^{\hor}_n \le k\} \le \Min + C\bigg(\frac{\log n}{n}\bigg)^{1/3} \label{E124}
\end{align}
for some constant $C = C(\delta) > 0$. For $x \in [\delta, 1-\delta]$, setting $k = \lc nx \rc$ in \eqref{E124} and using Lemma \ref{LMinEst} yield
\begin{align*}
\P\{\cif^{\hor}_n \le nx\} &\le \P\{\cif^{\hor}_n \le \lc nx \rc\} \le \Min(\lc nx \rc+1, n-\lc nx \rc) + C\bigg(\frac{\log n}{n}\bigg)^{1/3} \\
&\le \Min(x, 1-x) + \frac{c}{n} +C\bigg(\frac{\log n}{n}\bigg)^{1/3} \le \Min(x, 1-x) + C_0\bigg(\frac{\log n}{n}\bigg)^{1/3}
\end{align*}
for some constants $c = c(\delta) > 0$ and $C_0 = C_0(\delta) > 0$. The last bound also holds for $n \in \{2, \dotsc, N_0\}$ after adjusting $C_0$. 

To prove complementary lower bound, use Lemma \ref{LCros} to replace 
the 
inequality at \eqref{E1000} 
with 
\begin{align*}
&\P\{\B^{\hor}_{1, 1}(k+1, n-k+1) \le \B^{\ver}_{1, 1}(k+1, n-k+1)\}  \\
&\ge \P\{\Gne_{k+1, n-k+2}^{z}(1, 1)-\Gne_{k+1, n-k+2}^z(2, 1) \le \Gne_{k+1, n-k+2}^{z}(1, 1)-\Gne_{k+1, n-k+2}^z(1, 2)\}
\end{align*} 
and follow similar steps.  
\end{proof}

\begin{proof}[Proof of Theorem \ref{TCifSt}]
Let $w > 0$ and $z < 1$. From definition \eqref{Ecifbd}, one has 
\begin{align}
\label{E100}
\begin{split}
&\wh{\cif}^{w, z, \hor}_n < k \quad \text{ if and only if } \quad \Eh^{w, z, \hor}(k, n-k+1) > 0, \quad \text{ and }\\
&\wh{\cif}^{w, z, \hor}_n > k \quad \text{ if and only if } \quad \Eh^{w, z, \ver}(k+1, n-k) > 0
\end{split}
\end{align}
for $n \in \bbZ_{>0}$ and  $k \in [n] \cup \{0\}$. 

Set $w = z \in (0, 1)$ from here on. Let $n \in \bbZ_{>0}$, and pick $x \in [\delta, 1-\delta]$ assuming that $\delta \le 1/2$ to avoid vacuous statements.  It follows from Lemmas \ref{LMinEst} and \ref{LShpMinBnd}(b) that 
\begin{align}
\label{E101}
\begin{split}
|\Min(\lf nx \rf + 1, n - \lf n x \rf) -  \Min(x, 1-x)| \le \frac{A_0} {2n}
\end{split}
\end{align} 
for some constant $A_0 = A_0(\delta) > 0$. Using \eqref{E100}, Proposition \ref{PGeoInStep}(a) and \eqref{E101} leads to 
\begin{align}
\label{E102}
\begin{split}
\P\{\wh{\cif}_n^{z, \hor} \le nx\} &= \P\{\wh{\cif}_n^{z, \hor} < \lf nx \rf + 1\} = \P\{\Eh^{z, \hor}(\lf nx \rf + 1, n - \lf nx \rf) > 0\} \\
&\le \exp\{-c_0 (n+1) [z-\Min(\lf nx \rf + 1, n - \lf nx \rf)]^3 \} \\
&\le \exp\{-c_0 n [z-\Min(x, 1-x)]^3\}
\end{split}
\end{align}
for some constant $c_0 = c_0(\delta) > 0$ provided that $z-\Min(x, 1-x) \ge A_0 n^{-1}$. 
A similar computation using Proposition \ref{PGeoInStep}(b) also gives 
\begin{align}
\label{E103}
\begin{split}
1-\P\{\wh{\cif}_n^{z, \hor} \le nx\} &=  \P\{\wh{\cif}_n^{z, \hor} > \lf nx \rf \} = \P\{\Eh^{z, \ver}(\lf nx \rf +1, n- \lf nx \rf) > 0\} \\
&\le \exp\{-c_0 n[\Min(x, 1-x)-z]^3\}
\end{split}
\end{align}
provided that $z-\Min(x, 1-x) \le -A_0 n^{-1}$. Combining \eqref{E102} and \eqref{E103} yields (a). 

For part (b), pick a sufficiently large constant $N_0 = N_0(\delta, \epsilon) > 0$ such that $\epsilon N_0^{-1/3} \ge A_0 N_0^{-1}$, and work with $n \in \bbZ_{\ge N_0}$.  Assume also that $|\zeta(x, 1-x)-z| \le \epsilon_0$ where $\epsilon_0 = \epsilon_0(\delta) > 0$ is a constant to be chosen small. 
After the first line of \eqref{E102}, invoke Proposition \ref{PGeoInStepLB}(a) and appeal to \eqref{E101} to obtain 
\begin{align*}
\begin{split}
\P\{\wh{\cif}_n^{z, \hor} \le nx\} &\ge \exp\{-C_0 (n+1) [z-\Min(\lf nx \rf + 1, n - \lf nx \rf)]^3\} \\
&\ge \exp\{-C_0 n [z-\Min(x, 1-x)]^3\}
\end{split}
\end{align*}
for some constant $C_0 = C_0(\delta, \epsilon) > 0$ provided that $z - \Min (x, 1- x) \ge \epsilon n^{-1/3} \ge A_0 n^{-1}$,  $\epsilon_0$ is sufficiently small,  and $N_0$ is sufficiently large.  In the same vein, 
\begin{align*}
\begin{split}
1-\P\{\wh{\cif}_n^{z, \hor} \le nx\} &\ge \exp\{-C_0 n [\Min(x, 1-x)-z]^3\}
\end{split}
\end{align*}
provided that $z- \Min (x, 1- x) \le -\epsilon n^{-1/3}$. Hence, (b). 
\end{proof}

\appendix

\section{Some deterministic properties of exit points and last-passage times}
\label{SAux}

This section collects some general properties of last-passage times and exit points defined from arbitrary real weights.  Only subsections \ref{SsExGeod} and \ref{SsCros} are needed for the main text.  The purpose of the remaining material is to service Appendix \ref{SExBnd2}.

\subsection{Maximal exit points of geodesics}
\label{SsExGeod}

Let $w = \{w(i, j): i, j \in \bbZ\}$ be a collection of real (nonrandom) weights on $\bbZ^2$. As in \eqref{Elpp}, define the corresponding last-passage times by 
\begin{align}
G_{p, q}(m, n) = \max_{\pi \in \Pi_{p, q}^{m, n}} \sum_{(i, j) \in \pi} w(i, j) \quad \text{ for } m, n, p, q \in \bbZ.  \label{Elpp2}
\end{align}
Any maximizing path $\pi \in \Pi_{p, q}^{m, n}$ in \eqref{Elpp2} is called a geodesic (or $w$-geodesic) from $(p, q)$ to $(m, n)$. When $p \le m$ and $q \le n$, being a nonempty and finite set, $\Pi_{p, q}^{m, n}$ contains a geodesic,   and possibly more than one. 

Fix $(m_0, n_0) \in \bbZ^2$. Pick a down-right path $\nu$ with $\nu_{\ell(\nu)} \in \{u\} \times \bbZ_{\ge n_0}$ and $\nu_1 \in \bbZ_{\ge m_0} \times \{v\}$ for some $(u, v) \in \bbZ_{\ge m_0} \times \bbZ_{\ge n_0}$, and a base vertex $(i_0, j_0) = \nu_{b}$ for some $b \in [\ell(\nu)]$. Refer to \eqref{EVDef} for the definition of the set $V_\nu$. 


Recalling \eqref{EExitPts}, introduce the maximal (rightmost/leftmost) exit points of $w$-geodesics in $\Pi_{m_0, n_0}^{m, n}$ from $\nu$ by
\begin{align}
\label{EExitPts1}
\begin{split}
Z_{\nu, i_0, j_0}^\square(m, n) &= \max \{Z_{\pi, \nu, i_0, j_0}^\square: \pi \in \Pi_{m_0, n_0}^{m, n} \text{ is a $w$-geodesic}\}
\end{split}
\end{align}
for $(m, n) \in V_\nu$ and $\square \in \{+, -\}$. These are deterministic versions of the exit points in \eqref{EEx}. The right-hand side of \eqref{EExitPts1} is well-defined and finite since $\pi \cap \nu \neq \emptyset$ for each $\pi \in \Pi_{m_0, n_0}^{m, n} \neq \emptyset$. 


The next lemma records a monotonicity property for the exit points in \eqref{EExitPts1}. 
\begin{lem}
\label{LExMon}
The exit points in \eqref{EExitPts1} satisfy the following properties for each $(m, n) \in V_{\nu}$, $p \in \bbZ \cap [m_0, u]$ and $q \in \bbZ \cap [n_0, v]$. 
\begin{enumerate}[\normalfont (a)]
\item $Z^+_{\nu, i_0, j_0}(m, n)$ is nondecreasing and $Z^-_{\nu, i_0, j_0}(m, n)$ is nonincreasing as a function of $w(p, n_0)$. 
\item $Z^+_{\nu, i_0, j_0}(m, n)$ is nonincreasing and $Z^-_{\nu, i_0, j_0}(m, n)$ is nondecreasing as a function of $w(m_0, q)$. 
\end{enumerate}
\end{lem}
\begin{proof}
Let $(m, n) \in V_{\nu}$ and $\pi \in \Pi_{m_0, n_0}^{m, n}$. Write $L$ for the L-shaped path $L_{m_0, n_0}^{m, n}$. 
Let $k, l \in [\ell(\pi)]$ denote the unique indices such that $\pi_k = \nu_{Z_{\pi, \nu}}$ and $\pi_l = L_{Z_{\pi, L}}$. Pick another path $\wt{\pi} \in \Pi_{m_0, n_0}^{m, n}$ and define the indices $\wt{k}, \wt{l} \in [\ell(\wt{\pi})]$ analogously. Note that if $k < \ell(\pi)$ then $\pi_{k+1} \in V_{\nu} \smallsetminus \nu \subset \bbZ_{>m_0} \times \bbZ_{>n_0}$ and therefore $\pi_{k+1} \not \in L$. Hence, $k \ge l$ and, similarly, $\wt{k} \ge \wt{l}$. 

To prove the first statement in (a), let $p \in \bbZ \cap [m_0, u]$, and consider real weights $\wt{w} = \{\wt{w}(i, j): i, j \in \bbZ\}$ such that $\wt{w}(i, j) \ge w(i, j)$ if $(i, j) = (p, n_0)$ and $\wt{w}(i, j) = w(i, j)$ otherwise. Let $\wt{Z}_{\nu, i_0, j_0}^{\square}(m, n)$ denote the exit points computed as in \eqref{EExitPts1} using the $\wt{w}$-weights in place of the $w$-weights. From here on, choose $\pi$ as a $w$-geodesic and $\wt{\pi}$ as a $\wt{w}$-geodesic. 

Arguing by contradiction, suppose that 
\begin{align}
Z_{\nu, i_0, j_0}^+(m, n) > \wt{Z}_{\nu, i_0, j_0}^+(m, n). \label{E-27}
\end{align}
Assume that $Z_{\pi, \nu}$ is maximal over all choices of $\pi$ as a $w$-geodesic. Then $Z_{\wt{\pi}, \nu} < Z_{\pi, \nu}$ because otherwise $Z_{\nu, i_0, j_0}^+(m, n) = [Z_{\pi, \nu}-b]^+ \le [Z_{\wt{\pi}, \nu}-b]^+ \le \wt{Z}_{\nu, i_0, j_0}^+(m, n)$, which would violate assumption \eqref{E-27}. This means that the vertex $\wt{\pi}_{\wt{k}}$ appears strictly earlier than $\pi_k$ in the sequence $\nu$ (namely, $\nu$ first visits $\wt{\pi}_{\wt{k}}$ and then $\pi_k \neq \wt{\pi}_{\wt{k}}$ in the down-right direction). The maximality of $Z_{\pi, \nu}$ and \eqref{E-27} also imply that $\pi$ is not a $\wt{w}$-geodesic. Consequently, 
\begin{align}
\sum_{s \in [\ell(\pi)]} \wt{w}(\pi_s) < \sum_{s \in [\ell(\wt{\pi})]} \wt{w}(\wt{\pi}_s).  \label{E-40}
\end{align}
Since $\pi$ is a $w$-geodesic, the strict inequality in \eqref{E-40} is possible only if $(p, n_0) \not \in \pi$ and $(p, n_0) \in \wt{\pi}$ due to the structure of the $\wt{w}$-weights. Since also $\pi_1 = \wt{\pi}_1 = (m_0, n_0) \in L$, necessarily $p > m_0$ and $Z_{\pi, L} < Z_{\wt{\pi}, L}$. Thus, $\pi_l$ comes strictly earlier than $\wt{\pi}_{\wt{l}}$ in the sequence $L$. 

Now because $L$ and $\nu$ are both down-right, $\pi$ and $\wt{\pi}$ are both up-right, $k \ge l$ and $\wt{k} \ge \wt{l}$, it follows from the orderings of the exit points above that $\pi_{r} = \wt{\pi}_{\wt{r}}$ for some indices $r \in \bbZ \cap (l, k]$ and $\wt{r} \in \bbZ \cap (\wt{l}, \wt{k}]$. See Figure \ref{FExitMon}. Let $\rho \in \Pi_{m_0, n_0}^{m, n}$ denote the up-right path obtained from $\wt{\pi}$ by replacing the segment  $\{\wt{\pi}_s: s \in [\ell(\wt{\pi})] \smallsetminus [\wt{r}-1]\}$ with the segment $\{\pi_s: s \in [\ell(\pi)] \smallsetminus [r-1]\}$ of $\pi$. Being disjoint from $L$, the preceding segments are both $w$-geodesics. Hence, $\rho$ is a $\wt{w}$-geodesic. Furthermore, by the construction, $Z_{\rho, \nu} = Z_{\pi, \nu}$ since $r \le k$. Then $\wt{Z}_{\nu, i_0, j_0}^+(m, n) \ge [Z_{\rho, \nu}-b]^+ = [Z_{\pi, \nu}-b]^+ = Z_{\nu, i_0, j_0}^+(m, n)$ contradicting \eqref{E-27}. The claimed monotonicity of $Z^+_{\nu, i_0, j_0}(m, n)$ follows as a consequence. 

The second statement in (a), and part (b) are obtained similarly. 
\end{proof}

\begin{figure}
\centering
\begin{overpic}[scale=0.65]{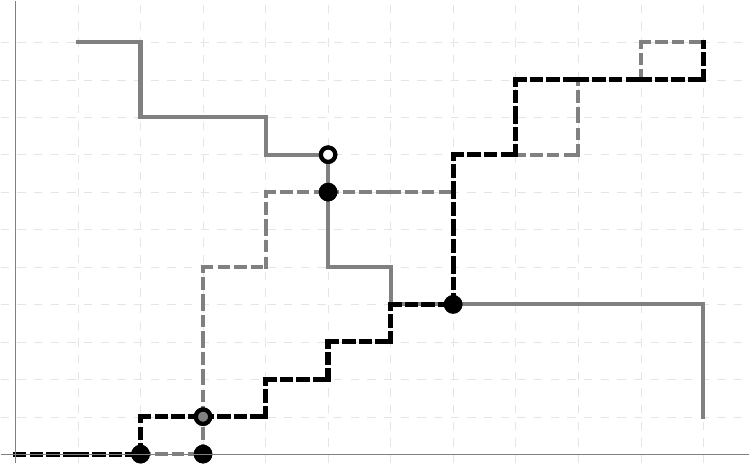}
\put(21, 48){$\nu$}
\put(45, 43){$\nu_b$}
\put(3, 25){$L$}
\put(14, 6){$\pi$}
\put(23, 20){$\wt{\pi}$}
\put(40, 8){$\rho$}
\put(21, 9){$\pi_r$\ \ $\wt{\pi}_{\wt{r}}$}
\put(16, -3){$\pi_l$}
\put(28, -3){$\wt{\pi}_{\wt{l}}$}
\put(45, 32){$\wt{\pi}_{\wt{k}}$}
\put(59, 17){$\pi_k$}
\end{overpic}
\caption{Illustrates the argument by contradiction in the proof of Lemma \ref{LExMon}. Assumption \eqref{E-27} implies that the exit points (black dots) of the maximal (rightmost) $w$-geodesic $\pi$ (dashed black) and $\widetilde{w}$-geodesic $\widetilde{\pi}$ (dashed gray) from the downright paths $\nu$ (gray) and $L$ (light gray) are positioned as displayed. Therefore, $\pi$ and $\wt{\pi}$ intersect at a vertex $\pi_r = \widetilde{\pi}_{\widetilde{r}}$ after both exiting $L$ but before exiting $\nu$. Then the path $\rho$ consisting of the segment of $\widetilde{\pi}$ until $\pi_r$ and the segment of $\pi$ after $\pi_r$ is also a $\widetilde{\pi}$-geodesic. $\rho$ exits $\nu$ at the same vertex as $\pi$ contradicting \eqref{E-27}.}
\label{FExitMon}
\end{figure}
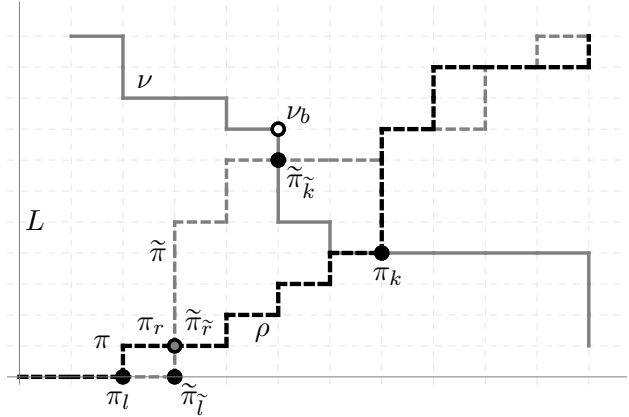

\subsection{Crossing (comparison) lemma}
\label{SsCros}

The next lemma states a  well-known monotonicity for the increments of  planar first- and last-passage percolation.  
Different proofs can be found in 
\cite[Lemma 6.2]{rass-cgm-18} and \cite[Lemma 4.6]{sepp-cgm-18}. 

\begin{lem}
\label{LCros} 
For the LPP values in \eqref{Elpp2}, the following inequalities hold for $i, j, m, n \in \bbZ$ with $i \le m$ and $j \le n$: 
\begin{align*}
\begin{split}
G_{i, j}(m+1, n)-G_{i+1, j}(m+1, n) &\le  G_{i, j}(m, n)-G_{i+1, j}(m, n) \\
& \le
 G_{i, j}(m, n+1)-G_{i+1, j}(m, n+1) ;   
 \end{split} \\[3pt] 
\begin{split}
G_{i, j}(m, n+1)-G_{i, j+1}(m, n+1)   &\le  G_{i, j}(m, n)-G_{i, j+1}(m, n) \\
&\le  G_{i, j}(m+1, n)-G_{i, j+1}(m+1, n) . 
 \end{split} 
\end{align*}
\end{lem}

\subsection{Induced path-to-point LPP and exit points}
\label{SsLppInd}

Resume with the setting in Section \ref{SsExGeod}. 
The weights and LPP values in \eqref{Elpp2} together with the path $\nu$ induce new weights $\underline{w} = \{\underline{w}(i, j): i, j \in \bbZ\}$ on $\bbZ^2$ defined as follows. If $(i, j) = \nu_k$ for some $k \in [\ell(\nu)]$ then 
\begin{align}
\underline{w}(i, j) = \one_{\{k > b\}}\{G_{m_0, n_0}(\nu_k)-G_{m_0, n_0}(\nu_{k-1})\} + \one_{\{k < b\}}\{G_{m_0, n_0}(\nu_{k})-G_{m_0, n_0}(\nu_{k+1})\} \label{EIndw}
\end{align}
In particular, $\underline{w}(i_0, j_0) = \underline{w}(\nu_b) = 0$. If $(i, j) \in \bbZ^2 \smallsetminus \nu$ then $\underline{w}(i, j) = w(i, j)$. The weights in \eqref{EIndw} are well-defined and finite due to the choice of $\nu$, and satisfy the identity 
\begin{align}
G_{m_0, n_0}(\nu_k)-G_{m_0, n_0}(i_0, j_0) = \one_{\{k > b\}} \sum_{r = b}^{k} \underline{w}(\nu_r) + \one_{\{k < b\}} \sum_{r=k}^b \underline{w}(\nu_r) \quad \text{ for } k \in [\ell(\nu)]. \label{ETelId}
\end{align}

Write $\underline{G}_{p, q}(m, n)$ for the last-passage time from $(p, q) \in \bbZ^2$ to $(m, n) \in \bbZ^2$ computed from the $\underline{w}$-weights as in \eqref{Elpp2}. Using the notation from \eqref{ERDsteps}, also define 
\begin{align}
\underline{G}_{\nu_k}^{\circ}(m, n) &= \max_{\pi \in \underline{\Pi}^{m, n}_{\nu, k}} \sum_{(i, j) \in \pi} \underline{w}(i, j) \quad \text{ for } (m, n) \in \bbZ^2 \text{ and } k \in [\ell(\nu)], \label{ELppRes}
\end{align}
where the admissible paths are given by 
\begin{align}
\underline{\Pi}^{m, n}_{\nu, k} &= 
\begin{cases}
\{\pi \in \Pi_{\nu_k}^{m, n}: \nu_k + (1, 0) \in \pi\} \quad &\text{ if } k > b \text{ and } k \in D_\nu \\
\{\pi \in \Pi_{\nu_k}^{m, n}: \nu_k + (0, 1) \in \pi\}\quad &\text{ if } k < b \text{ and } k \in R_{\nu} \\
\Pi_{\nu_k}^{m, n} \quad &\text{ otherwise. } \label{EResPath}
\end{cases} 
\end{align}
The second conditions in the first and second cases of \eqref{EResPath} mean that 
$\nu_k+ (0, 1) \in \nu$ and $\nu_k+(1, 0) \in \nu$, respectively. 

Let us refer to any maximizing path in \eqref{ELppRes} as a \emph{restricted} geodesic ($\underline{w}$-geodesic). 

Now define the path-to-point last-passage time $\underline{G}_{\nu}$ by 
\begin{align}
\begin{split}
\underline{G}_{\nu}(m, n) &= \max_{k \in [\ell(\nu)]} \bigg\{\one_{\{k > b\}} \sum_{r = b}^{k-1} \underline{w}(\nu_r)+ \one_{\{k < b\}} \sum_{r = k+1}^{b} \underline{w}(\nu_r) + \underline{G}_{\nu_k}^{\circ}(m, n)\bigg\} \\
&= \max_{k \in [\ell(\nu)]} \bigg\{G_{m_0, n_0}(\nu_k)+ \underline{G}_{\nu_k}^{\circ}(m, n)-\underline{w}(\nu_k)\bigg\}-G_{m_0, n_0}(i_0, j_0)
\end{split}
\label{EIndLpp}
\end{align}
for $(m, n) \in \bbZ^2$. The second equality above comes from \eqref{ETelId}. 
\begin{example}
\label{ExLpath}
In the case of L-shaped 
path $\nu = L_{i_0, j_0}^{u, v}$, since the first two cases in \eqref{EResPath} never occur, \eqref{ELppRes} is the same as the unrestricted version $\underline{G}_{\nu_k}(m, n)$, and \eqref{EIndLpp} is precisely the LPP considered in \cite[(A.2)]{sepp-cgm-18}. Furthermore, $k$ is a maximizer in \eqref{EIndLpp} if and only if $\nu_k \in \pi$ for some $\underline{w}$-geodesic $\pi \in \Pi_{i_0, j_0}^{m, n}$. 
\end{example}

As is clear from the definition \eqref{EIndLpp}, $\underline{G}_{\nu}(m, n)$ is finite if and only if $\underline{G}_{\nu_k}^\circ(m, n)$ is finite for some $k \in [\ell(\nu)]$. The latter is equivalent to the condition \eqref{EVDef}. 
On $V_\nu$, the LPP values in \eqref{Elpp2} and \eqref{EIndLpp} are further related through the following lemma. 


\begin{lem}
\label{LLppPath2Pt}
For the path-to-point LPP in \eqref{EIndLpp}, the following statements hold for each $(m, n) \in V_{\nu}$. 
\begin{enumerate}[\normalfont (a)]
\item $\underline{G}_{\nu}(m, n) = G_{m_0, n_0}(m, n)-G_{m_0, n_0}(i_0, j_0)$. 
\item If $\pi \in \Pi_{m_0, n_0}^{m, n}$ is a $w$-geodesic then $Z_{\pi, \nu} \in [\ell(\nu)]$ is a maximizer in \eqref{EIndLpp}. 
\item  If $k \in [\ell(\nu)]$ is a maximizer in \eqref{EIndLpp} and $\underline{\pi} \in \underline{\Pi}_{\nu, k}^{m, n}$ is a restricted $\underline{w}$-geodesic then there exists a $w$-geodesic $\pi \in \Pi_{m_0, n_0}^{m, n}$ such that $Z_{\pi, \nu} = Z_{\underline{\pi}, \nu}$. 
\end{enumerate}
\end{lem}
\begin{rem}
The case $\nu = L_{i_0, j_0}^{u, v}$ of Lemma \ref{LLppPath2Pt} has the same content as \cite[Lemma A.1]{sepp-cgm-18} in view of the characterization of the maximizers in Example \ref{ExLpath}. 
\end{rem}

\begin{proof}[Proof of Lemma \ref{LLppPath2Pt}]
Let $(m, n) \in V_{\nu}$. 
Consider a $w$-geodesic $\pi \in \Pi_{m_0, n_0}^{m, n}$. Let $i \in [\ell(\pi)]$ denote the unique index such that $\pi_i = \nu_{z}$ where $z = Z_{\pi, \nu}$. Then, starting from \eqref{EIndLpp}, one obtains the lower bound 
\begin{align}
\label{E-15}
\begin{split}
\underline{G}_{\nu}(m, n)+G_{m_0, n_0}(i_0, j_0) &\ge G_{m_0, n_0}(\nu_z)+ \underline{G}_{\nu_z}^{\circ}(m, n)-\underline{w}(\nu_z) \\
&\ge G_{m_0, n_0}(\nu_z)+ \sum_{j \in [\ell(\pi)] \smallsetminus [i]} \underline{w}(\pi_j) \\
&= G_{m_0, n_0}(\nu_z)+ \sum_{j \in [\ell(\pi)] \smallsetminus [i]} w(\pi_j) \\
&= G_{m_0, n_0}(\nu_z) + G_{\nu_z}(m, n)-w(\nu_z) \\
&= G_{m_0, n_0}(m, n).  
\end{split}
\end{align}
The second inequality in \eqref{E-15} holds by definitions \eqref{ELppRes}-\eqref{EResPath} and the fact that if $i < \ell(\pi)$ then $\pi_{i+1} \not \in \nu$ and consequently $\{\pi_{j}: j \in [\ell(\pi)] \smallsetminus [i-1]\} \in \underline{\Pi}_{\nu, z}^{m, n}$. 
The first equality in \eqref{E-15} follows because $w$- and $\underline{w}$-weights coincide on 
$\bbZ^2 \smallsetminus \nu$. The subsequent equalities are due to the assumption that $\pi$ is a $w$-geodesic. 

To proceed in the converse direction, let $k \in [\ell(\nu)]$ be a maximizer in \eqref{EIndLpp}. Then $\underline{\Pi}_{\nu, k}^{m, n} \neq \emptyset$. Pick a restricted $\underline{w}$-geodesic $\underline{\pi} \in \underline{\Pi}_{\nu, k}^{m, n}$. Write $\underline{z} = Z_{\underline{\pi}, \nu} \in [\ell(\nu)]$, and $\underline{i} \in [\ell(\underline{\pi})]$ for the unique index such that $\underline{\pi}_{\underline{i}} = \nu_{\underline{z}}$. Then develop the upper bound 
\begin{align}
\label{E-16}
\begin{split}
\underline{G}_{\nu}(m, n)+G_{m_0, n_0}(i_0, j_0) &= G_{m_0, n_0}(\nu_{k})+ \underline{G}_{\nu_{k}}^{\circ}(m, n)-\underline{w}(\nu_{k})\\
&=  G_{m_0, n_0}(\nu_{k}) + \sum_{j \in [\ell(\underline{\pi})] \smallsetminus \{1\}} \underline{w}(\underline{\pi}_{j}) \\
&= G_{m_0, n_0}(\nu_{k}) + \sum_{j \in [\underline{i}] \smallsetminus \{1\}} \underline{w}(\underline{\pi}_{j}) + \sum_{j \in [\ell(\underline{\pi})] \smallsetminus [\underline{i}]} \underline{w}(\underline{\pi}_{j}) \\
&= G_{m_0, n_0}(\nu_{k}) + \sum_{j \in [\underline{i}] \smallsetminus \{1\}} \underline{w}(\underline{\pi}_{j}) + \sum_{j \in [\ell(\underline{\pi})] \smallsetminus [\underline{i}]} w(\underline{\pi}_{j}) \\
&= G_{m_0, n_0}(\nu_{\underline{z}})  + \sum_{j \in [\ell(\underline{\pi})] \smallsetminus [\underline{i}]} w(\underline{\pi}_{j}) \\
&\le G_{m_0, n_0}(\nu_{\underline{z}})  + G_{\nu_{\underline{z}}}(m, n)-w(\nu_{\underline{z}}) \\
&\le G_{m_0, n_0}(m, n). 
\end{split}
\end{align}
The first two equalities in \eqref{E-16} use that $k$ and $\underline{\pi}$ are maximizers in \eqref{EIndLpp} and \eqref{ELppRes}, respectively. The fourth equality holds because $\underline{\pi}_j \in \bbZ^2 \smallsetminus \nu$ for $j \in [\ell(\underline{\pi})] \smallsetminus [\underline{i}]$. To justify the last equality in \eqref{E-16}, consider the case $k > b$ first. Then the restrictions in \eqref{EResPath} imply that $\underline{\pi}_j = \nu_{k+j-1}$ for $j \in [\underline{i}]$, and $\underline{z} = k+\underline{i}-1$. Consequently and by virtue of \eqref{ETelId},
\begin{align}
\label{E-17}
\sum_{j \in [\underline{i}] \smallsetminus \{1\}} \underline{w}(\underline{\pi}_{j}) = \sum_{r \in [\underline{z}] \smallsetminus [k]} \underline{w}(\nu_{r}) = G_{m_0, n_0}(\nu_{\underline{z}})-G_{m_0, n_0}(\nu_{k}). 
\end{align} 
A symmetric argument also gives \eqref{E-17} in the case $k < b$.  
Consider now the situation $k = b$.  If $\underline{i} = 1$ then $\underline{z} = k$ and \eqref{E-17} holds trivially.  If $\underline{i} > 1$ then the two possibilities $\underline{z} > b$ and $\underline{z} < b$ are handled similarly to the cases $k > b$ and $k < b$,  respectively. 
Finally, the last two inequalities in \eqref{E-16} come from definition \eqref{Elpp2}. 

Combining \eqref{E-15} and \eqref{E-16} yields (a). Moreover, all inequalities in \eqref{E-15} and \eqref{E-16} are in fact equalities. Therefore, $z$ is a maximizer in \eqref{EIndLpp}, and the vertices $\{\underline{\pi}_j: j \in [\ell(\underline{\pi})] \smallsetminus [\underline{i}-1]\}$ extend to a $w$-geodesic $\pi' \in \Pi_{m_0, n_0}^{m, n}$ that exits $\nu$ at $Z_{\pi', \nu} = \underline{z}$. Hence, (b) and (c). 
\end{proof}

\subsection{An identity for the exit points}

Define the exit points associated to \eqref{EIndLpp} by 
\begin{align}
\begin{split}
\underline{Z}_{\nu, i_0, j_0}^\square(m, n) &=\max\left\{[k-b]^\square: k \text{ is a maximizer in \eqref{EIndLpp}}\right\} \text{ for each } \square \in \{+, -\}. 
\end{split}
\label{EExitPts2}
\end{align}
The following lemma connects the notions in \eqref{EExitPts1} and \eqref{EExitPts2}. A similar statement can be found in \cite[Example 2]{pime-21}. 

\begin{lem}
\label{LIndLppExit}
The exit points in \eqref{EExitPts1} and \eqref{EExitPts2} satisfy the relation
\begin{align*}
\underline{Z}_{\nu, i_0, j_0}^\square(m, n) = Z_{\nu, i_0, j_0}^{\square}(m, n) \quad \text{ for each } (m, n) \in V_{\nu} \text{ and } \square \in \{+, -\}. 
\end{align*}
\end{lem}
\begin{proof}
The claim is a consequence of Lemma \ref{LLppPath2Pt}(b)-(c), and definitions \eqref{EExitPts1} and \eqref{EExitPts2}. 
\end{proof}

\section{Exit point bounds for path-to-point exponential LPP}
\label{SExBnd2}

We now set out to reformulate Theorems \ref{TExitUB} and \ref{TExitLB} in terms of increment-stationary down-right-path-to-point exponential LPP. One benefit of this undertaking is to be able to connect the present work with the exit point bounds in recent articles \cite{ferr-ghos-nejj-19, pime-18}. 

\subsection{Increment-stationary path-to-point exponential LPP}

To introduce the path-to-point model, let $u, v \in \bbZ_{\ge 0}$ and $\nu$ be a down-right path such that $\nu_{\ell(\nu)} \in \{u\} \times \bbZ_{\ge 0}$ and $\nu_1 \in \bbZ_{\ge 0} \times \{v\}$. Fix a parameter $z \in (0, 1)$ and a base vertex $\nu_{b} = (i_0, j_0)$ for some $b \in [\ell(\nu)]$. Then using the i.i.d.\ $\Exp(1)$-distributed $\eta$-variables, 
define the weights $\overline{\w}^z = \{\overline{\w}^z(i, j): (i, j) \in \bbZ_{\ge 0}^2\}$ 
as follows: 
\begin{align}
\overline{\w}^z(i, j) &= \eta(i, j) \quad \text{ if } (i, j) \not \in \nu. \label{Ebw}
\end{align}
Otherwise, $(i, j) = \nu_k$ for some unique $k \in [\ell(\nu)]$. Then, recalling \eqref{ERDsteps}, set 
\begin{align}
\label{Epw}
\begin{split}
\overline{\w}^z(i, j) = \eta(i, j) \cdot \bigg(&\frac{1}{z}[\one\{k > b, k-1 \in R_\nu\} - \one\{k < b, k \in R_\nu\}] \\
+ &\frac{1}{1-z}[\one\{k < b, k+1 \in D_\nu\}-\one\{k > b, k \in D_\nu\}]\bigg). 
\end{split}
\end{align}
By definition, the weights $\overline{\w}^z$ are independent, marginally $\Exp(1)$-distributed on $\bbZ^2_{\ge 0} \smallsetminus \nu$, and $\overline{\w}^z(\nu_b) = 0$. The rule \eqref{Epw} for the remaining marginals along $\nu$ can be informally described as follows. As $\nu$ is traversed from $\nu_b$ to $\nu_{\ell(\nu)}$, each vertex encountered after $\nu_b$ receives an $\Exp(z)$ weight if preceded by a horizontal step and a $-\Exp(1-z)$ weight otherwise. The same also holds as $\nu$ is traversed from $\nu_b$ to $\nu_1$ except that the signs are now flipped. See Figure \ref{FLppPath2Pt}. 

\begin{figure}
\centering
\begin{overpic}[scale=0.6]{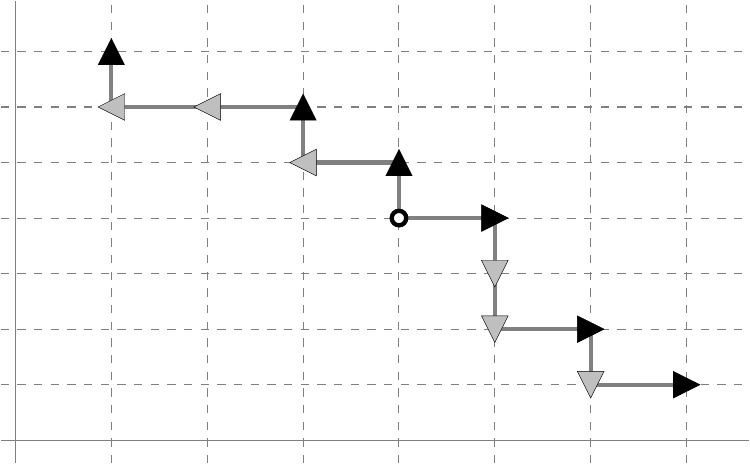}
\end{overpic}
\caption{Illustrates the assignment of the weights according to \eqref{Epw} to the vertices along a down-right path $\nu$ (gray). The choice of the base vertex $\nu_b$ is indicated (white dot) where the weight is zero. The remaining vertices along $\nu$ (marked with triangles) receive independent weights with the marginals chosen as follows: Horizontal and vertical triangles correspond to $\Exp[z]$ and $\Exp[1-z]$-distributed weights in absolute value, respectively. Black and gray triangles correspond to positive and negative weights, respectively.}
\label{FLppPath2Pt}
\end{figure}

\begin{example}\label{Ex1}
Consider the case $(i_0, j_0) = (0, 0)$. Necessarily, $\nu = L = L_{0, 0}^{u, v}$, the L-shaped path from $(0, v)$ to $(u, 0)$. Then comparing definitions \eqref{Ebw}-\eqref{Epw} with \eqref{ECplBlkBd} shows that 
\begin{align}
\label{E-13}
\overline{\w}^z(i, j) = \wb^z(i, j) \quad \text{ for } (i, j) \in V_{\nu}. 
\end{align}
\end{example}

Let us next relate the $\overline{\w}^z$-weights to the induced weights $\underline{\w}^z = \{\underline{\w}^z(i, j): (i, j) \in \bbZ^2_{\ge 0}\}$ obtained from the $\wb^z$-weights in the manner described around \eqref{EIndw}. More precisely, 
\begin{align}
\underline{\w}^z(i, j) =  \one_{\{k > b\}}\{\Gb^z(\nu_k)-\Gb^z(\nu_{k-1})\} + \one_{\{k < b\}}\{\Gb^z(\nu_k)-\Gb^z(\nu_{k+1})\} \label{EIndLppSt}
\end{align}
if $(i, j) = \nu_k$ for some $k \in [\ell(\nu)]$, and $\underline{\w}^z = \wb^{z}(i, j)$ otherwise. The following lemma records that $\overline{w}^z$ and $\underline{w}^z$ are identical in distribution on $V_\nu$. 
\begin{lem}
\label{Lwdist}
$\{\overline{\w}^z(i, j): (i, j) \in V_{\nu}\} \stackrel{\text{\rm{dist.}}}{=}  \{\underline{\w}^z(i, j): (i, j) \in V_{\nu}\}$. 
\end{lem}
\begin{proof}
From definition \eqref{EIndLppSt} and by virtue of the Burke property \eqref{EBurke}, the weights $\{\underline{\w}^z(\nu_k): k \in [\ell(\nu)]\}$ are independent with $\underline{\w}^z(\nu_b) = 0$ and the remaining marginals given by 
\begin{align}
\underline{\w}^z(\nu_k) \sim 
\begin{cases}
\Exp(z) \quad &\text{ if } k > b \text{ and } k-1 \in R_{\nu}\\ 
-\Exp(1-z) \quad &\text{ if } k > b \text{ and } k \in D_{\nu} \\
-\Exp(z) \quad &\text{ if } k < b \text{ and } k \in R_{\nu} \\
\Exp(1-z) \quad &\text{ if } k < b \text{ and } k+1 \in D_{\nu}. 
\end{cases}
\label{E-18}
\end{align}
Furthermore, the weights 
\begin{align}
\label{E-19}
\{\underline{\w}^{z}(i, j): (i, j) \in V_{\nu} \smallsetminus \nu\} = \{\wb^{z}(i, j): (i, j) \in V_{\nu} \smallsetminus \nu\}
\end{align} 
are independent, and marginally $\Exp(1)$-distributed by \eqref{Ewbd} because the set $V_{\nu} \smallsetminus \nu$ does not intersect the axes. Finally, since the weights in \eqref{E-19} do not enter definition \eqref{EIndLppSt}, the collection $\{\underline{\w}^z(i, j): (i, j) \in V_{\nu}\}$ is also independent. This completes the proof in view of the discussion following \eqref{Epw} where the joint distribution of $\{\overline{\w}^{z}(i, j): (i, j) \in V_{\nu} \smallsetminus \nu\}$ is described. 
\end{proof}

Using $\overline{\w}^z$ as the weights, let $\overline{\G}^z_{p, q}(m, n)$ denote the last-passage time from $(p, q) \in \bbZ_{\ge 0}^2$ to $(m, n) \in \bbZ_{\ge 0}^2$ computed via \eqref{Elpp2}, and write $\overline{\G}^{z, \circ}_{\nu_k}(m, n)$ for the restricted last-passage times computed according to \eqref{ELppRes} for each $k \in [\ell(\nu)]$.  (Recall that $\overline{\G}^{z, \circ}_{\nu_k}(m, n)$ need not agree with $\overline{\G}^z_{\nu_k}(m, n)$ due to the additional restrictions on the admissible paths in \eqref{ELppRes}).  Now define the path-to-point last-passage time from $\nu$ to $(m, n)$ (with the base vertex $(i_0, j_0)$) by 
\begin{align}
\begin{split}
\overline{\G}^z_{\nu, i_0, j_0}(m, n) = \max_{k \in [\ell(\nu)]} \bigg\{&\one_{\{k > b\}} \cdot \sum_{r = b}^{k-1} \overline{\w}^z(\nu_r) +  \one_{\{k < b\}} \cdot \sum_{r = k+1}^b  \overline{\w}^z(\nu_r)+\overline{\G}^{z, \circ}_{\nu_k}(m, n)\bigg\}. \label{ELppDR}
\end{split}
\end{align}
The right-hand side is the same as that of the first line in \eqref{EIndLpp} except that the underlying weights are now $\overline{\w}^z$.  

As explained in the next example and later in Examples \ref{ExLppAntd2Pt}-\ref{ExLppLine2Pt} below, \eqref{ELppDR} simultaneously generalizes 
the $\Gb^z$-process from Section \ref{SstatLPP} and some line-to-point LPP models considered in recent literature. 
\begin{example}
\label{Ex2}
When $(i_0, j_0) = (0, 0)$ ($\nu = L = L_{0, 0}^{u, v}$), it follows from \eqref{E-13} and definitions \eqref{Elppbd} and \eqref{ELppDR} that $\overline{\G}^z_{L, 0, 0}(m, n) = \Gb^z(m, n)$ for $(m, n) \in V_L = ([u] \cup \{0\}) \times ([v] \cup \{0\})$. 
\end{example}

In the general situation, \eqref{ELppDR} enjoys a simple relation to the $\Gb^z$-process on $V_{\nu}$ through the following distributional identity. 

\begin{prop} 
\label{PLppDist}
For the path-to-point LPP in \eqref{ELppDR}, 
\begin{align*}
\{\overline{\G}^z_{\nu, i_0, j_0}(m, n): (m, n) \in V_{\nu}\} \stackrel{\text{\rm{dist.}}}{=} \{\Gb^z(m, n)-\Gb^z(i_0, j_0): (m, n) \in  V_{\nu}\}. 
\end{align*}
\end{prop}
\begin{proof}
Consider the path-to-point LPP given by 
\begin{align}
\begin{split}
\underline{\G}_{\nu, i_0, j_0}^z(m, n) &= \max_{k \in [\ell(\nu)]} \bigg\{\one_{\{k > b\}} \sum_{r = b}^{k-1} \underline{\w}^z(\nu_r)+ \one_{\{k < b\}} \sum_{r = k+1}^{b} \underline{\w}^z(\nu_r) + \underline{\G}_{\nu_k}^{z, \circ}(m, n)\bigg\}
\end{split}
\label{E-20}
\end{align}
for $(m, n) \in \bbZ_{\ge 0}^2$. The two maps that compute the LPP values in \eqref{E-20} from the $\underline{\w}^z$-weights and the LPP values in \eqref{ELppDR} from the $\overline{\w}^z$-weights are identical. Hence, by Lemma \ref{Lwdist}, 
\begin{align}
\{\overline{\G}^z_{\nu, i_0, j_0}(m, n): (m, n) \in V_{\nu}\} \stackrel{\text{\rm{dist.}}}{=} \{\underline{\G}^z_{\nu, i_0, j_0}(m, n): (m, n) \in  V_{\nu}\}. \label{E-21}
\end{align}
Appealing to Lemma \ref{LLppPath2Pt}(a), the terms on the right-hand side of \eqref{E-21} can be written as 
\begin{equation*}
\underline{\G}^z_{\nu, i_0, j_0}(m, n) = \Gb^z(m, n)-\Gb^z(i_0, j_0) \quad \text{ for } (m, n) \in V_\nu. \qedhere 
\end{equation*}
\end{proof}

As a consequence of Proposition \ref{PLppDist}, the $\overline{\G}^z_{\nu, i_0, j_0}$-process inherits the Burke property \eqref{EBurke}.

\subsection{Distributional identities for the exit points}
 
Introduce the exit points associated with the path-to-point LPP \eqref{ELppDR} by 
\begin{align}
\begin{split}
\overline{Z}_{\nu, i_0, j_0}^{z, \square}(m, n) &= \max \{[k-b]^\square: k \text{ is a maximizer in \eqref{ELppDR}}\} \text{ for each } \square \in \{+, -\}.
\end{split}\label{EEx2}
\end{align}
These recover \eqref{EExhv} in the case $i_0 = j_0 = 0$; see Examples \ref{Ex1} and \ref{Ex2}. Moreover, one has the following key distributional connection to the exit points in \eqref{EEx}. 

\begin{prop}
\label{PExitDistId}
The exit points in \eqref{EEx2} satisfy the distributional identity: 
\begin{align*}
\{\overline{\Eh}^{z, \square}_{\nu, i_0, j_0}(m, n): (m, n) \in V_{\nu}, \square \in \{+, -\}\} \stackrel{\text{\rm{dist.}}}{=} \{\Eh^{z, \square}_{\nu, i_0, j_0}(m, n): (m, n) \in V_{\nu}, \square \in \{+, -\}\}. 
\end{align*}
\end{prop}
\begin{proof}
Consider the exit points given by 
\begin{align}
\begin{split}
\underline{\Eh}_{\nu, i_0, j_0}^{z, \square}(m, n) &= \max \{[k-b]^\square: k \text{ is a maximizer in \eqref{E-20}}\} \text{ for each } \square \in \{+, -\}.
\end{split}\label{EEx3}
\end{align}
Definitions \eqref{EEx2} and \eqref{EEx3} are the same except that they input $\overline{\w}^z$ and $\underline{\w}^z$, respectively, as the weights. Hence, by virtue of Lemma \ref{Lwdist}, 
\begin{align*}
\{\overline{\Eh}^{z, \square}_{\nu, i_0, j_0}(m, n): (m, n) \in V_{\nu}, \square \in \{+, -\}\} \stackrel{\text{\rm{dist.}}}{=} \{\underline{\Eh}^{z, \square}_{\nu, i_0, j_0}(m, n): (m, n) \in V_{\nu}, \square \in \{+, -\}\}. 
\end{align*}
On the other hand,  
\begin{align*}
\underline{\Eh}^{z, \square}_{\nu, i_0, j_0}(m, n) =  \Eh^{z, \square}_{\nu, i_0, j_0}(m, n)  \quad \text{ for each } (m, n) \in V_{\nu} \text{ and } \square \in \{+, -\}, 
\end{align*}
by Lemma \ref{LIndLppExit} and definitions \eqref{EEx} and \eqref{EExitPts1}. 
\end{proof}

\subsection{Right-tail bounds for the exit points restated}

The next pair of propositions bounds the right tail of the exit points in \eqref{EEx2}. These statements are in fact equivalent to Theorems \ref{TExitUB} and \ref{TExitLB}, respectively, by virtue of Proposition \ref{PExitDistId} and Lemma \ref{LExMon}.  It is worth emphasis that the main text does not rely on the present section.  In particular,  the following propositions should be viewed as applications of our main results. 
\begin{prop}
\label{PExitUB2}
Fix $\delta > 0$. There exist finite positive constants $c_0 = c_0(\delta)$, $\epsilon_0 = \epsilon_0(\delta)$ and $N_0 = N_0(\delta)$ such that
\begin{align*}
\P\{\overline{\Eh}^{z, \square}_{\nu, i_0, j_0}(m+i_0, n+j_0) > s(m+n)^{2/3}\} \le \exp\{-c_0 \min \{s^3, m+n\}\}
\end{align*}
whenever $\square \in \{+, -\}$, $(m, n) \in S_\delta \cap \bbZ_{\ge N_0}^2$, $s \ge (m+n)^{-2/3}$, $z \in (0, 1)$ with $|z-\Min(m, n)| \le \epsilon_0 s(m+n)^{-1/3}$, $(i_0, j_0) \in \bbZ_{\ge 0}^2$, and $\nu$ is a down-right path on $\bbZ_{\ge 0}^2$ with $(i_0, j_0) \in \nu$ and $(m+i_0, n+j_0) \in V_\nu$.  
\end{prop}
\begin{proof}
The result follows from Proposition \ref{PExitDistId} and Theorem \ref{TExitUB}. 
\end{proof}

%

\begin{prop}
\label{PExitLB2}
Fix $\delta > 0$,  $\epsilon > 0$ and $K \ge 0$. There exist finite positive constants $c_0 = c_0(\delta, K)$, $C_0 = C_0(\delta, \epsilon, K)$ and $N_0 = N_0(\delta, \epsilon,  K)$ such that
\begin{align*}
\P\{\overline{\Eh}^{z, \square}_{\nu, i_0, j_0}(m+i_0, n+j_0) > s(m+n)^{2/3}\} \ge \exp\{-C_0s^3\}
\end{align*}
whenever $\square \in \{+, -\}$, $(m, n) \in S_\delta \cap \bbZ_{\ge N_0}^2$, $s \in [\epsilon, c_0(m+n)^{1/3}]$, $z \in (0, 1)$ with $|z-\Min(m, n)| \le K s(m+n)^{-1/3}$, $(i_0, j_0) \in \bbZ_{\ge 0}^2$, and $\nu$ is a down-right path on $\bbZ_{\ge 0}^2$ with $(i_0, j_0) \in \nu$ and $(m+i_0, n+j_0) \in V_\nu$.  
\end{prop}
\begin{proof}
Combine Proposition \ref{PExitDistId} and Theorem \ref{TExitLB}. 
\end{proof}

\subsection{Some line-to-point LPP as special cases}

Our aim in this part is to demonstrate that the increment-stationary line-to-point LPP processes introduced in \cite{ferr-ghos-nejj-19, pime-18} arise from \eqref{ELppDR} as special cases. To this end, it is convenient to first develop formula  \eqref{ELppDR} into an alternative form as follows. 

Introduce another collection $\w' = \{\w'(i, j) :(i, j) \in \bbZ_{\ge 0}^2\}$ of weights by 
\begin{align}
\label{Ew'}
\w'(i, j) = \one_{\{(i, j) \not \in \nu \}} \cdot \eta(i, j)
\end{align}
Let $\G'_{p, q}(m, n)$ denote the last-passage time from $(p, q) \in \bbZ_{\ge 0}^2$ and $(m, n) \in \bbZ_{\ge 0}^2$ computed with the $\w'$-weights as in \eqref{Elpp2}. Then define a new path-to-point LPP by 
\begin{align}
\label{ELppDR2}
\G^{\prime, z}_{\nu, i_0, j_0}(m, n) = \max_{k \in [\ell(\nu)]} \bigg\{\one_{\{k > b\}} \sum_{r = b}^{k} \overline{\w}^z(\nu_r)+ \one_{\{k < b\}} \sum_{r = k}^{b} \overline{\w}^z(\nu_r) + \G'_{\nu_k}(m, n)\bigg\}
\end{align}
for $(m, n) \in \bbZ_{\ge 0}^2$. This process a.s.\ coincides on $V_\nu$ with the LPP given by \eqref{ELppDR} as the next lemma shows. 

\begin{lem}
\label{LLppDRAlt}
$\overline{\G}^z_{\nu, i_0, j_0}(m, n) \stackrel{\text{a.s.}}{=} \G^{\prime, z}_{\nu, i_0, j_0}(m, n)$ for each $(m, n) \in V_\nu$. 
\end{lem}
\begin{proof}

Let $(m, n) \in V_\nu$, $k \in [\ell(\nu)]$ and $\pi \in \underline{\Pi}_{\nu, k}^{m, n}$. Let $i \in [\ell(\pi)]$ denote the unique index such that $\pi_i = \nu_Z$ where $Z = Z_{\pi, \nu}$. Note from definition \eqref{EResPath} that $Z \ge k$ if $k > b$, and $Z \le k$ if $k < b$. (When $k = b$, both $Z \ge b$ and $Z \le b$ are possible).  
Using these implications with the fact that $\overline{\w}^z(\nu_b) = 0$, the agreement of the $\overline{\w}^z$ and $\w'$ weights on $V_\nu \smallsetminus \nu$, and definition \eqref{ELppDR2}, one obtains that 
\begin{align*}
&\one_{\{k > b\}} \cdot \sum_{r = b}^{k-1} \overline{\w}^z(\nu_r) +  \one_{\{k < b\}} \cdot \sum_{r = k+1}^{b} \overline{\w}^z(\nu_r) + \sum_{s \in [\ell(\pi)]} \overline{\w}^z(\pi_s) \\
&= (\one_{\{k > b\}} + \one_{\{k = b, Z \ge b\}}) \cdot \sum_{r = b}^{Z} \overline{\w}^z(\nu_r) +  (\one_{\{k < b\}} + \one_{\{k = b, Z \le b\}}) \cdot \sum_{r = Z}^{b} \overline{\w}^z(\nu_r) + \sum_{s \in [\ell(\pi)] \smallsetminus [i]} \overline{\w}^z(\pi_s) \\
&= \one_{\{Z > b\}} \cdot \sum_{r = b}^{Z} \overline{\w}^z(\nu_r) +  \one_{\{Z < b\}} \cdot \sum_{r = Z}^{b} \overline{\w}^z(\nu_r) + \sum_{s \in [\ell(\pi)] \smallsetminus [i]} \w'(\pi_s) \\
&\le \G^{\prime, z}_{\nu, i_0, j_0}(m, n). 
\end{align*}
In view of definition \eqref{ELppDR}, maximizing the first line above yields $\overline{\G}^z_{\nu, i_0, j_0}(m, n) \le \G^{\prime, z}_{\nu, i_0, j_0}(m, n)$. 

To obtain the converse inequality, assume that $k$ is maximizer in \eqref{ELppDR2} and pick a $\w'$-geodesic $\pi' \in \Pi_{\nu_k}^{m, n}$. \emph{Claim:} $\pi' \in \underline{\Pi}_{\nu, k}(m, n)$ a.s. To verify this, consider the case $k > b$ and $k \in D_\nu$. Then $\overline{\w}^z(\nu_k) \sim -\Exp[1-z]$ by definition \eqref{Epw}. Restrict to the a.s.\ event that $\overline{\w}^z(\nu_k) < 0$. For a contradiction, suppose that $\pi'_2 = \nu_k + (0, 1)$. Then, using the preceding inequality and also that $\pi'_2 = \nu_{k-1}$ and $\w'(\nu_k) = 0$ yields  
\begin{align*}
&\sum_{r = b}^{k} \overline{\w}^z(\nu_r)+\G'_{\nu_k}(m, n) = \sum_{r = b}^{k} \overline{\w}^z(\nu_r)+\sum_{s=1}^{\ell(\pi')} \w'(\pi'_s) = \sum_{r = b}^{k} \overline{\w}^z(\nu_r)+\sum_{s=2}^{\ell(\pi')} \w'(\pi'_s)\\
&< \sum_{r = b}^{k-1} \overline{\w}^z(\nu_r)+ \sum_{s=2}^{\ell(\pi')} \w'(\pi'_s) \le \sum_{r = b}^{k-1} \overline{\w}^z(\nu_r)+\G'_{\nu_{k-1}}(m, n), 
\end{align*}
which contradicts the choices of $k$ and $\pi'$.  (The strict inequality above comes from dropping the negative weight $\overline{\w}^z(\nu_k)$). The verification of the claim is similar in the case $k < b$ and $k \in R_{\nu}$, and is trivial in the remaining case. 

Write $i' \in [\ell(\pi')]$ for the unique index such that $\pi'_{i'} = \nu_{Z'}$ where $Z' = Z_{\pi', \nu}$. It follows from the claim above that, a.s., $Z' \ge k$ if $k > b$ and $Z' \le k$ if $k < b$. Furthermore, definition \eqref{Epw} implies that $\overline{\w}^z(\nu_r) \sim \Exp(z)$ when $k < r \le Z'$, and $\overline{\w}^z(\nu_r) \sim \Exp(1-z)$ when $Z' < r < k$ (including the case $k = b$). In particular, these weights are all a.s.\ positive. 
This together with the fact that the $\w'$-weights vanish on $\nu$ justifies the first inequality below. The subsequent step uses that the $\overline{\w}^z$ and $\w'$ weights agree on $V_\nu \smallsetminus \nu$, and $\overline{\w}^z(\nu_b) = 0$. The final inequality comes from definition \eqref{ELppDR}.   
\begin{align*}
\G^{\prime, z}_{\nu, i_0, j_0}(m, n) &= \one_{\{k > b\}} \sum_{r = b}^{k} \overline{\w}^z(\nu_r)+ \one_{\{k < b\}} \sum_{r = k}^{b} \overline{\w}^z(\nu_r) + \sum_{s \in [\ell(\pi')]} \w'(\pi'_s) \\
&\stackrel{\text{a.s.}}{\le} \one_{\{Z' \ge b\}} \cdot \sum_{r = b}^{Z'} \overline{\w}^z(\nu_r) +  \one_{\{Z' \le b\}} \sum_{r = Z'}^{b} \overline{\w}^z(\nu_r) + \sum_{s \in [\ell(\pi')] \smallsetminus [i']} \w'(\pi'_s) \\
&= \one_{\{Z' > b\}} \cdot \sum_{r = b}^{Z'-1} \overline{\w}^z(\nu_r) +  \one_{\{Z' < b\}}  \cdot \sum_{r = Z'+1}^{b} \overline{\w}^z(\nu_r) + \sum_{s \in [\ell(\pi')] \smallsetminus [i'-1]} \overline{\w}^z(\pi'_s) \\
&\le \overline{\G}^{z}_{\nu, i_0, j_0}(m, n). \qedhere
\end{align*}
\end{proof}

Let us now discuss two special cases of \eqref{ELppDR}.  

\begin{example}\label{ExLppAntd2Pt}
Let $z \in (0, 1)$ and $n \in \bbZ_{>0}$. Consider the down-right path $\nu$ of length $\ell(\nu) = 2n+1$ given by 
$\nu_k = \lf (k-1)/2 \rf, n-\lf k/2 \rf)$ for $k \in [2n+1]$. 
In other words, $\nu$ consists of the vertices on the anti-diagonals $\{(i, n-i): i \in [n] \cup \{0\}\}$ and $\{(i-1, n-i): i \in [n]\}$. Choose the base vertex on $\nu$ as $\nu_b = (i_0, j_0)$ where $j_0 = n-i_0$ and $b = 2i_0+1$ for some $i_0 \in [n]$. 

It follows from Lemma \ref{LLppDRAlt} and the structure of $\nu$ that 
\begin{align}
\label{E-28}
\begin{split}
&\overline{\G}^z_{\nu, i_0, j_0}(n, n) \stackrel{\text{a.s.}}{=} \max \limits_{\substack{k \in [2n+1] \\ k \text{ is odd}}} \bigg\{\one_{\{k > b\}} \sum_{r = b}^{k} \overline{\w}^z(\nu_r)+ \one_{\{k < b\}} \sum_{r = k}^{b} \overline{\w}^z(\nu_r) + \G'_{\nu_k}(n, n)\bigg\} \\
&= \max \limits_{\substack{i \in [n] \cup \{0\}}} \bigg\{\one_{\{i > i_0\}} \cdot \sum_{r = 2i_0+1}^{2i+1} \overline{\w}^z(\nu_r) +  \one_{\{i < i_0\}} \cdot \sum_{r = 2i+1}^{2i_0+1}\overline{\w}^z(\nu_r) +\G'_{\nu_{2i+1}}(n, n)\bigg\}. 
\end{split}
\end{align}
To justify dropping the terms with even $k \in [2n+1]$ from the maximum above, consider the case $k > b$ for example. Then since $\overline{\w}^z(\nu_{k+1}) \ge 0 \ge \overline{\w}^z(\nu_{k})$ and $\w'(\nu_{k}) = 0$, 
\begin{align*}
\sum_{r = b}^{k} \overline{\w}^z(\nu_r) + \G'_{\nu_k}(n, n) &= \sum_{r = b}^{k} \overline{\w}^z(\nu_r) + \max \{\G'_{\nu_{k-1}}(n, n), \G'_{\nu_{k+1}}(n, n)\} \\
&\le \max \bigg\{\sum_{r = b}^{k-1} \overline{\w}^z(\nu_r)  + \G'_{\nu_{k-1}}(n, n), \sum_{r = b}^{k+1} \overline{\w}^z(\nu_r)  +\G'_{\nu_{k+1}}(n, n)\bigg\}.  
\end{align*}
A similar reasoning also holds for the case $k < b$.  Hence, the first step in \eqref{E-28} is justified. 

The first two terms within the last maximum in \eqref{E-28} can be written as 
\begin{align}
\mrS_i^z = \one_{\{i > i_0\}} \sum_{s = i_0+1}^{i} \{\overline{\w}^z(\nu_{2s}) + \overline{\w}^z(\nu_{2s+1})\} + \one_{\{i < i_0\}} \sum_{s = i+1}^{i_0}\{\overline{\w}^z(\nu_{2s}) + \overline{\w}^z(\nu_{2s-1})\}\label{ES}
\end{align}
for each $i \in [n]$. In view of definition \eqref{Epw}, $\mrS_i^z$ is a sum of $|i-i_0|$ i.i.d.\ terms with marginal distributions $\Exp[z]-\Exp[1-z]$ when $i > i_0$ and $\Exp[1-z]-\Exp[z]$ when $i < i_0$. Returning to the last line in \eqref{E-28}, one has 
\begin{align}
\overline{\G}^z_{\nu, i_0, j_0}(n, n) &\stackrel{\text{a.s.}}{=} \max \limits_{\substack{i \in [n]}} \bigg\{\mrS_i^z + \G'_{(i, n-i)}(n, n)\bigg\}.  \label{E-30}
\end{align}
Up to the irrelevant shift by $(i_0, j_0)$, the right-hand side of \eqref{E-30} coincides with the type of line-to-point LPP with stationary initial data considered in \cite[Section 1.1]{pime-18}. 
\end{example}

\begin{example}\label{ExLppLine2Pt}
Let us now generalize Example \ref{ExLppAntd2Pt} to down-right paths with an arbitrary negative \emph{slope}. Let $z \in (0, 1)$, $\mu \in \bbR_{<0}$, $n \in \bbZ_{>0}$ and $j_0 \in [n] \cup \{0\}$. Put $i_0 = - \lf \mu (n-j_0) \rf \in \bbZ_{\ge 0}$ and $m = i_0 +\lf -\mu j_0 \rf \in \bbZ_{\ge 0}$. Consider a down-right path $\nu$ that contains the vertices $\{(i_0 + \lf \mu (j-j_0)\rf, j): j \in [n] \cup \{0\}\} \subset \bbZ_{\ge 0}^2$. For each $j \in [n] \cup \{0\}$, let $p_j \in [\ell(\nu)]$ denote the unique index such that 
$\nu_{p_j} = (i_0 + \lf \mu (n-j-j_0) \rf, n-j)$. 
Assume further that $p_0 = 1$, $D_\nu = \{p_j+1: j \in [n-1] \cup \{0\}\}$ and $p_n = \ell(\nu)$. Then $\nu$ is determined uniquely as the down-right path that starts from $\nu_1 = (0, n)$, ends at $\nu_{\ell(\nu)} = (m, 0)$ and, for each $k \in [\ell(\nu)-1]$, leaves vertex $\nu_k$ via a down-step if and only if $k = p_j$ for some $j \in [n-1] \cup \{0\}$. Choose the base vertex on $\nu$ as $\nu_b = (i_0, j_0)$ where $b = p_{n-j_0}$. Note that the case $\mu = -1$ of the preceding setup corresponds exactly to Example \ref{ExLppAntd2Pt}. 

Restrict to the case $\mu \in [-1, 0)$ below. In the next computation, the first equality is by virtue of Lemma \ref{LLppDRAlt} and the choice of $\nu$. The justification for dropping the terms $k \in R_{\nu}$ from the maximum is similar\footnote{The assumption $\mu \ge -1$ comes in crucially here by ensuring that $\nu$ does not take consecutive right steps.  Consequently,  each vertex $\nu_k$ with $k \in R_\nu$ lies between two vertices of the form $\nu_{p_j}$ and $\nu_{p_{j+1}}$ for some $j \in [n-1] \cup \{0\}$.} to the one given after \eqref{E-28} and omitted. The second equality holds since $[\ell(\nu)] \smallsetminus R_{\nu} = \{p_j: j \in [n] \cup \{0\}\}$. 
\begin{align}
\label{E-29}
\begin{split}
&\overline{\G}^z_{\nu, i_0, j_0}(m, n) \stackrel{\text{a.s.}}{=} \max \limits_{\substack{k \in [\ell(\nu)] \\ k \not \in R_{\nu}}} \bigg\{\one_{\{k > b\}} \sum_{r = b}^{k} \overline{\w}^z(\nu_r)+ \one_{\{k < b\}} \sum_{r = k}^{b} \overline{\w}^z(\nu_r) + \G'_{\nu_k}(m, n)\bigg\} \\
&= \max \limits_{\substack{j \in [n] \cup \{0\}}} \bigg\{\one_{\{j > n-j_0\}} \cdot \sum_{r = b}^{p_j} \overline{\w}^z(\nu_{r}) +  \one_{\{j < n-j_0\}} \cdot \sum_{r = p_j}^{b}\overline{\w}^z(\nu_{r}) +\G'_{\nu_{p_j}}(m, n)\bigg\}
\end{split}
\end{align}
Generalizing \eqref{ES}, introduce the sum 
\begin{align}
\label{ES2}
\mrS_j^{z, \mu} = \one_{\{j < j_0\}} \cdot \sum_{r=b}^{p_{n-j}} \overline{\w}^z(\nu_r) +  \one_{\{j > j_0\}} \cdot \sum_{r = p_{n-j}}^{b}\overline{\w}^z(\nu_r)
\end{align}
for each $j \in [n] \cup \{0\}$. Using \eqref{ES2} in \eqref{E-29} yields 
\begin{align}
\label{E-31}
\overline{\G}^z_{\nu, i_0, j_0}(m, n) \stackrel{\text{a.s.}}{=} \max \limits_{\substack{j \in [n] \cup \{0\}}} \bigg\{\mrS_{j}^{z, \mu}+\G'_{(i_0 + \lf \mu (j-j_0) \rf, j)}(m, n)\bigg\}. 
\end{align}
The right-hand side of \eqref{E-31} can be recognized as the LPP defined in \cite[(3.40)]{ferr-ghos-nejj-19} up to the shift of the origin to $(i_0, j_0)$. 
\end{example}

\section{Auxiliary estimates}\label{SEstim}

\begin{lem}
\label{LShpMinBnd} Let $\delta > 0$. The following statements hold.
\begin{enumerate}[\normalfont (a)]
\item 
$x+y \le \Shp(x, y) \le 2(x+y)$ for $x, y > 0$.
\item  
$
\Min(x, y) \in (\epsilon, 1-\epsilon) 
$
for $(x, y) \in S_\delta$ where $\epsilon = \sqrt{\delta}/2$. 
\item $(x+y)^{1/3} \le \curv(x, y) \le 2 \delta^{-1/3}(x+y)^{1/3}$ for $(x, y) \in S_\delta$. 
\end{enumerate}
\end{lem}
\begin{proof}
The claims are elementary consequences of definitions \eqref{EShp}, \eqref{EMin} and \eqref{Ecurv}. 
\end{proof}

Next is an estimate of the mean function in \eqref{EM}. 
\begin{lem}
\label{LMShpId}
Let $x, y \in \bbR_{>0}$ and $z \in (0, 1)$. Abbreviate $\Shp = \Shp(x, y)$, $\Min = \Min(x, y)$ and $\curv = \curv(x, y)$. 
Fix $\delta > 0$ and $\epsilon > 0$. Then there exists a constant $C_0 = C_0(\delta, \epsilon) > 0$ such that  
\begin{align*}
|\M^z(x, y)-\Shp-\curv^3(z-\Min)^2| \le C_0 (x+y) |z-\Min|^3 \quad \text{ for } (x, y) \in S_\delta \text{ and } z \in (\epsilon, 1-\epsilon).
\end{align*}
\end{lem}
\begin{proof}
From definitions \eqref{EM}, \eqref{EShp} and \eqref{EMin}, and the identity $\dfrac{x}{\Min^2} = \dfrac{y}{(1-\Min)^2} = \Shp$, 
\begin{align*}
\M^{z}(x, y)-\Shp &= (z-\Min)\bigg(-\frac{x}{z\Min}+ \frac{y}{(1-z)(1-\Min)}\bigg) \\
&= (z-\Min)\bigg(-\frac{x}{z\Min}+ \frac{x}{\Min^2} + \frac{y}{(1-z)(1-\Min)}-\frac{y}{(1-\Min)^2}\bigg) \\
&= (z-\Min)^2\bigg(\frac{x}{z\Min^2}+\frac{y}{(1-z)(1-\Min)^2}\bigg) = \frac{(z-\Min)^2\Shp}{z(1-z)}. 
\end{align*}
Recalling  the definition of $\curv$ from \eqref{Ecurv}, one obtains that 
\begin{align*}
\M^{z}(x, y)-\Shp - \curv^3 (z-\Min)^2  = (z-\Min)^2\Shp \cdot \bigg(\frac{1}{z(1-z)}-\frac{1}{\Min(1-\Min)}\bigg) = (z-\Min)^3\Shp\cdot \frac{\Min+z-1}{z(1-z)\Min(1-\Min)}.  
\end{align*}
The result follows from bounding the last expression above using Lemma \ref{LShpMinBnd}(a)--(b) and the assumption $z \in (\epsilon, 1-\epsilon)$. 
\end{proof}

\bibliographystyle{habbrv}
\bibliography{refs}

\end{document}